\documentclass[11pt]{article}
\usepackage{amsmath}
\usepackage{amssymb}
\usepackage{amsthm}
\usepackage[utf8x]{inputenc}
\usepackage{tikz-network}
\usepackage[colorlinks=true, allcolors=blue]{hyperref}

\newtheorem{theorem}{Theorem}[section]
\newtheorem{proposition}[theorem]{Proposition}
\newtheorem{lemma}[theorem]{Lemma}
\newtheorem{claim}[theorem]{Claim}
\newtheorem{corollary}[theorem]{Corollary}
\newtheorem{observation}[theorem]{Observation}
\newtheorem{definition}[theorem]{Definition}
\newtheorem{remark}[theorem]{Remark}

\begin{document}

\title{
On the local structure of oriented graphs --\\ a case study in flag algebras
}
\author{
Shoni Gilboa
\thanks{Department of Mathematics and Computer Science, The Open University of Israel, Raanana 43107, Israel} 
\and
Roman Glebov
\thanks{Department of Computer Science, Ben Gurion University of the Negev, Beer Sheva 84105, Israel.} 
\and
Dan Hefetz
\thanks{Department of Computer Science, Ariel University, Ariel 40700, Israel. Research supported by ISF grant 822/18.} \and
Nati Linial
\thanks{School of Computer Science and Engineering, Hebrew University, Jerusalem 91904, Israel.} 
\and 
Avraham Morgenstern
\thanks{School of Computer Science and Engineering, Hebrew University, Jerusalem 91904, Israel.}}

\maketitle
\begin{abstract}
Let $G$ be an $n$-vertex oriented graph. Let $t(G)$ (respectively $i(G)$) be the probability that a random set of $3$ vertices of $G$ spans a transitive triangle (respectively an independent set). We prove that $t(G) + i(G) \geq \frac{1}{9}-o_n(1)$. 
Our proof uses the method of flag algebras that we supplement with several steps that make it more easily comprehensible. We also prove a stability result and an exact result. Namely, we describe an extremal construction, prove that it is essentially unique, and prove that if $H$ is sufficiently far from that construction, then $t(H) + i(H)$ is significantly larger than $\frac{1}{9}$.  

We go to greater technical detail than is usually done in papers that rely on flag algebras. Our hope is that as a result this text can serve others as a useful introduction to this powerful and beautiful method.
\end{abstract}

\begin{quote}
\end{quote}

\section{Introduction}\label{sec:intro}

More than sixty years ago, Goodman~\cite{goodman} proved a quantitative Ramsey-type result for triangles. He determined the minimum over all $n$-vertex (undirected) graphs of the number of triples of vertices which form a triangle or an independent set. It readily follows from his result that the density of triangles plus the density of independent triples in a graph is asymptotically at least $\frac{1}{4}$. It is natural to look for an analogous statement for directed graphs. Goodman's Theorem clearly applies for directed graphs if one considers both transitive triangles and cyclic triangles. Since, moreover, a transitive tournament admits no cyclic triangles and no independent triples, the only related quantity which may be of interest is the minimum over all $n$-vertex directed graphs of the number of triples of vertices which form a transitive triangle or an independent set.

We start with the asymptotic version of the problem. For simplicity we consider only oriented graphs, i.e., directed graphs having no loops and no multiple (parallel or anti-parallel) edges. For an oriented graph $G$, denote by $t(G)$ (respectively  $i(G)$) the probability that a randomly chosen set of $3$ vertices of $G$ induces a transitive triangle (respectively an independent set). For every positive integer $n$, let 
$$
\tau(n) = \min \{t(G) + i(G) : G \textrm{ is an oriented graph on } n \textrm{ vertices}\}.
$$

\begin{observation} \label{obs::limitoftau}
$\left(\tau(n)\right)_{n=3}^{\infty}$ is a non-decreasing sequence.
\end{observation}

\begin{proof}
Let $n \geq 3$ be an integer and let $G = (V,E)$ be an oriented graph on $n+1$ vertices for which $t(G) + i(G) = \tau(n+1)$. Then
\begin{equation*}
\tau(n+1) = t(G) + i(G) = \textstyle\frac{1}{n+1} \sum_{v \in V} (t(G \setminus v) + i(G \setminus v)) \geq \tau(n) \,. 
\qedhere\end{equation*}
\end{proof}
Since the sequence $(\tau(n))_{n=3}^{\infty}$ is non-decreasing and is bounded from above by 1, it has a limit which we denote by $\tau$. 
Our main result is the following Goodman-type inequality.

\begin{theorem}\label{main}
$\tau = \frac{1}{9}$. That is, every $n$-vertex oriented graph $G$ satisfies
$$
t(G) + i(G) \geq \textstyle\frac{1}{9} - o_n(1)
$$
and, moreover, this bound is tight.
\end{theorem}

The tightness of the bound stated in Theorem~\ref{main} follows from the following observation.

\begin{observation}\label{obs:blowup} 
For every positive integer $n$, let ${\cal B}_n=(V,E)$ be the {\em balanced blowup} of a {\em cyclic triangle}, where $V$ is the disjoint union of sets $V_0, V_1,V_2$ with $|V_i| = \lfloor (n + i )/3 \rfloor$ for every $0 \leq i \leq 2$, and $E$ is comprised of all directed edges from $V_i$ to $V_{i+1 \bmod 3}$ for every $0 \leq i \leq 2$. For every positive integer $n$, it holds that
\begin{equation*}\label{eq:blowup}
t({\cal B}_n)+i({\cal B}_n)<\textstyle\frac{1}{9}.
\end{equation*}
Consequently, $\tau\leq\frac{1}{9}$.
\end{observation}

\begin{proof}
Clearly, the oriented graph ${\cal B}_n$ contains no transitive triangles and so $t({\cal B}_n) = 0$. Moreover, ${\cal B}_n$ contains exactly $\binom{\lfloor n/3 \rfloor}{3} + \binom{\lfloor (n + 1)/3 \rfloor}{3} + \binom{\lfloor (n + 2)/3 \rfloor}{3}$ independent triples and so 
\begin{equation*}
t({\cal B}_n)+i({\cal B}_n)=i({\cal B}_n) = \frac{\binom{\lfloor n/3 \rfloor}{3} + \binom{\lfloor (n + 1)/3 \rfloor}{3} + \binom{\lfloor (n + 2)/3 \rfloor}{3}}{\binom{n}{3}} < \frac{1}{9}.
\qedhere\end{equation*}
\end{proof}

In order for our methods to succeed, we must find all the asymptotically optimal oriented graphs. There are earlier examples in the flag-algebra literature where certain slight variations of the construction must be considered optimal as well, a phenomenon that is called {\em phantom edge} in~\cite{pikhurko-vaughan} and~\cite{flagmatic}.

Concretely, it is possible to delete a few edges from ${\cal B}_n$ without creating any new independent triples. Clearly no transitive triangles are created. 
Let ${\cal B}_n^{\varepsilon}$ be the random oriented graph that results upon randomly deleting each edge of ${\cal B}_n$ independently with probability $\varepsilon>0$. This oriented graph clearly contains no transitive triangles, and with high probability only $O(\varepsilon^2 n^3)$ new independent triples emerge (this follows, e.g., from Azuma's inequality). Hence, with $\varepsilon \to 0$, this oriented graph is optimal up to the second order term. It will be crucial to consider this altered construction as well to derive some necessary information.

We also prove a stability version of Theorem~\ref{main}. As usual, we say that two $n$-vertex oriented graphs $G$ and $H$ are $\varepsilon$-close if there are sets $E_1, E_2 \subseteq \binom{V(G)}{2}$ such that $|E_1| + |E_2| \leq \varepsilon n^2$ and $(G \setminus E_1) \cup E_2$ is isomorphic to $H$.  

\begin{theorem} \label{thm:stabil}
For every $\varepsilon > 0$, there exist $n_0$ and $\delta > 0$ such that if
$$
t(G) + i(G) \leq \textstyle\frac{1}{9} + \delta,
$$
for some $n$-vertex oriented graph $G$ with $n \geq n_0$,
then $G$ is $\varepsilon$-close to ${\cal B}_n$. 
\end{theorem}

Building on Theorem~\ref{thm:stabil} we can prove that ${\cal B}_n$ in essentially the unique extremal construction. 
This is in stark contrast to Goodman's inequality for which the family of extremal constructions is very rich. More precisely, we prove that a sufficiently large oriented graph is extremal for the number of transitive triangles plus the number of independent triples if and only if it belongs to the rigid family $\mathcal{E}_n$ which we will now describe. 
Let $\mathcal{M}_n$ denote the family of all triangle-free $n$-vertex oriented graphs which are the union of three matchings: one between $V_0$ and $V_1$, one between $V_1$ and $V_2$, and one between $V_2$ and $V_0$. Let $\mathcal{E}_n = \{\mathcal{B}_n \setminus H : H \in \mathcal{M}_n\}$. It is evident that $t(G) + i(G) = t(\mathcal{B}_n) + i(\mathcal{B}_n)$ for every $G \in \mathcal{E}_n$. It remains to prove that every large extremal oriented graph lies in $\mathcal{E}_n$. 
\begin{theorem}\label{thm:exact}
There exists an integer $n_0$ such that for every $n>n_0$, if $G$  minimizes $t(G) + i(G)$ among all $n$-vertex oriented graphs, then $G \in \mathcal{E}_n$.
\end{theorem}
The statement of Theorem~\ref{thm:exact} need not apply for small $n$. Consider the oriented graph with vertex set $V = \{0,1,2,3,4,5,6\}$, where for every $0 \leq i \leq 6$ vertex $i$ has a directed edge to $i+1 \bmod 7$ and to $i+3 \bmod 7$. This oriented graph has no independent triple nor a transitive triangle, whereas every $G \in \mathcal{E}_7$ has an independent triple. Similarly, the oriented graph with vertex set $\{0,1,2,3,4,5,6,7\}$ and all directed edges $(i,i+2 \bmod 8)$ and $(i, i+3 \bmod 8)$ for $0 \leq i \leq 7$, has no independent triples and no transitive triangle, whereas every $G \in \mathcal{E}_8$ has two independent triples.

Our proof of Theorem~\ref{main} follows mostly the flagmatic workflow~\cite{flagmatic}. 
We find it worthwhile to describe the entire process, even though most of it is not new. 
Our objective is to depend as little as possible on computer calculations, and rely on  theoretical arguments whenever possible. 
We also chose to write a self-contained paper since we could not find a comprehensive accessible documentation of flagmatic's entire computational process. Even the standard flag algebra arguments are not easy to understand from, e.g., \cite{razborov}, and are couched in logic and algebra terminology beyond what is required to prove Theorem~\ref{main} and other similar results in local combinatorics. 
Also, while \cite{fr-vaughan} provides much of the necessary information, we believe there is need for a more accessible source. Hence, to carry out the more technical parts of this research we had to pave our own path. 
The only available guide to this process that we managed to find was the flagmatic code which is only partially documented and is hard to penetrate. 
We hope that readers can use this paper as a simpler and fully self-contained case study, of proving inequalities in local combinatorics using flag algebra techniques.

We should mention the paper~\cite{pikhurko-vaughan} which provides some further details on the practice of the flag algebra method. Moreover, results in~\cite{pikhurko-vaughan} and, independently, \cite{DHMNS} yield a special case of our Theorem~\ref{main}; namely, that $i(G) \geq \frac{1}{9} - o_n(1)$ for every $n$-vertex oriented graph $G$ with $t(G) = 0$. Indeed, both papers prove that every $K_4$-free $n$-vertex {\em undirected} graph has at least $(\frac{1}{9}-o_n(1))\binom{n}{3}$ independent triples and this is tight. The relevant conclusion follows, since every orientation of $K_4$ contains a transitive triangle, so that the underlying graph of an oriented graph with no transitive triangles must be $K_4$-free. Both~\cite{pikhurko-vaughan} and~\cite{DHMNS}, use the flag algebra method.

\subsection{Flag algebras for the uninitiated}
This subsection deals with graphs as archetypical combinatorial objects, though everything we discuss here applies just as well to a whole range of mathematical objects. 
In fact, in this paper we apply this framework to {\em oriented} graphs. 
Let $H$ be a fixed $k$-vertex graph and let $G$ be a (typically large) graph. We denote by $p(H,G)$ the probability that a randomly chosen set of $k$ vertices in $G$ induces a subgraph that is isomorphic to $H$. 
Let $H_1, \ldots, H_m$ be finite graphs and let $\mathcal{H} = \{H_1, \ldots, H_m\}$. The $\cal H${\em -profile} of $G$ is the vector $\Phi_{\mathcal H}(G) = (p(H_1, G), \ldots, p(H_m, G))$. Understanding $\mathcal {H}$-profiles of large graphs is a key challenge of modern combinatorics. It is usually considered within the framework of {\em extremal graph theory}, or what one might call {\em local combinatorics}. Flag algebras offer a systematic approach to the study of such questions. As previously mentioned, this methodology applies to various combinatorial structures, and in the present paper we focus on oriented graphs. In order to apply the flag algebras method, one must first choose some collection $\cal F$ of $t$-{\em flagged} graphs, i.e., graphs in which some $t$ vertices are labeled $1, \ldots, t$. Associated with $\cal F$ and a graph $Z$ is the {\em flag probability matrix} $A_Z^{\cal F}$ whose rows and columns are indexed by $\cal F$. Let $H_1, \ldots, H_m$ be an arbitrary ordering of all $k$-vertex graphs and let $\mathcal{H} = \{H_1, \ldots, H_m\}$. Suppose that the vector $(p_{H_1}, \ldots, p_{H_m})$ is a limit point of $\cal H$-profiles $\Phi_{\cal H}(G)$ of graphs $G$ whose orders tend to infinity. The key feature of these matrices is that the matrix $\sum_{i=1}^m p_{H_i} \cdot A_{H_i}^{\cal F}$ is positive semi-definite (abbreviated henceforth PSD). By a well-known property of PSD matrices, its inner product with any PSD matrix $Q$ is non-negative. By choosing $Q$ (called below a {\em certificate}) cleverly, we can obtain interesting linear inequalities in the numbers $p_{H_1}, \ldots, p_{H_m}$. As we explain below, it is the proper choice of $Q$ that is the main technical challenge here and in many other papers that rely on the method of flag algebras.

To prove Theorem~\ref{main}, it suffices to find a \emph{$1/9$-certificate}. That is, for some choice of $k$ and ${\cal F}$, we wish to find a PSD matrix $Q$ that satisfies the linear inequality $\langle Q,  A_{H_i}^{\cal F} \rangle \leq t(H_i) + i(H_i) - 1/9$ for every $1 \leq i \leq m$. 
In the linear space of symmetric $|\cal F| \times |\cal F|$ matrices we find an affine subspace that contains all the $1/9$-certificates. 
In order to find a $1/9$-certificate, we run a semidefinite programming (SDP) solver on a computer. Such solvers output a solution of the SDP up to an additive error. This error term can be made arbitrarily small, but decreasing it increases the running time of the solver program. We then carefully `round` the matrix found by the SDP solver and obtain the desired $1/9$-certificate.
 
Rounding must be carried out with special care for those indices $i$ for which  $\langle Q,  A_{H_i}^{\cal F} \rangle \approx t(H_i) + i(H_i) - 1/9$. For other $i$'s, the inequality is strict and we may hope that it will remain true after some perturbation. Similarly, positive eigenvalues of the approximate matrix will hopefully remain positive after perturbation, but near-zero eigenvalues must be treated more carefully for the result to remain PSD.

\medskip
The rest of this paper is organized as follows. In Section~\ref{sec:flags} we introduce some of the foundations of the flag algebras method.
In Section~\ref{sec:framework} we present a family of semidefinite programs. An appropriate solution of such an SDP would imply Theorem~\ref{main}. We also define the notion of a \emph{certificate}. Section~\ref{sec:digression} is a warm-up for the actual proof, where we illustrate the methodology through two different proofs of the asymptotic version of Goodman's Theorem. In addition we provide several proofs of weaker versions of Theorem~\ref{main}. In Section~\ref{sec:back} we start working on our proof of Theorem~\ref{main}. Using a computer, we verify that $\tau$ is very close to $1/9$. The next four sections are dedicated to finding a $\frac{1}{9}$-certificate matrix. It turns out that every $\frac{1}{9}$-certificate matrix has a nontrivial kernel, and that, in fact, the intersection of all such kernels (over all $\frac{1}{9}$-certificates) is a nonempty linear space. In Section~\ref{sec::kernel} we determine this space. In Section~\ref{sec::sharp} we use certain $4$-vertex oriented graphs which are abundant in ${\mathcal B}_n^{\varepsilon}$ to impose additional restrictions on the entries of $Q$. In Section~\ref{sec::projection} we use the common kernel space of Section~\ref{sec::kernel} to project the problem to a space of lower dimension.
In Section~\ref{sec:rounding} we complete the proof of Theorem~\ref{main} by finding an approximate certificate for the projected problem with the aid of the computer, rounding it, and pulling it back to a certificate for the original problem. In Section~\ref{sec:stability} we prove Theorem~\ref{thm:stabil}. In Section~\ref{sec:exact} we use Theorem~\ref{thm:stabil} to prove Theorem~\ref{thm:exact}. Finally, in Section~\ref{sec:concluding} we consider possible directions for future research.

\section{Flags}\label{sec:flags}
A $t$-vertex {\em type} is an oriented graph whose vertices are labeled $1, \ldots, t$. A {\em flag} $F$ over the $t$-vertex type $\sigma$ with $\ell$ {\em petals} is an oriented graph on $t + \ell$ vertices with an isomorphic embedding $\varphi : \sigma \to F$. Two flags are isomorphic if they have the same type and there exists an isomorphism between them which preserves the type as well as the labeling of the type's vertices. Occasionally, we will view an oriented graph as a flag over the empty type. 

For flags $F_1, F_2$ over $\sigma$, define  $p(F_1, F_2)$ as follows. 
Choose uniformly at random a set $L$ of $|F_1| - |\sigma|$ vertices in $F_2 \setminus S_2$, where $S_2$ is the image of $\sigma$'s vertex set in $F_2$. 
Consider the flag $F$ induced by $F_2$ on $S_2\cup L$ and accompany it with $F_2$'s embedding of $\sigma$. 
Now, $p(F_1, F_2)$ is defined to be the probability that $F$ is isomorphic to $F_1$. Observe that $p(F_1, F_2)=0$ when $|F_1| > |F_2|$. 
For convenience, we also define $p(F_1, F_2)$ to be zero if $F_1, F_2$ are flags over different types. 
A {\em $\sigma$-rooting} of an oriented graph $G$ (also called a rooting of $G$ over $\sigma$) is a flag over $\sigma$ whose underlying oriented graph (i.e., just the oriented graph, without the embedding of $\sigma$) is $G$. 
For an oriented graph $G$ and a flag $F$ over $\sigma$, define $p(F,G)$ to be the mean of $p(F,\tilde{F})$ where $\tilde{F}$ is chosen uniformly at random from the set of $\sigma$-rootings of $G$. If there is no embedding of $\sigma$ into $G$, we define $p(F,G)$ to be zero.

Let $F_1,F_2$ be flags over $\sigma$ and let $G=(V,E)$ be an oriented graph such that there is an embedding of  $\sigma$ into $G$. 
We define $p(F_1,F_2;G)$ as follows. Choose uniformly at random a rooting of $G$ over $\sigma$, and denote by $S$ the image of $\sigma$'s vertex set in $G$. 
Now, choose uniformly at random two disjoint sets of vertices $L_1, L_2 \subseteq V \setminus S$ such that $|L_i| = |F_i| - |\sigma|$ for $i \in \{1,2\}$. Finally, define $p(F_1,F_2;G)$ to be the probability that the induced flag on $L_i\cup S$ is isomorphic to $F_i$ for $i \in \{1,2\}$.  We also define $\tilde{p}(F_1,F_2;G)$ in a similar way, where the sets $L_1$ and $L_2$ are chosen, uniformly and independently, at random (we still require $L_1, L_2 \subseteq V \setminus S$ but allow $L_1 \cap L_2 \neq \emptyset$). For convenience, in all cases where this process is ill-defined (namely, if $F_1$ and $F_2$ have different types, or if $\sigma$ does not embed into $G$, or if there is no such pair of disjoint sets $L_1, L_2$), we define $p(F_1, F_2; G)$ to be zero.

As an example, consider $\sigma$, $F_1, F_2$, and $G$ in Figure~\ref{fig:example}. There are $6$ $\sigma$-rootings of $G$ (one per edge), one of which is shown in Figure~\ref{fig:rooting_example}, along with the three flags of order $3$ over $\sigma$ which appear with positive probability in that rooting. A straightforward calculation shows that $p(F_1,G) = 1/9$ and $p(F_2,G) = 1/6$. Similarly, $p(F_1, F_2; G) = \tilde{p}(F_1, F_2; G) = p(F_1, F_1; G) = p(F_2, F_2; G) = 0$, $\tilde{p}(F_1, F_1; G) = 1/27$ and $\tilde{p}(F_2, F_2; G) = 1/18$. 
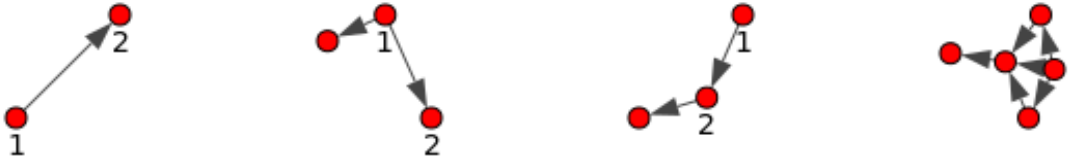
\begin{figure}
\centering
\begin{tikzpicture}
\Vertex[x=-5,  y=-0.7, label=$1$, size=0.4]{A1};
\Vertex[x=-5, y=0.7, label= $2$, size=0.4]{B1};
\Edge[Direct](A1)(B1); 

\Vertex[x=-1, y=-0.7, label=$1$, size=0.4]{A2};
\Vertex[x=-1, y=0.7, label= $2$, size=0.4]{B2};
\Vertex[x=-2.2, y=0, size=0.4]{C2};
\Edge[Direct](A2)(B2);
\Edge[Direct](A2)(C2);

\Vertex[x=2.2, y=-0.7, label=$1$, size=0.4]{A3};
\Vertex[x=2.2, y=0.7, label= $2$, size=0.4]{B3};
\Vertex[x=1,y=0, size=0.4]{C3};
\Edge[Direct](A3)(B3);
\Edge[Direct](B3)(C3);

\Vertex[x=6.2, y=1.4, size=0.4]{D1};
\Vertex[x=6.2, y=0, size=0.4]{D2};
\Vertex[x=5, y=-0.7, size=0.4]{D3};
\Vertex[x=5, y=0.7, size=0.4]{D4};
\Vertex[x=6.2, y=-1.4, size=0.4]{D5};
\Edge[Direct](D2)(D1);
\Edge[Direct](D4)(D2);
\Edge[Direct](D5)(D2);
\Edge[Direct](D3)(D2);
\Edge[Direct](D3)(D4);
\Edge[Direct](D3)(D5);

\end{tikzpicture}
\caption{An example of (left to right) a type $\sigma$, two flags $F_1, F_2$ over $\sigma$ and an oriented graph $G$.}
\label{fig:example}
\end{figure}

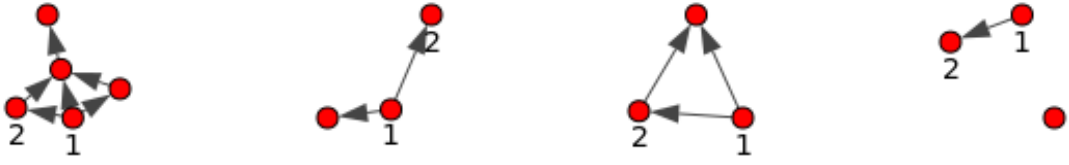
\begin{figure}
\centering
\begin{tikzpicture}
\Vertex[x=-5, y=1.4, size=0.4]{D1};
\Vertex[x=-5,y=0, size=0.4]{D2};
\Vertex[x=-6.2,y=-0.7, label=$1$, size=0.4]{D3};
\Vertex[x=-6.2,y=0.7, label=$2$, size=0.4]{D4};
\Vertex[x=-5,y=-1.4, size=0.4]{D5};
\Edge[Direct](D2)(D1);
\Edge[Direct](D4)(D2);
\Edge[Direct](D5)(D2);
\Edge[Direct](D3)(D2);
\Edge[Direct](D3)(D4);
\Edge[Direct](D3)(D5);

\Vertex[x=-2.2, y=-0.7, label=$1$, size=0.4]{C3};
\Vertex[x=-2.2, y=0.7, label=$2$, size=0.4]{C4};
\Vertex[x=-1, y=-1.4, size=0.4]{C5};
\Edge[Direct](C3)(C4);
\Edge[Direct](C3)(C5);

\Vertex[x=3,y=0, size=0.4]{B2};
\Vertex[x=1.8,y=-0.7, label=$1$, size=0.4]{B3};
\Vertex[x=1.8,y=0.7, label=$2$, size=0.4]{B4};
\Edge[Direct](B4)(B2);
\Edge[Direct](B3)(B2);
\Edge[Direct](B3)(B4);

\Vertex[x=7, y=1.4, size=0.4]{A1};
\Vertex[x=5.8, y=-0.7, label=$1$, size=0.4]{A3};
\Vertex[x=5.8, y=0.7, label=$2$, size=0.4]{A4};
\Edge[Direct](A3)(A4);

\end{tikzpicture}
\caption{An example of a rooting of $G$ over $\sigma$, and the three flags over $\sigma$ with one petal which appear with positive probability in that rooting.}
\label{fig:rooting_example}
\end{figure}

\subsection{The flag probability matrix}
Let $\Sigma$ be a finite set of types and let ${\mathcal F}$ be a finite set of flags over types in $\Sigma$. 
For an oriented graph $G$ we define the matrix $A_G$ (which depends on $\Sigma$ and ${\mathcal F}$ as well) as follows. It is an $|{\mathcal F}|\times|{\mathcal F}|$ matrix whose $(F_1, F_2)$ entry is $p(F_1, F_2; G)$. It readily follows from the definition of $p(F_1,F_2;G)$ that the entries of $A_G$ are rational numbers and that it is a block-diagonal matrix, with one block corresponding to each type $\sigma\in \Sigma$. Similarly, we define the $|{\mathcal F}|\times|{\mathcal F}|$ matrix $\tilde{A}_G$ whose $(F_1, F_2)$ entry is $\tilde{p}(F_1,F_2;G)$.

The following simple example, that we also use later, demonstrates how to compute the matrices $A_G$ and $\tilde{A}_G$. Consider the $1$-vertex type, and the three different flags with one petal over it (see Figure~\ref{fig:flags_vertex}). 
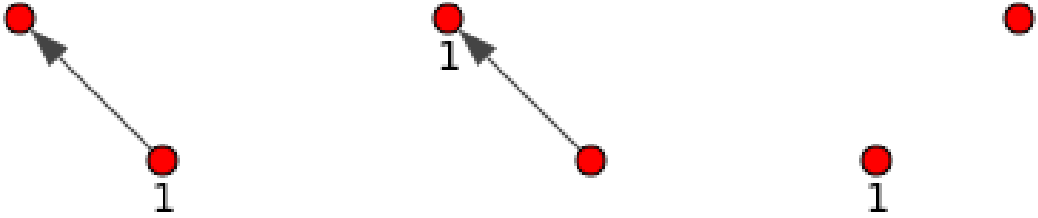
\begin{figure}
\centering
\begin{tikzpicture}
\Vertex[x=-4, y=-0.7, label=$1$, size=0.4]{A1};
\Vertex[x=-4, y=0.7, size=0.4]{B1};
\Edge[Direct](A1)(B1);
\Vertex[x=0, y=-0.7, label=$1$, size=0.4]{A2};
\Vertex[x=0, y=0.7, size=0.4]{B2};
\Edge[Direct](B2)(A2);
\Vertex[x=4, y=-0.71, label=$1$, size=0.4]{A3};
\Vertex[x=4, y=0.7, size=0.4]{B3};
\end{tikzpicture}
\caption{The three $2$-vertex flags over the $1$-vertex type.}
\label{fig:flags_vertex}
\end{figure}
With respect to this family of flags, the flag probability matrix $A_G$ is a symmetric $3 \times 3$ matrix which is defined by six numbers as follow. Sample uniformly at random a vertex from a large oriented graph (in our terminology, a random {\em $\sigma$-rooting}, where $\sigma$ is the $1$-vertex type), and calculate the expectations of the probabilities of two random distinct vertices having any particular ``relationship'' with the chosen vertex. For example, the contribution of any root vertex to the $(3,3)$ entry of $A_G$, is the probability that the two randomly chosen vertices are both non-neighbors of that root vertex. Clearly, these probabilities can be expressed as averages of quadratic terms in the degrees of a vertex, namely, its out-degree $d_{+}$, its in-degree $d_{-}$, and its non-degree $d_0$. Therefore
\begin{equation}\label{eq:3matrix}
A_G=\frac{1}{n(n-1)(n-2)}\begin{pmatrix}
2\sum_v \binom{d_{+}(v)}{2} & \sum_v d_{+}(v)d_{-}(v) & \sum_v d_{+}(v)d_{0}(v) \\
\sum_v d_{+}(v)d_{-}(v) & 2\sum_v \binom{d_{-}(v)}{2} & \sum_v d_{-}(v)d_{0}(v) \\ 
\sum_v d_{+}(v)d_{0}(v) & \sum_v d_{-}(v)d_{0}(v) & 2\sum_v \binom{d_{0}(v)}{2} 
\end{pmatrix}. 
\end{equation}
Similarly,
\begin{align*} 
\tilde{A}_G &= \frac{1}{n(n-1)^2}\begin{pmatrix}
\sum_v d_{+}(v)^2 & \sum_v d_{+}(v)d_{-}(v) & \sum_v d_{+}(v)d_{0}(v) \\
\sum_v d_{+}(v)d_{-}(v) & \sum_v d_{-}(v)^2 & \sum_v d_{-}(v)d_{0}(v) \\ 
\sum_v d_{+}(v)d_{0}(v) & \sum_v d_{-}(v)d_{0}(v) & \sum_v d_{0}(v)^2  
\end{pmatrix}\\ 
&=\frac{1}{n(n-1)^2}\sum_v
\begin{pmatrix} d_{+}(v) \\ d_{-}(v) \\ d_{0}(v) \end{pmatrix}
(d_{+}(v),d_{-}(v),d_{0}(v)).
\end{align*}
This matrix is clearly PSD, since it has the form $B^{\rm T} B$.

Our real interest is in $A_G$, whereas $\tilde{A}_G$ is merely a supporting actor, as the following lemma illustrates.
\begin{lemma} \label{lem::almostPSD}
Let $\Sigma$ be a set of types, let ${\mathcal F}$ be a set of flags over types from $\Sigma$, and let $G=(V,E)$ be an $n$-vertex oriented graph. Then
\begin{enumerate}
\item 
For a type $\sigma$ in $\Sigma$, let ${\mathcal R}_{\sigma}$ be the set of all rootings of $G$ over $\sigma$, let ${\mathcal F}_{\sigma}$ be the set of flags in ${\mathcal F}$ over the type $\sigma$ and let $B_{\sigma}$ be the $|{\mathcal F}_{\sigma}|\times|{\mathcal R}_{\sigma}|$ matrix such that for every rooting $r\in {\mathcal R}_{\sigma}$ and every flag $F$ in ${\mathcal F}_{\sigma}$, the $(F, r)$ entry of $B_{\sigma}$ is $p(F,r)$. 

For every $\sigma$ in $\Sigma$, the block of $\tilde{A}_G$ which corresponds to $\sigma$ equals $\frac{1}{|{\mathcal R}_{\sigma}|} B_{\sigma}  B_{\sigma}^{\rm T}$. Consequently, the matrix $\tilde{A}_G$ is PSD.

\item 
For $F_1, F_2 \in \mathcal{F}$, if $F_1, F_2$ are flags over the same $t$-vertex type with $t<n$, then
\begin{equation*}
|p(F_1, F_2; G) - \tilde{p}(F_1, F_2; G)| \leq \frac{(|F_1|-t)(|F_2|-t)}{n-t},
\end{equation*} 
and otherwise $|p(F_1, F_2; G) - \tilde{p}(F_1, F_2; G)| =0$.
Consequently,  
$$
\lVert A_G - \tilde{A}_G\rVert_{\infty} \leq \frac{C_{\mathcal F}}{n},
$$ 
where $||\cdot||_{\infty}$ is the max norm and $C_{\mathcal F}$ is a positive constant depending only on $\mathcal F$.
\end{enumerate}
\end{lemma}

\begin{proof}
\begin{enumerate}
\item  
For every two flags $F_1, F_2 \in \mathcal{F}$ over the same type $\sigma$, clearly
$$
\tilde{p}(F_1, F_2; G) = \frac{1}{|{\mathcal R}_{\sigma}|} \sum_{r \in {\mathcal R}_{\sigma}} p(F_1,r) p(F_2,r).
$$
Our claim readily follows.
\item 
By definition, $p(F_1, F_2; G) = \tilde{p}(F_1, F_2; G) =0$ if $F_1, F_2$ are flags over different types 
or\\$\max\{|F_1|,|F_2|\}>n$.
If $F_1, F_2$ are flags over the same $t$-vertex type such that $t\geq n$ and $|F_1|,|F_2|\leq n$, then necessarily $|F_1|=|F_2|=t$ and hence
$p(F_1, F_2; G) = \tilde{p}(F_1, F_2; G)$. 

Therefore, assume that $F_1, F_2$ are flags over the same $t$-vertex type $\sigma$ with $t<n$ and $|F_1|,|F_2|\leq n$.
Choose uniformly at random a rooting of $G$ over $\sigma$ and denote the image of $\sigma$'s vertex set in $G$ by $S$. Now, choose uniformly and independently at random two sets of vertices $L_1, L_2 \subseteq V \setminus S$ such that $|L_i| = |F_i| - t$ for $i \in \{1,2\}$. Let $\Omega$ denote the event that the induced flag on $L_i \cup S$ is isomorphic to $F_i$ for $i \in \{1,2\}$.
Note that
$$
p(F_1, F_2; G) = \Pr(\Omega \mid L_1 \cap L_2 = \emptyset) 
$$
and
\begin{align*}
\tilde{p}(F_1, F_2; G)&=\Pr(\Omega)\\
&=\Pr(L_1\cap L_2=\emptyset)\Pr(\Omega\mid L_1\cap L_2=\emptyset)+\Pr(L_1\cap L_2\neq\emptyset)\Pr(\Omega\mid L_1\cap L_2\neq\emptyset). 
\end{align*}
It follows that
\begin{align*}
p(&F_1, F_2; G) - \tilde{p}(F_1, F_2; G) = \left(1 - \Pr(L_1 \cap L_2 = \emptyset)\right) \Pr(\Omega \mid L_1 \cap L_2 = \emptyset)\\
&- \Pr(L_1 \cap L_2 \neq \emptyset) \Pr(\Omega \mid L_1 \cap L_2 \neq \emptyset)\\
=& \Pr(L_1 \cap L_2 \neq \emptyset) \Pr(\Omega \mid L_1 \cap L_2 = \emptyset) - \Pr(L_1 \cap L_2 \neq \emptyset) \Pr(\Omega \mid L_1 \cap L_2 \neq \emptyset)\\
=& \Pr(L_1 \cap L_2 \neq \emptyset) \left(\Pr(\Omega \mid L_1 \cap L_2 = \emptyset) - \Pr(\Omega \mid L_1 \cap L_2 \neq \emptyset)\right).
\end{align*}
Hence
\begin{align*}
|p(F_1, F_2; G)&-\tilde{p}(F_1, F_2; G)|\\
&=\Pr(L_1\cap L_2\neq\emptyset)\left\lvert\Pr(\Omega\mid L_1\cap L_2=\emptyset)-\Pr(\Omega\mid L_1\cap L_2\neq\emptyset)\right\rvert\\
&\leq  \Pr(L_1\cap L_2\neq\emptyset).
\end{align*}
For every $v \in V$ it holds that
\begin{align*}
\Pr(v \in L_1 \text{ and } v \in L_2)&=\Pr(v\notin S)\Pr(v \in L_1 \text{ and } v \in L_2\mid v\notin S)\\
&=\frac{n-t}{n} \cdot \frac{|F_1| - t}{n-t} \cdot \frac{|F_2| - t}{n-t}.
\end{align*}
A union bound then implies that
\begin{align*}
|p(F_1, F_2; G) - \tilde{p}(F_1, F_2; G)| &\leq \Pr(L_1 \cap L_2 \neq \emptyset)\\
&\leq n\cdot \frac{n-t}{n} \cdot \frac{|F_1| - t}{n-t} \cdot \frac{|F_2| - t}{n-t} = \frac{(|F_1|-t)(|F_2|-t)}{n-t}.
\qedhere\end{align*}
\end{enumerate}
\end{proof}

Note that Lemma~\ref{lem::almostPSD} appears implicitly in~\cite{razborov} and is proved in~\cite{HLNPS}.  

\section{Applying flags to prove graph inequalities}\label{sec:framework}
In this section we explain how to obtain lower bounds on the densities of fixed oriented graphs in large oriented graphs, using flags. To make the presentation simpler and more concrete, we concentrate on the problem at hand, i.e., bounding $\tau$.

\subsection{An SDP problem}\label{subsec:SDP}

Let $k \geq 3$ be an integer and let $G_1, \ldots, G_m$ be the complete list of all $k$-vertex oriented graphs, up to isomorphism.
First, we restate the quantity $t(G) + i(G)$ in terms of $k$-vertex subgraphs. 
For every $1 \leq i \leq m$, let 
$$
c_i = t(G_i) + i(G_i).
$$ 
\begin{observation}\label{obs:objective}
For every oriented graph $G$ it holds that 
$$
t(G) + i(G) = \sum_{i=1}^{m} c_i  p(G_i, G).
$$
\end{observation}

\begin{proof}
The quantities $t(G), i(G)$ are defined by sampling 3 vertices of $G$ uniformly at random. Instead, we can first sample $k$ vertices of $G$ uniformly at random and then sample 3 vertices uniformly at random out of these $k$. The two resulting expressions are equal by the law of total probability.
\end{proof}

Let $\Sigma$ be a set of types and let ${\mathcal F}$ be a set of flags over $\Sigma$. For an oriented graph $G$, let $A_G$ be the flag probability matrix of $G$ with respect to the set ${\mathcal F}$ of flags over the types in $\Sigma$.

\begin{observation}\label{obs:matrix}
Let $G$ be an oriented graph and let ${\mathcal F}$ be a family of flags. If $|F_1| + |F_2| - |\sigma| \leq k$ for all flags $F_1, F_2 \in \mathcal{F}$ over the same type $\sigma$, then
$$
A_G = \sum_{i=1}^{m} p(G_i, G) A_{G_i}.
$$
\end{observation}

\begin{proof}
For every two flags $F_1, F_2 \in \mathcal{F}$, it follows by the law of total probability that
$$
p(F_1, F_2; G) = \sum_{i=1}^{m} p(G_i, G) p(F_1, F_2; G_i).
$$
Our claim readily follows.
\end{proof}

\begin{theorem} \label{thm2gen}
Let $\mathcal{F}$ be a family of flags satisfying the assumption of Observation~\ref{obs:matrix}. Suppose that $\sum_{i=1}^{m} p_i  c_i \geq \alpha$ for every non-negative real numbers $p_1, \ldots, p_{m}$ that sum up to $1$ for which the matrix $\sum_{i=1}^{m} p_i A_{G_i}$ is PSD. Then $\tau \geq \alpha$.
\end{theorem}

\begin{proof}
Recall that for every positive integer $n$,
$$
\tau(n) = \min \{t(G) + i(G) : G \textrm{ is an oriented graph on } n \textrm{ vertices}\}
$$
Let $G^{(n)}$ be an oriented graph on $n$ vertices for which $t(G^{(n)}) + i(G^{(n)}) = \tau(n)$. By passing to a subsequence of $(G^{(n)})_{n=1}^{\infty}$ if needed, we may assume that for every $1 \leq i \leq m$, the sequence $(p(G_i,G^{(n)}))_{n=1}^{\infty}$ converges to a limit which we denote by $p_i$. Clearly, the real numbers $p_1, \ldots, p_{m}$ are non-negative and $\sum_{i=1}^{m} p_i = 1$. We will show that the matrix $\sum_{i=1}^{m} p_i A_{G_i}$ is PSD. 

Fix a vector $v \in \mathbb{R}^{|{\mathcal F}|}$ and a positive integer $n$. Then
\begin{align*}
v^{\rm T} \left(\sum_{i=1}^{m} p(G_i, G^{(n)}) A_{G_i}\right) v &= v^{\rm T} A_{G^{(n)}} v = v^{\rm T} \left(A_{G^{(n)}} - \tilde{A}_{G^{(n)}} \right) v + v^{\rm T} \tilde{A}_{G^{(n)}} v\\
&\geq v^{\rm T} \left(A_{G^{(n)}} - \tilde{A}_{G^{(n)}}\right) v \geq - ||v||_1^2 \cdot ||A_{G^{(n)}} - \tilde{A}_{G^{(n)}}||_{\infty} \\ &= - ||v||_1^2 \cdot O(1/n).
\end{align*}
The first equality holds by Observation~\ref{obs:matrix}. The first inequality holds since $\tilde{A}_{G^{(n)}}$ is PSD, see Lemma~\ref{lem::almostPSD}, part 1. For the last equality we use Lemma~\ref{lem::almostPSD} part 2. It thus follows that
$$
v^{\rm T} \left(\sum_{i=1}^{m} p_i A_{G_i}\right) v = \lim_{n \to \infty} v^{\rm T} \left(\sum_{i=1}^{m} p(G_i, G^{(n)}) A_{G_i}\right) v \geq 0,
$$
and thus the matrix $\sum_{i=1}^{m} p_i A_{G_i}$ is indeed PSD, as claimed.

We conclude that
\begin{equation*}
\tau = \lim_{n \to \infty} \tau(n) = \lim_{n \to \infty} t(G^{(n)}) + i(G^{(n)}) = \lim_{n \to \infty} \sum_{i=1}^{m} c_i p(G_i, G^{(n)}) = \sum_{i=1}^{m} c_i p_i \geq \alpha,
\end{equation*}
where the third equality holds by Observation~\ref{obs:objective}, and the inequality holds by the assumption of the theorem and the proven fact that $\sum_{i=1}^{m} p_i A_{G_i}$ is PSD.
\end{proof}
 
In other words, Theorem \ref{thm2gen} shows that $\tau$ is bounded from below by the optimum of the following semidefinite program.
\begin{eqnarray}\label{sdpgen} 
\begin{aligned}
&\text{Variables: } p_1, \ldots, p_{m} \\
&\text{Goal: minimize } \sum_{i=1}^{m} p_i c_i \\
&\text{Constraints: } \\
& p_1, \ldots, p_{m} \ge 0 \\
& \sum_{i=1}^{m} p_i = 1 \\
& \sum_{i=1}^{m} p_i A_{G_i} \succeq 0 \text{ (this inequality means that the matrix is PSD.)}
\end{aligned} 
\end{eqnarray}

This is a key idea of the flag algebra method. An asymptotic inequality about graph densities can be proved by solving an SDP problem that seems hardly related to graphs.

How should one choose
the set of types $\Sigma$ and the set of flags ${\mathcal F}$? For any fixed $k$, there are finitely many types and finitely many flags over them that induce non-zero blocks in the matrices $A_{G_i}$. We would like to use Theorem~\ref{thm2gen} and thus also Observation~\ref{obs:matrix}. Therefore, flags over a type $\sigma$ should be of size at most $\ell_{\sigma} := \lfloor (k + |\sigma|)/2 \rfloor$. Since we would like to gain as much information as possible, it makes sense to use all flags of size at most $\ell_{\sigma}$ over every type $\sigma$ of size at most $k$. However, it is in fact sufficient to use only flags of size precisely $\ell_{\sigma}$ over every type $\sigma$ such that $|\sigma| < k$ and  $|\sigma| \equiv k \bmod 2$, since they carry the same information. In hindsight, and after some trial and error, it transpires that one can actually give up some additional flags and still obtain the same results.  

Finally, we need to choose $k$. As $k$ grows, we gain more information, but the calculations become more complex. We therefore seek the smallest $k$ that yields the desired results. As expected, $k = 1,2$ yield nothing. With $k=3$ we already obtain a non-trivial lower bound, but not the desired inequality $\tau \geq 1/9$. Finally, $k=4$ delivers the goods. We still present the analysis for $k=3$ in Section~\ref{sec:digression}, since we think that it is insightful. 

\subsection {Certificate matrices}\label{subsec:certificate}
The inner product of two $N\times N$ real matrices $A = (a_{i,j})$ and $B = (b_{i,j})$ is defined as usual to be
$$
\langle A, B \rangle := {\rm Tr}(A B^{\rm T}) = \sum_{1 \leq i \leq N,\, 1 \leq j \leq N} a_{i,j} b_{i,j}.
$$
For a symmetric $B$ clearly, $\langle A, B \rangle = {\rm Tr}(A B)$. We recall a standard fact from linear algebra.
\begin{lemma}\label{lem:scalar_product}
A matrix is PSD if and only if its inner product with every PSD matrix is non-negative.
\end{lemma}

\begin{definition}
For $\alpha \geq 0$, an $\alpha${\em-certificate} for the  SDP~\eqref{sdpgen} is an $|{\mathcal F}| \times |{\mathcal F}|$ PSD matrix $Q$ such that for every $1 \leq i \leq m$ there holds
\begin{equation*} 
c_i \geq \langle Q, A_{G_i} \rangle +  \alpha.
\end{equation*}
\end{definition}

Applying SDP weak duality to~\eqref{sdpgen} yields the following useful proposition. For the sake of completeness, we include its (short and simple) proof.  
\begin{proposition}\label{prop:certificate}
If SDP~\eqref{sdpgen} has an $\alpha$-certificate, then its optimum is at least $\alpha$, whence $\tau \geq \alpha$ by Theorem~\ref{thm2gen}.
\end{proposition}

\begin{proof}
Let $Q$ be an $\alpha$-certificate for~\eqref{sdpgen}. 
Suppose that the matrix $\sum_{i=1}^{m} p_i A_{G_i}$ is PSD, 
where $p_1, \ldots, p_{m}\ge 0$ and $\sum_{i=1}^{m} p_i = 1$. Then
$$
\sum_{i=1}^{m} p_i c_i  =  \sum_{i=1}^{m} p_i (c_i - \langle Q, A_{G_i} \rangle) + \left\langle Q, \sum_{i=1}^{m} p_i A_{G_i} \right\rangle 
\geq \sum_{i=1}^{m} p_i (c_i - \langle Q, A_{G_i} \rangle) \geq \sum_{i=1}^m p_i \alpha = \alpha,
$$
where the first inequality follows from Lemma~\ref{lem:scalar_product} and the second inequality holds since $Q$ is an $\alpha$-certificate and $p_1,\ldots,p_m \geq 0$.
\end{proof}

\section{A slight digression} \label{sec:digression}
We start with two proofs of Goodman's bound for undirected graphs, which we then adjust
to derive the (suboptimal) bound $\tau\ge 1/10$.

\subsection{A toy example -- Goodman's bound for undirected graphs}
In this subsection we deal with undirected graphs, not with oriented ones. We do not detail the slight necessary terminological changes.

For $1 \leq i \leq 4$, let $U_i$ be the unique (up to isomorphism) undirected graph with $3$ vertices and $i-1$ edges (see Figure~\ref{fig:undirected_graphs}). We denote by $\Delta = p(U_4,G)$ (resp., $\bar\Delta = p(U_1,G)$) the density of triangles (resp., independent triples) in an undirected graph $G$.

\begin{figure}
\centering
\begin{tikzpicture}

\Vertex[x=-7,y=0, size=0.4]{A1};
\Vertex[x=-5,y=0, size=0.4]{A2};
\Vertex[x=-6,y=1.7,  size=0.4]{A3};
\Text[x=-6,y=0.6]{$U_1$}

\Vertex[x=-3,y=0, size=0.4]{B1};
\Vertex[x=-1,y=0, size=0.4]{B2};
\Vertex[x=-2,y=1.7,  size=0.4]{B3};
\Edge(B1)(B3);
\Text[x=-2,y=0.6]{$U_2$}

\Vertex[x=1,y=0, size=0.4]{C1};
\Vertex[x=3,y=0, size=0.4]{C2};
\Vertex[x=2,y=1.7,  size=0.4]{C3};
\Edge(C1)(C3);
\Edge(C2)(C3);
\Text[x=2,y=0.6]{$U_3$}

\Vertex[x=5,y=0, size=0.4]{D1};
\Vertex[x=7,y=0, size=0.4]{D2};
\Vertex[x=6,y=1.7,  size=0.4]{D3};
\Edge(D1)(D2);
\Edge(D1)(D3);
\Edge(D2)(D3);
\Text[x=6,y=0.6]{$U_4$}

\end{tikzpicture}
\caption{The four $3$-vertex undirected graphs.}
\label{fig:undirected_graphs}
\end{figure}
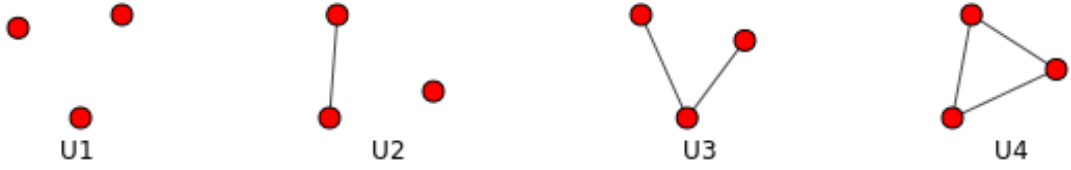

We recall two of the many proofs of Goodman's Theorem~\cite{goodman}.
\begin{theorem}[Goodman]\label{thm:Goodman}
For every $n$-vertex undirected graph $G=(V,E)$ we have
$$
\Delta + \bar\Delta \geq \textstyle\frac{1}{4} - o_n(1).
$$
\end{theorem}
\begin{proof}[First proof]
Let $m=|E|$ denote the number of edges of $G$. Observe that
\begin{equation} \label{eq:delta0}
\frac{1}{\binom{n}{3}} \sum_{v \in V} d(v) (n - 1 - d(v)) = 2 p(U_2,G) + 2 p(U_3,G) 
\end{equation}
and therefore
\begin{eqnarray} \label{eq:delta1}
\Delta + \bar\Delta &=& p(U_1, G) + p(U_4, G) = 1 - p(U_2,G) - p(U_3,G) \nonumber\\
&=& 1 - \frac{n-1}{2\binom{n}{3}} \sum_{v \in V} d(v) + \frac{1}{2\binom{n}{3}} \sum_{v \in V} d(v)^2 \nonumber \\ 
&=& 1 - \frac{6 m}{n(n-2)} + \frac{1}{2\binom{n}{3}}\sum_{v \in V} d(v)^2.
\end{eqnarray}
It follows by the Cauchy-Schwarz inequality that
\begin{equation} \label{eq::convexity}
\frac{1}{n} \sum_{v \in V} d(v)^2 \geq \left(\frac{1}{n} \sum_{v \in V} d(v)\right)^2 = \frac{4m^2}{n^2}.
\end{equation}
Combining~\eqref{eq:delta1} and~\eqref{eq::convexity} we obtain
\begin{align*}
\Delta+\bar\Delta &\geq 1 - \frac{6 m}{n(n-2)} + \frac{2m^2}{n\binom{n}{3}} = 1 - \frac{6m}{n^2} + \frac{12m^2}{n^4} - O\left(\frac{1}{n}\right)\\
&= \frac{1}{4} + \frac{3(n^2 - 4m)^2}{4n^4} - O\left(\frac{1}{n}\right) \geq \frac{1}{4} - O\left(\frac{1}{n}\right).
\qedhere
\end{align*}
\end{proof}

\begin{proof}[Second proof]

We apply the framework from Section~\ref{sec:framework} to undirected graphs, with $k=3$. 
As was elaborated in Subsection~\ref{subsec:SDP}, we consider the two different one-petal flags over the $1$-vertex type (see Figure~\ref{fig:undirected_flags}).
\begin{figure}
\centering
\begin{tikzpicture}
\Vertex[x=-4, y=-0.7, label=$1$, size=0.4]{A1};

\Vertex[x=0, y=-0.7, label=$1$, size=0.4]{A2};
\Vertex[x=0, y=0.7, size=0.4]{B2};
\Edge(B2)(A2);

\Vertex[x=4, y=-0.7, label=$1$, size=0.4]{A3};
\Vertex[x=4, y=0.7, size=0.4]{B3};
\end{tikzpicture}
\caption{The 1-vertex type, and the two 2-vertex flags over it.}
\label{fig:undirected_flags}
\end{figure}
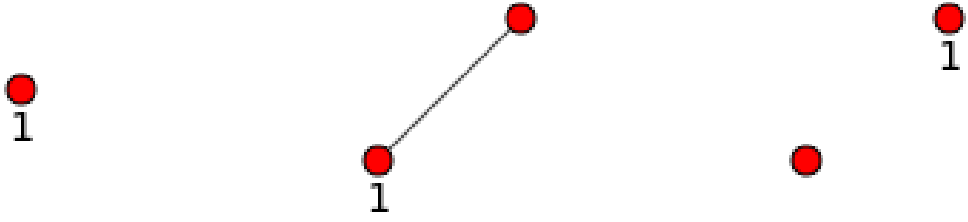
For every $1 \leq i \leq 4$, let $c_i = p(U_1, U_i) + p(U_4, U_i)$. Clearly,
$$
c_1 = 1, \quad c_2 = 0, \quad c_3 = 0, \quad c_4 = 1.
$$
For every $1 \leq i \leq 4$ let $A_i = A_{U_i}$. A straightforward calculation then shows that
$$
A_1 = \begin{pmatrix}
0 & 0 \\
0 & 1 
\end{pmatrix}, 
\quad A_2 = \begin{pmatrix}
0 & \frac{1}{3} \\
\frac{1}{3} & \frac{1}{3} 
\end{pmatrix}, 
\quad A_3=\begin{pmatrix}
\frac{1}{3} & \frac{1}{3} \\
\frac{1}{3} & 0 
\end{pmatrix}, 
\quad A_4 =  \begin{pmatrix}
1 & 0 \\
0 & 0 
\end{pmatrix}.
$$
An undirected analog of Proposition~\ref{prop:certificate} then implies that, for every $n$-vertex undirected graph $G=(V,E)$, the quantity $\Delta + \bar\Delta$ is bounded from below, up to $o_n(1)$, by the optimum of the following semidefinite program.
\begin{eqnarray}\label{sdptoy}
\begin{aligned}
&\text{Variables: } p_1, p_2, p_3, p_4 \\
&\text{Goal:  minimize }  p_1 + p_4  \\
&\text{Constraints: } \\
& p_1, p_2, p_3, p_4 \geq 0 \\
& p_1 + p_2 + p_3 + p_4 = 1 \\
&  \begin{pmatrix}
\frac{1}{3}p_3 + p_4 & \frac{1}{3}p_2 + \frac{1}{3}p_3 \\
\frac{1}{3}p_2 + \frac{1}{3}p_3 & p_1 + \frac{1}{3}p_2
\end{pmatrix} \succeq 0
\end{aligned} 
\end{eqnarray}
In order to complete this proof of Theorem~\ref{thm:Goodman} it suffices to show that the optimum of SDP~\eqref{sdptoy} is at least $\frac{1}{4}$; this can be done by finding a $\frac{1}{4}$-certificate for this SDP. In fact, the optimum of the SDP~\eqref{sdptoy} is exactly $\frac{1}{4}$; the upper bound can be proved, e.g., by taking
$$
(p_1, p_2, p_3, p_4) := \lim_{n \to \infty} (p(U_1, K_{n,n}), p(U_2, K_{n,n}), p(U_3, K_{n,n}), p(U_4, K_{n,n}))=\left(\textstyle\frac{1}{4}, 0, \textstyle\frac{3}{4}, 0\right).
$$
In order for a matrix
$$
\begin{pmatrix}
\alpha & \beta  \\
\beta  & \gamma  
\end{pmatrix}
$$
to be a $\frac{1}{4}$-certificate for SDP~\eqref{sdptoy}, it should be PSD and satisfy the following inequalities
$$
\gamma \leq \textstyle\frac{3}{4}, \quad
\textstyle \frac{2}{3} \beta + \textstyle\frac{1}{3} \gamma \leq - \textstyle \frac{1}{4}, \quad
\textstyle\frac{1}{3} \alpha + \textstyle\frac{2}{3} \beta \leq - \textstyle\frac{1}{4}, \quad
\alpha \leq \textstyle\frac{3}{4}.
$$
Choosing $\alpha = \gamma = \textstyle\frac{3}{4}$ and $\beta = - \textstyle\frac{3}{4}$ satisfies all of the above inequalities as equalities, and the resulting matrix
$$
\begin{pmatrix}
\frac{3}{4} & -\frac{3}{4}  \\
-\frac{3}{4} & \frac{3}{4}
\end{pmatrix}
$$
is indeed PSD.
\end{proof}

\begin{remark}
For every $v \in V(G)$, let $d_0(v) = n-1-d(v)$ denote the number of non-neighbours of $v$ in $G$. Note that (similarly to the derivation of \eqref{eq:3matrix}), it holds that
$$
A_G=\frac{1}{3\binom{n}{3}}\begin{pmatrix}
\sum_{v\in V} \binom{d(v)}{2} & \frac{1}{2}\sum_{v\in V}  d(v)d_0(v) \\
\frac{1}{2}\sum_{v\in V}  d_0(v)d(v) & \sum_{v\in V} \binom{d_0(v)}{2}  
\end{pmatrix}.
$$
It then follows, by stripping off the flag algebra terminology, that the second proof of Theorem~\ref{thm:Goodman} assumes the form of the following direct argument. For every $v \in V$, we have
\begin{equation} \label{eq::d0}
\binom{d_0(v)}{2} + \binom{d(v)}{2} = \frac{d_0(v)^2 + d(v)^2}{2} - \frac{n-1}{2} \geq d(v) d_0(v) - \frac{n-1}{2}.
\end{equation}
Therefore,
\begin{align*} 
(3p(U_1,G)&+p(U_2,G))+(p(U_3,G)+3p(U_4,G))=\frac{1}{\binom{n}{3}} \sum_{v \in V} \binom{d_0(v)}{2}+\frac{1}{\binom{n}{3}} \sum_{v \in V}\binom{d(v)}{2}\\
&\geq \frac{1}{\binom{n}{3}} \sum_{v \in V} d(v) d_0(v) - \frac{1}{\binom{n}{3}}n\frac{n-1}{2} = 2p(U_2,G)+2p(U_3,G)-\frac{3}{n-2},
\end{align*}
where the inequality holds by~\eqref{eq::d0} and the second equality holds by~\eqref{eq:delta0}. We conclude that
\begin{align*}
\Delta+\bar\Delta &= p(U_1,G) + p(U_4,G)\\
&\geq \frac{p(U_1,G) + p(U_2,G) + p(U_3,G) + p(U_4,G)}{4} - \frac{3}{4(n-2)} = \frac{1}{4} - O\left(\frac{1}{n}\right).
\end{align*}
\end{remark}
\subsection{Back to oriented graphs}
Recall that we want to prove that $\tau \geq \frac{1}{9}$. We now show how the two proofs of Goodman's Theorem we presented can be easily adjusted to yield a weaker, albeit nontrivial, bound.  

\begin{proposition}\label{tenth}
$$\tau\geq \textstyle\frac{1}{10},$$
i.e., every $n$-vertex oriented graph $G$ satisfies
\[
t(G)+i(G) \geq \textstyle\frac{1}{10} - o_n(1).
\]
\end{proposition}
 
\begin{proof}[First proof]
As in the first proof of Goodman's inequality, we denote the number of edges in $G$ by $m$. Also, let $c(G)$ denote the probability that a randomly chosen set of $3$ vertices of $G$ induces a cyclic triangle. It is easy to see that 
\begin{equation} \label{eq::fewCyclicTriangles}
c(G) \leq \frac{1}{3\binom{n}{3}} \sum_{v \in V} d_+(v) d_-(v) \leq \frac{1}{3\binom{n}{3}} \sum_{v \in V} \left(\frac{d_+(v) + d_-(v)}{2}\right)^2 = \frac{1}{12\binom{n}{3}}\sum_{v \in V} d(v)^2.
\end{equation}
Combining~\eqref{eq:delta1}, ~\eqref{eq::convexity} and~\eqref{eq::fewCyclicTriangles} we obtain, 
\begin{align*}
t(G) + i(G) &= \Delta+\bar\Delta-c(G)\geq 1 - \frac{6m}{n(n-2)} + \frac{5}{12\binom{n}{3}} \sum_{v \in V} d(v)^2 \\
&\geq  1 - \frac{6m}{n(n-2)} + \frac{5m^2}{3n\binom{n}{3}} = 1 - \frac{6m}{n^2} + \frac{10m^2}{n^4} - O\left(\frac{1}{n}\right) \\
&= \frac{1}{10}+\frac{(3n^2 - 10m)^2}{10n^4} -  O\left(\frac{1}{n}\right)\geq \frac{1}{10}  -  O\left(\frac{1}{n}\right).
\qedhere
\end{align*}
\end{proof}
\begin{proof}[Second proof]
As in the second proof of Goodman's Theorem, we follow the framework of Section \ref{sec:framework}, with $k=3$, but this time for oriented graphs. 

Let $D_1,D_2,D_3,D_4,D_5,D_6,D_7$ be the different oriented graphs on $3$ vertices, up to isomorphism (see Figure~\ref{fig:3graphs}).
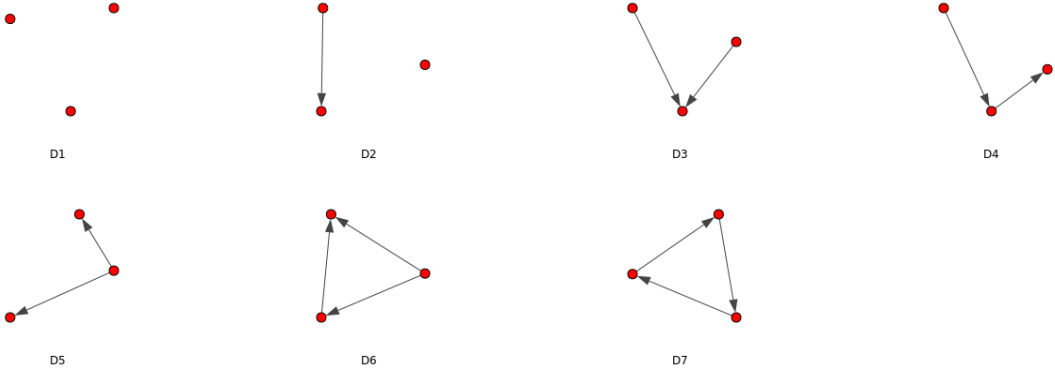
\begin{figure}
\centering
\begin{tikzpicture}

\Vertex[x=-7,y=0, size=0.4]{A1};
\Vertex[x=-5,y=0, size=0.4]{B1};
\Vertex[x=-6,y=1.7,  size=0.4]{C1};
\Text[x=-6,y=0.6]{$D_1$}

\Vertex[x=-3,y=0, size=0.4]{A2};
\Vertex[x=-1,y=0, size=0.4]{B2};
\Vertex[x=-2,y=1.7,  size=0.4]{C2};
\Edge[Direct](A2)(C2);
\Text[x=-2,y=0.6]{$D_2$}

\Vertex[x=1,y=0, size=0.4]{A3};
\Vertex[x=3,y=0, size=0.4]{B3};
\Vertex[x=2,y=1.7,  size=0.4]{C3};
\Edge[Direct](A3)(C3);
\Edge[Direct](B3)(C3);
\Text[x=2,y=0.6]{$D_3$}

\Vertex[x=5,y=0, size=0.4]{A4};
\Vertex[x=7,y=0, size=0.4]{B4};
\Vertex[x=6,y=1.7,  size=0.4]{C4};
\Edge[Direct](A4)(C4);
\Edge[Direct](B4)(A4);
\Text[x=6,y=0.6]{$D_4$}

\Vertex[x=-5,y=-3, size=0.4]{A5};
\Vertex[x=-3,y=-3, size=0.4]{B5};
\Vertex[x=-4,y=-1.3,  size=0.4]{C5};
\Edge[Direct](B5)(A5);
\Edge[Direct](B5)(C5);
\Text[x=-4,y=-2.4]{$D_5$}

\Vertex[x=-1,y=-3, size=0.4]{A6};
\Vertex[x=1,y=-3, size=0.4]{B6};
\Vertex[x=0,y=-1.3,  size=0.4]{C6};
\Edge[Direct](A6)(C6);
\Edge[Direct](B6)(C6);
\Edge[Direct](B6)(A6);
\Text[x=0,y=-2.4]{$D_6$}

\Vertex[x=3,y=-3, size=0.4]{A7};
\Vertex[x=5,y=-3, size=0.4]{B7};
\Vertex[x=4,y=-1.3,  size=0.4]{C7};
\Edge[Direct](A7)(C7);
\Edge[Direct](C7)(B7);
\Edge[Direct](B7)(A7);
\Text[x=4,y=-2.4]{$D_7$}

\end{tikzpicture}
\caption{The $7$ isomorphism types of oriented graphs of order $3$.}
\label{fig:3graphs}
\end{figure}
For every $1 \leq i \leq 7$, let $c_i = t(D_i) + i(D_i)$.  
Clearly,
$$
c_1=1,\quad c_2=0,\quad c_3=0,\quad c_4=0,\quad c_5=0,\quad c_6=1,\quad c_7=0.
$$
Each $3$-vertex type corresponds to a $1 \times 1$ block which is trivially PSD and bears no value for our purposes. Therefore, we only consider the $1$-vertex type, and the three different flags with one petal over it (see Figure~\ref{fig:flags_vertex}). Using~\eqref{eq:3matrix} we obtain
\begin{gather*}
A_1=\begin{pmatrix}
0 & 0 & 0 \\
0 & 0 & 0 \\ 
0 & 0 & 1
\end{pmatrix}, 
\quad 
A_2=\begin{pmatrix}
0 & 0 & \frac{1}{6} \\
0 & 0 & \frac{1}{6} \\ 
\frac{1}{6} & \frac{1}{6} & \frac{1}{3}
\end{pmatrix},
\quad 
A_3=\begin{pmatrix}
0 & 0 & \frac{1}{3} \\
0 & \frac{1}{3} & 0 \\ 
\frac{1}{3} & 0 & 0
\end{pmatrix},\\
A_4=\begin{pmatrix}
0 & \frac{1}{6} & \frac{1}{6} \\
\frac{1}{6} & 0 & \frac{1}{6} \\ 
\frac{1}{6} & \frac{1}{6} & 0
\end{pmatrix}, 
\quad 
A_5=\begin{pmatrix}
\frac{1}{3} & 0 & 0 \\
0 & 0 & \frac{1}{3} \\ 
0 & \frac{1}{3} & 0
\end{pmatrix},
\quad 
A_6=\begin{pmatrix}
\frac{1}{3} & \frac{1}{6} & 0 \\
\frac{1}{6} & \frac{1}{3} & 0 \\ 
0 & 0 & 0
\end{pmatrix},
\quad 
A_7=\begin{pmatrix}
0 & \frac{1}{2} & 0 \\
\frac{1}{2} & 0 & 0 \\ 
0 & 0 & 0
\end{pmatrix}.
\end{gather*}
By Theorem~\ref{thm2gen}, $\tau$ is bounded from below by the optimum of the following SDP:
\begin{eqnarray}\label{sdptoy2}
\begin{aligned}
&\text{Variables: } p_1,p_2,p_3,p_4,p_5,p_6,p_7 \\
&\text{Goal:  minimize }  p_1+p_6  \\
&\text{Constraints: } \\
& p_1,p_2,p_3,p_4,p_5,p_6,p_7\geq 0 \\
& p_1+p_2+p_3+p_4+p_5+p_6+p_7= 1 \\
& \begin{pmatrix}
\frac{1}{3}p_5+\frac{1}{3}p_6 & \frac{1}{6}p_4+\frac{1}{6}p_6+\frac{1}{2}p_7 & \frac{1}{6}p_2+\frac{1}{3}p_3+\frac{1}{6}p_4 \\
\frac{1}{6}p_4+\frac{1}{6}p_6+\frac{1}{2}p_7 & \frac{1}{3}p_3+\frac{1}{3}p_6 & \frac{1}{6}p_2+\frac{1}{6}p_4+\frac{1}{3}p_5 \\ 
\frac{1}{6}p_2+\frac{1}{3}p_3+\frac{1}{6}p_4 & \frac{1}{6}p_2+\frac{1}{6}p_4+\frac{1}{3}p_5 & p_1+\frac{1}{3}p_2
\end{pmatrix} \succeq 0
\end{aligned} 
\end{eqnarray}
If we take 
$$
(p_1,p_2,p_3,p_4,p_5,p_6,p_7):=\textstyle\frac{1}{100}(10,18,27,0,27,0, 18),
$$
then $p_1,p_2,p_3,p_4,p_5,p_6,p_7 \geq 0$, $p_1+p_2+p_3+p_4+p_5+p_6+p_7=1$, and
\begin{align*}\begin{pmatrix}
\frac{1}{3}p_5+\frac{1}{3}p_6 & \frac{1}{6}p_4+\frac{1}{6}p_6+\frac{1}{2}p_7 & \frac{1}{6}p_2+\frac{1}{3}p_3+\frac{1}{6}p_4 \\
\frac{1}{6}p_4+\frac{1}{6}p_6+\frac{1}{2}p_7 & \frac{1}{3}p_3+\frac{1}{3}p_6 & \frac{1}{6}p_2+\frac{1}{6}p_4+\frac{1}{3}p_5 \\ 
\frac{1}{6}p_2+\frac{1}{3}p_3+\frac{1}{6}p_4 & \frac{1}{6}p_2+\frac{1}{6}p_4+\frac{1}{3}p_5 & p_1+\frac{1}{3}p_2
\end{pmatrix}& =\frac{1}{100}\begin{pmatrix}
9 & 9 & 12 \\
9 & 9 & 12 \\ 
12 & 12 & 16
\end{pmatrix}\\
&=\begin{pmatrix}\frac{3}{10} \\ \frac{3}{10} \\ \frac{4}{10}\end{pmatrix} \left(\frac{3}{10},\frac{3}{10},\frac{4}{10}\right) \succeq 0.
\end{align*}
Therefore the optimum of SDP~\eqref{sdptoy2} is at most $p_1 + p_6 = \frac{1}{10}$, that is, this proof technique with $k=3$ cannot yield a lower bound larger than $\frac{1}{10}$. 
Next, we use Proposition~\ref{prop:certificate} to show that this bound is tight, by finding a $\frac{1}{10}$-certificate for SDP~\eqref{sdptoy2}.
The symmetries of the problem indicate that it might suffice (and, as the proof shows, it does suffice) to consider matrices of the form
$$Q=\begin{pmatrix}
\alpha & \beta & \gamma \\
\beta & \alpha & \gamma \\ 
\gamma & \gamma & \delta
\end{pmatrix}.$$
For this matrix to be a $\frac{1}{10}$-certificate for SDP~\eqref{sdptoy2}, it must satisfy all of the following inequalities
\begin{align}
\delta&\leq\textstyle\frac{9}{10},\label{ineq1}\\
\textstyle\frac{1}{3}\delta+\textstyle\frac{2}{3}\gamma&\leq-\textstyle\frac{1}{10},\label{ineq2}\\
\textstyle\frac{1}{3}\alpha+\textstyle\frac{2}{3}\gamma&\leq-\textstyle\frac{1}{10},\label{ineq35}\\
\textstyle\frac{1}{3}\beta+\textstyle\frac{2}{3}\gamma&\leq-\textstyle\frac{1}{10},\label{ineq4}\\
\textstyle\frac{2}{3}\alpha+\textstyle\frac{1}{3}\beta&\leq\textstyle\frac{9}{10},\label{ineq6}\\
\beta&\leq-\textstyle\frac{1}{10}.\label{ineq7}
\end{align}
In addition, $Q$ must be PSD. In particular, $|\beta| \leq \alpha$ must hold  and thus inequality~\eqref{ineq35} implies~\eqref{ineq4}. Choosing 
$$\alpha=\textstyle\frac{9}{10},\quad\beta=-\textstyle\frac{1}{10},\quad\gamma=-\textstyle\frac{6}{10},\quad\delta=\textstyle\frac{9}{10}$$
satisfies inequality~\eqref{ineq6}, whereas \eqref{ineq1}, \eqref{ineq2}, \eqref{ineq35}, \eqref{ineq7} hold as equalities. Moreover, the resulting matrix
\begin{equation}\label{eq:Qtoy2}Q=\frac{1}{10}
\begin{pmatrix}
9 & -1 & -6 \\
-1 & 9 & -6 \\ 
-6 & -6 & 9
\end{pmatrix}
\end{equation}
is PSD.
\end{proof}

\begin{remark}\label{rem:second}
As with the second proof of Theorem~\ref{thm:Goodman}, the following direct argument strips off the flag algebra terminology from the second proof of Proposition~\ref{tenth}. For simplicity, we denote $p_i = p(D_i, G)$ for every $1 \leq i \leq 7$. For every $v \in V(G)$, it holds that
\begin{align*} 
9&\binom{d_0(v)}{2} + 9\binom{d_-(v)}{2}+9\binom{d_+(v)}{2}\\ &=\frac{(3d_0(v))^2 + (2d_-(v)+2d_+(v))^2}{2} -4d_-(v)d_+(v) +5\frac{d_-(v)^2 + d_+(v)^2}{2} -\frac{9}{2}(n-1) \\
&\geq 3d_0(v)(2d_-(v)+2d_+(v))-4d_-(v)d_+(v) +5d_-(v)d_+(v)-\frac{9}{2}(n-1)\\
&=6d_0(v)d_-(v) + 6d_0(v)d_+(v) + d_-(v)d_+(v)-\frac{9}{2}(n-1).
\end{align*}
Therefore
\begin{align*} 
9(3p_1&+p_2)+9(p_3+p_6)+9(p_5+p_6) \\
&=\frac{9}{\binom{n}{3}} \sum_{v \in V} \binom{d_0(v)}{2}+\frac{9}{\binom{n}{3}} \sum_{v \in V} \binom{d_-(v)}{2}+\frac{9}{\binom{n}{3}} \sum_{v \in V} \binom{d_+(v)}{2} \\
&\geq \frac{6}{\binom{n}{3}} \sum_{v \in V} d_0(v) d_-(v) +\frac{6}{\binom{n}{3}} \sum_{v \in V} d_0(v) d_+(v) +\frac{1}{\binom{n}{3}} \sum_{v \in V} d_-(v) d_+(v) -\frac{1}{\binom{n}{3}}n\frac{9}{2}(n-1) \\
&= 6(p_2+p_4+2p_5) + 6(p_2+2p_3+p_4) + (p_4+p_6+3p_7) - \frac{27}{n-2},
\end{align*}
implying that
\begin{equation} \label{eq::p1p6}
p_1 + \textstyle\frac{2}{3} p_6 \geq  \textstyle\frac{1}{10} \left(p_1 + p_2 + p_3 + \textstyle\frac{13}{3} p_4 + p_5 + p_6 + p_7 \right) - \textstyle\frac{9}{10(n-2)}.
\end{equation}
We conclude that
\begin{align*}
t(G)+i(G) &= p_1+p_6 \geq p_1+\textstyle\frac{2}{3}p_6\geq  \textstyle\frac{1}{10}\left(p_1+p_2+p_3+\textstyle\frac{13}{3}p_4+p_5+p_6+p_7\right)-\textstyle\frac{9}{10(n-2)}\\
&\geq  \textstyle\frac{1}{10}\left(p_1+p_2+p_3+p_4+p_5+p_6+p_7\right)-\textstyle\frac{9}{10(n-2)}=\textstyle\frac{1}{10}- O\left(\textstyle\frac{1}{n}\right),
\end{align*}
where the second inequality holds by~\eqref{eq::p1p6}.
\end{remark}

The arguments used in both proofs of Proposition~\ref{tenth} can be refined to yield the following better bound.
\begin{proposition}\label{prop:refined}
Every $n$-vertex oriented graph $G$ satisfies
$$\textstyle\frac{2}{3}t(G)+i(G) \geq \textstyle\frac{1}{10} - o_n(1).$$
\end{proposition}

\begin{proof}
It is easy to verify that the matrix \eqref{eq:Qtoy2} is also a $\frac{1}{10}$-certificate for the following SDP:
\begin{eqnarray*}
\begin{aligned}
&\text{Variables: } p_1,p_2,p_3,p_4,p_5,p_6,p_7 \\
&\text{Goal:  minimize }  p_1+\textstyle\frac{2}{3}p_6  \\
&\text{Constraints: } \\
& p_1,p_2,p_3,p_4,p_5,p_6,p_7\geq 0 \\
& p_1+p_2+p_3+p_4+p_5+p_6+p_7= 1 \\
& \begin{pmatrix}
\frac{1}{3}p_5+\frac{1}{3}p_6 & \frac{1}{6}p_4+\frac{1}{6}p_6+\frac{1}{2}p_7 & \frac{1}{6}p_2+\frac{1}{3}p_3+\frac{1}{6}p_4 \\
\frac{1}{6}p_4+\frac{1}{6}p_6+\frac{1}{2}p_7 & \frac{1}{3}p_3+\frac{1}{3}p_6 & \frac{1}{6}p_2+\frac{1}{6}p_4+\frac{1}{3}p_5 \\ 
\frac{1}{6}p_2+\frac{1}{3}p_3+\frac{1}{6}p_4 & \frac{1}{6}p_2+\frac{1}{6}p_4+\frac{1}{3}p_5 & p_1+\frac{1}{3}p_2
\end{pmatrix}\succeq 0
\end{aligned} 
\end{eqnarray*}
and the result follows exactly as in the second proof of Proposition \ref{tenth} (see also Remark~\ref{rem:second}). 
Alternatively, we can prove Proposition~\ref{prop:refined} by slightly modifying the first proof of Proposition~\ref{tenth}. First, observe the following simple improvement of~\eqref{eq::fewCyclicTriangles}: 
\begin{align} \label{eq::fewCyclicTriangles+}
\frac{1}{3}t(G)+c(G) &\leq \frac{1}{3\binom{n}{3}} \sum_{v \in V} d_+(v) d_-(v) \nonumber \\
&\leq \frac{1}{3\binom{n}{3}} \sum_{v \in V} \left(\frac{d_+(v) + d_-(v)}{2}\right)^2 = \frac{1}{12\binom{n}{3}} \sum_{v \in V} d(v)^2.
\end{align}
Similarly to the first proof of Proposition~\ref{tenth}, combining~\eqref{eq::fewCyclicTriangles+} with~\eqref{eq:delta1} and~\eqref{eq::convexity} we obtain 
\begin{align*}
\frac{2}{3}t(G) + i(G) &= \Delta+\bar\Delta-\left(\frac{1}{3}t(G)+c(G)\right)\geq 1 - \frac{6m}{n(n-2)} + \frac{5}{12\binom{n}{3}} \sum_{v \in V} d(v)^2 \\
&\geq  1 - \frac{6m}{n(n-2)} + \frac{5m^2}{3n\binom{n}{3}} = 1 - \frac{6m}{n^2} + \frac{10m^2}{n^4} - O\left(\frac{1}{n}\right) \\
&= \frac{1}{10}+\frac{(3n^2 - 10m)^2}{10n^4} -  O\left(\frac{1}{n}\right)\geq \frac{1}{10}  -  O\left(\frac{1}{n}\right).
\qedhere
\end{align*}
\end{proof}
As noted in the introduction
\begin{equation}\label{eq:DHMNS+PV}
i(\tilde{G}) \geq \textstyle\frac{1}{9} - o_n(1)
\end{equation}
for every undirected $K_4$-free $n$-vertex graph $\tilde{G}$. This was proved in~\cite{DHMNS} and, independently, in~\cite{pikhurko-vaughan} using flag algebras. Combined with Proposition~\ref{prop:refined}, this yields the following slight improvement of Proposition~\ref{tenth}.
\begin{proposition}
$$
\tau > \textstyle\frac{1}{10}.
$$
\end{proposition}

\begin{proof}
Let $G_n$ be an oriented graph that attains the minimum of $t(G) + i(G)$ among all $n$-vertex oriented graphs; namely, $t(G_n) + i(G_n) = \tau(n)$.

By the graph removal lemma~\cite{ADLRY} (see also~\cite{CF} and the many references therein), there is a positive integer $n_0$ and a real number  $\delta_0 > 0$ such that every (undirected) graph $H$ on $n \geq n_0$ vertices for which $p(K_4, H) < \delta_0$, can be made $K_4$-free by deleting at most $\frac{n (n-1)}{1080}$ edges. 

Choose some $0 < \delta < \min \{\delta_0, \frac{1}{15}\}$ and suppose for a contradiction that $\tau \leq \frac{1}{10} + \frac{1}{12} \delta$. It follows by Proposition~\ref{prop:refined} that there is a positive integer $n_1$ such that for every oriented graph $G$ on $n\geq n_1$ vertices,
$$
\textstyle\frac{2}{3}t(G)+i(G)\geq\textstyle\frac{1}{10}-\textstyle\frac{1}{12}\delta.
$$
For every $n \geq \max\{n_0,n_1\}$, it holds that
\begin{align*}
\textstyle\frac{1}{3}t(G_n)&=\left(t(G_n)+i(G_n)\right)-\left(\textstyle\frac{2}{3}t(G_n)+i(G_n)\right)=\tau(n)-\left(\textstyle\frac{2}{3}t(G_n)+i(G_n)\right)\\
&\leq\tau-\left(\textstyle\frac{1}{10}-\textstyle\frac{1}{12}\delta\right)\leq\left(\textstyle\frac{1}{10}+\textstyle\frac{1}{12}\delta\right)-\left(\textstyle\frac{1}{10}-\textstyle\frac{1}{12}\delta\right)=\textstyle\frac{1}{6}\delta.
\end{align*}
Let $G_n^{(0)}$ be the underlying undirected graph of $G_n$. Since every orientation of $K_4$ contains at least two transitive triangles, it follows that
$$
p(K_4,G_n^{(0)}) \leq 2t(G_n) \leq \delta < \delta_0.
$$
Therefore, by the graph removal lemma, there is an undirected $K_4$-free graph $G_n^{(1)}$, obtained from $G_n^{(0)}$ by deleting at most $\frac{n (n-1)}{1080}$ edges. Therefore
\begin{align*}
i(G_n^{(1)})&\leq i(G_n^{(0)})+6\textstyle\frac{1}{1080}=i(G_n)+\textstyle\frac{1}{180}\leq t(G_n)+i(G_n)+\textstyle\frac{1}{180}\\
&=\tau(n)+\textstyle\frac{1}{180}\leq \tau+\frac{1}{180}\leq\textstyle\frac{1}{10}+\textstyle\frac{1}{12}\delta+\textstyle\frac{1}{180}=\textstyle\frac{1}{9}-\textstyle\frac{1}{12}\left(\textstyle\frac{1}{15}-\delta\right)
\end{align*}
contrary to~\eqref{eq:DHMNS+PV}.
We conclude that, $\tau>\frac{1}{10} + \frac{1}{12} \delta > \frac{1}{10}$.
\end{proof}

\section{Back to the main track}\label{sec:back}
Running flagmatic with $k = 3$ yields $\tau \geq 1/10$, whereas our goal is to prove that $\tau \geq 1/9$. Therefore, we try the same proof technique with $k=4$. Figure~\ref{fig:4graphs} depicts all of the different $4$-vertex oriented graphs, up to isomorphism, $G_1, \ldots, G_{42}$. 
For convenience we abbreviate $A_{G_i}$ to $A_i$.

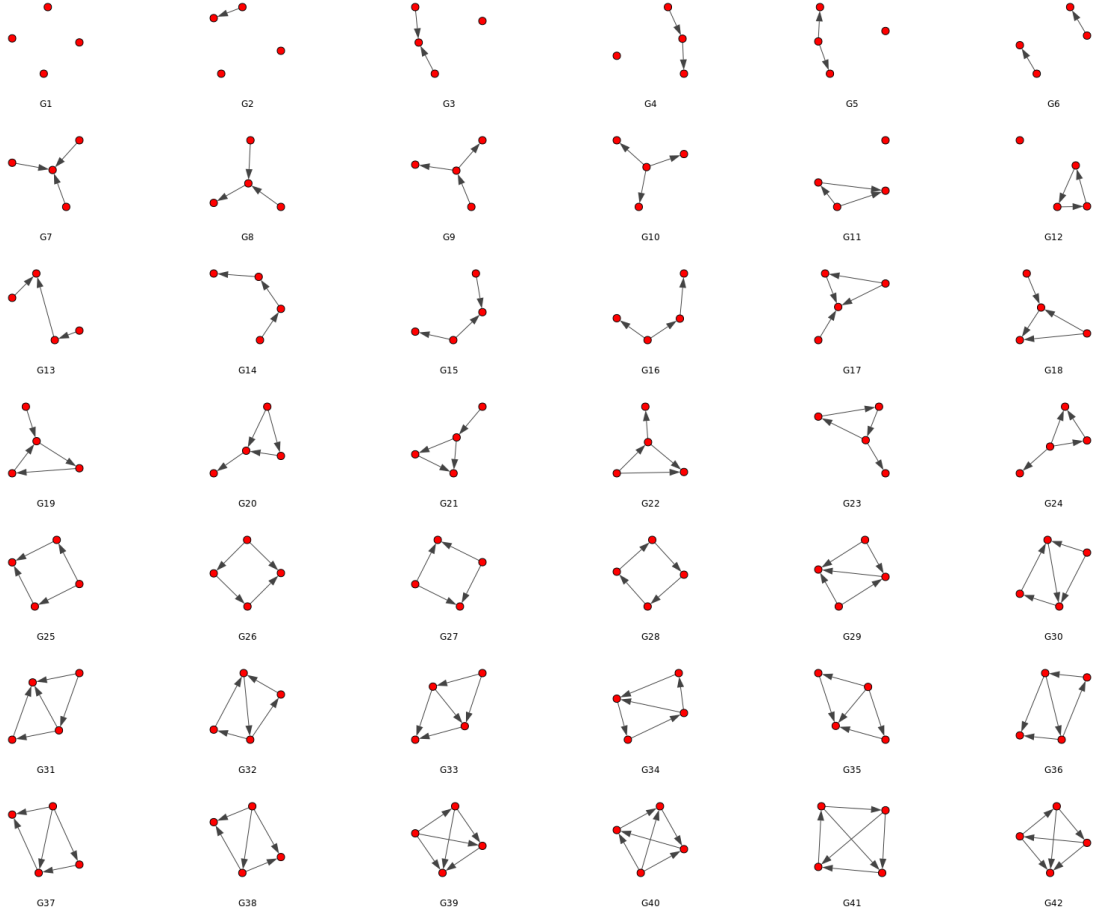
\begin{figure}
\centering
\begin{tikzpicture}

\Vertex[x=-6,y=9.6, size=0.2]{A11};
\Vertex[x=-6,y=8.9, size=0.2]{B11};
\Vertex[x=-6.6,y=8.6, size=0.2]{C11};
\Vertex[x=-5.4,y=8.6, size=0.2]{D11};
\Text[x=-6,y=8.1]{$G_{1}$}

\Vertex[x=-3.5,y=9.6, size=0.2]{A12};
\Vertex[x=-3.5,y=8.9, size=0.2]{B12};
\Vertex[x=-4.1,y=8.6, size=0.2]{C12};
\Vertex[x=-2.9,y=8.6, size=0.2]{D12};
\Text[x=-3.5,y=8.1]{$G_{2}$}
\Edge[Direct](C12)(D12);

\Vertex[x=-1,y=9.6, size=0.2]{A13};
\Vertex[x=-1,y=8.9, size=0.2]{B13};
\Vertex[x=-1.6,y=8.6, size=0.2]{C13};
\Vertex[x=-0.4,y=8.6, size=0.2]{D13};
\Text[x=-1,y=8.1]{$G_{3}$}
\Edge[Direct](A13)(D13);
\Edge[Direct](C13)(D13);

\Vertex[x=1.5,y=9.6, size=0.2]{A14};
\Vertex[x=1.5,y=8.9, size=0.2]{B14};
\Vertex[x=0.9,y=8.6, size=0.2]{C14};
\Vertex[x=2.1,y=8.6, size=0.2]{D14};
\Text[x=1.5,y=8.1]{$G_{4}$}
\Edge[Direct](C14)(A14);
\Edge[Direct](A14)(D14);

\Vertex[x=4,y=9.6, size=0.2]{A15};
\Vertex[x=4,y=8.9, size=0.2]{B15};
\Vertex[x=3.4,y=8.6, size=0.2]{C15};
\Vertex[x=4.6,y=8.6, size=0.2]{D15};
\Text[x=4,y=8.1]{$G_{5}$}
\Edge[Direct](C15)(A15);
\Edge[Direct](C15)(D15);

\Vertex[x=6.5,y=9.6, size=0.2]{A16};
\Vertex[x=6.5,y=8.9, size=0.2]{B16};
\Vertex[x=5.9,y=8.6, size=0.2]{C16};
\Vertex[x=7.1,y=8.6, size=0.2]{D16};
\Text[x=6.5,y=8.1]{$G_{6}$}
\Edge[Direct](C16)(D16);
\Edge[Direct](A16)(B16);

\Vertex[x=-6,y=6.9, size=0.2]{A21};
\Vertex[x=-6,y=6.2, size=0.2]{B21};
\Vertex[x=-6.6,y=5.9, size=0.2]{C21};
\Vertex[x=-5.4,y=5.9, size=0.2]{D21};
\Text[x=-6,y=5.4]{$G_{7}$}
\Edge[Direct](A21)(B21);
\Edge[Direct](C21)(B21);
\Edge[Direct](D21)(B21);

\Vertex[x=-3.5,y=6.9, size=0.2]{A22};
\Vertex[x=-3.5,y=6.2, size=0.2]{B22};
\Vertex[x=-4.1,y=5.9, size=0.2]{C22};
\Vertex[x=-2.9,y=5.9, size=0.2]{D22};
\Text[x=-3.5,y=5.4]{$G_{8}$}
\Edge[Direct](A22)(B22);
\Edge[Direct](C22)(B22);
\Edge[Direct](B22)(D22);

\Vertex[x=-1,y=6.9, size=0.2]{A23};
\Vertex[x=-1,y=6.2, size=0.2]{B23};
\Vertex[x=-1.6,y=5.9, size=0.2]{C23};
\Vertex[x=-0.4,y=5.9, size=0.2]{D23};
\Text[x=-1,y=5.4]{$G_{9}$}
\Edge[Direct](A23)(B23);
\Edge[Direct](B23)(C23);
\Edge[Direct](B23)(D23);

\Vertex[x=1.5,y=6.9, size=0.2]{A24};
\Vertex[x=1.5,y=6.2, size=0.2]{B24};
\Vertex[x=0.9,y=5.9, size=0.2]{C24};
\Vertex[x=2.1,y=5.9, size=0.2]{D24};
\Text[x=1.5,y=5.4]{$G_{10}$}
\Edge[Direct](B24)(A24);
\Edge[Direct](B24)(C24);
\Edge[Direct](B24)(D24);

\Vertex[x=4,y=6.9, size=0.2]{A25};
\Vertex[x=4,y=6.2, size=0.2]{B25};
\Vertex[x=3.4,y=5.9, size=0.2]{C25};
\Vertex[x=4.6,y=5.9, size=0.2]{D25};
\Text[x=4,y=5.4]{$G_{11}$}
\Edge[Direct](C25)(A25);
\Edge[Direct](A25)(D25);
\Edge[Direct](C25)(D25);

\Vertex[x=6.5,y=6.9, size=0.2]{A26};
\Vertex[x=6.5,y=6.2, size=0.2]{B26};
\Vertex[x=5.9,y=5.9, size=0.2]{C26};
\Vertex[x=7.1,y=5.9, size=0.2]{D26};
\Text[x=6.5,y=5.4]{$G_{12}$}
\Edge[Direct](C26)(A26);
\Edge[Direct](A26)(D26);
\Edge[Direct](D26)(C26);

\Vertex[x=-6,y=4.2, size=0.2]{A31};
\Vertex[x=-6,y=3.5, size=0.2]{B31};
\Vertex[x=-6.6,y=3.2, size=0.2]{C31};
\Vertex[x=-5.4,y=3.2, size=0.2]{D31};
\Text[x=-6,y=2.7]{$G_{13}$}
\Edge[Direct](A31)(D31);
\Edge[Direct](C31)(D31);
\Edge[Direct](B31)(A31);

\Vertex[x=-3.5,y=4.2, size=0.2]{A32};
\Vertex[x=-3.5,y=3.5, size=0.2]{B32};
\Vertex[x=-4.1,y=3.2, size=0.2]{C32};
\Vertex[x=-2.9,y=3.2, size=0.2]{D32};
\Text[x=-3.5,y=2.7]{$G_{14}$}
\Edge[Direct](A32)(D32);
\Edge[Direct](D32)(C32);
\Edge[Direct](B32)(A32);

\Vertex[x=-1,y=4.2, size=0.2]{A33};
\Vertex[x=-1,y=3.5, size=0.2]{B33};
\Vertex[x=-1.6,y=3.2, size=0.2]{C33};
\Vertex[x=-0.4,y=3.2, size=0.2]{D33};
\Text[x=-1,y=2.7]{$G_{15}$}
\Edge[Direct](A33)(D33);
\Edge[Direct](C33)(D33);
\Edge[Direct](A33)(B33);

\Vertex[x=1.5,y=4.2, size=0.2]{A34};
\Vertex[x=1.5,y=3.5, size=0.2]{B34};
\Vertex[x=0.9,y=3.2, size=0.2]{C34};
\Vertex[x=2.1,y=3.2, size=0.2]{D34};
\Text[x=1.5,y=2.7]{$G_{16}$}
\Edge[Direct](A34)(D34);
\Edge[Direct](D34)(C34);
\Edge[Direct](A34)(B34);

\Vertex[x=4,y=4.2, size=0.2]{A35};
\Vertex[x=4,y=3.5, size=0.2]{B35};
\Vertex[x=3.4,y=3.2, size=0.2]{C35};
\Vertex[x=4.6,y=3.2, size=0.2]{D35};
\Text[x=4,y=2.7]{$G_{17}$}
\Edge[Direct](A35)(B35);
\Edge[Direct](C35)(B35);
\Edge[Direct](D35)(B35);
\Edge[Direct](A35)(D35);

\Vertex[x=6.5,y=4.2, size=0.2]{A36};
\Vertex[x=6.5,y=3.5, size=0.2]{B36};
\Vertex[x=5.9,y=3.2, size=0.2]{C36};
\Vertex[x=7.1,y=3.2, size=0.2]{D36};
\Text[x=6.5,y=2.7]{$G_{18}$}
\Edge[Direct](A36)(B36);
\Edge[Direct](C36)(B36);
\Edge[Direct](B36)(D36);
\Edge[Direct](A36)(D36);

\Vertex[x=-6,y=1.5, size=0.2]{A41};
\Vertex[x=-6,y=0.8, size=0.2]{B41};
\Vertex[x=-6.6,y=0.5, size=0.2]{C41};
\Vertex[x=-5.4,y=0.5, size=0.2]{D41};
\Text[x=-6,y=0]{$G_{19}$}
\Edge[Direct](A41)(B41);
\Edge[Direct](C41)(B41);
\Edge[Direct](B41)(D41);
\Edge[Direct](D41)(A41);

\Vertex[x=-3.5,y=1.5, size=0.2]{A42};
\Vertex[x=-3.5,y=0.8, size=0.2]{B42};
\Vertex[x=-4.1,y=0.5, size=0.2]{C42};
\Vertex[x=-2.9,y=0.5, size=0.2]{D42};
\Text[x=-3.5,y=0]{$G_{20}$}
\Edge[Direct](A42)(B42);
\Edge[Direct](C42)(B42);
\Edge[Direct](B42)(D42);
\Edge[Direct](C42)(A42);

\Vertex[x=-1,y=1.5, size=0.2]{A43};
\Vertex[x=-1,y=0.8, size=0.2]{B43};
\Vertex[x=-1.6,y=0.5, size=0.2]{C43};
\Vertex[x=-0.4,y=0.5, size=0.2]{D43};
\Text[x=-1,y=0]{$G_{21}$}
\Edge[Direct](A43)(B43);
\Edge[Direct](B43)(C43);
\Edge[Direct](B43)(D43);
\Edge[Direct](C43)(D43);

\Vertex[x=1.5,y=1.5, size=0.2]{A44};
\Vertex[x=1.5,y=0.8, size=0.2]{B44};
\Vertex[x=0.9,y=0.5, size=0.2]{C44};
\Vertex[x=2.1,y=0.5, size=0.2]{D44};
\Text[x=1.5,y=0]{$G_{22}$}
\Edge[Direct](A44)(B44);
\Edge[Direct](B44)(C44);
\Edge[Direct](B44)(D44);
\Edge[Direct](A44)(C44);

\Vertex[x=4,y=1.5, size=0.2]{A45};
\Vertex[x=4,y=0.8, size=0.2]{B45};
\Vertex[x=3.4,y=0.5, size=0.2]{C45};
\Vertex[x=4.6,y=0.5, size=0.2]{D45};
\Text[x=4,y=0]{$G_{23}$}
\Edge[Direct](A45)(B45);
\Edge[Direct](B45)(C45);
\Edge[Direct](B45)(D45);
\Edge[Direct](C45)(A45);

\Vertex[x=6.5,y=1.5, size=0.2]{A46};
\Vertex[x=6.5,y=0.8, size=0.2]{B46};
\Vertex[x=5.9,y=0.5, size=0.2]{C46};
\Vertex[x=7.1,y=0.5, size=0.2]{D46};
\Text[x=6.5,y=0]{$G_{24}$}
\Edge[Direct](B46)(A46);
\Edge[Direct](B46)(C46);
\Edge[Direct](B46)(D46);
\Edge[Direct](C46)(A46);

\Vertex[x=-6,y=-1.2, size=0.2]{A51};
\Vertex[x=-6,y=-1.9, size=0.2]{B51};
\Vertex[x=-6.6,y=-2.2, size=0.2]{C51};
\Vertex[x=-5.4,y=-2.2, size=0.2]{D51};
\Text[x=-6,y=-2.7]{$G_{25}$}
\Edge[Direct](C51)(A51);
\Edge[Direct](B51)(C51);
\Edge[Direct](B51)(D51);
\Edge[Direct](D51)(A51);

\Vertex[x=-3.5,y=-1.2, size=0.2]{A52};
\Vertex[x=-3.5,y=-1.9, size=0.2]{B52};
\Vertex[x=-4.1,y=-2.2, size=0.2]{C52};
\Vertex[x=-2.9,y=-2.2, size=0.2]{D52};
\Text[x=-3.5,y=-2.7]{$G_{26}$}
\Edge[Direct](C52)(A52);
\Edge[Direct](C52)(B52);
\Edge[Direct](B52)(D52);
\Edge[Direct](D52)(A52);

\Vertex[x=-1,y=-1.2, size=0.2]{A53};
\Vertex[x=-1,y=-1.9, size=0.2]{B53};
\Vertex[x=-1.6,y=-2.2, size=0.2]{C53};
\Vertex[x=-0.4,y=-2.2, size=0.2]{D53};
\Text[x=-1,y=-2.7]{$G_{27}$}
\Edge[Direct](C53)(A53);
\Edge[Direct](C53)(B53);
\Edge[Direct](D53)(B53);
\Edge[Direct](D53)(A53);

\Vertex[x=1.5,y=-1.2, size=0.2]{A54};
\Vertex[x=1.5,y=-1.9, size=0.2]{B54};
\Vertex[x=0.9,y=-2.2, size=0.2]{C54};
\Vertex[x=2.1,y=-2.2, size=0.2]{D54};
\Text[x=1.5,y=-2.7]{$G_{28}$}
\Edge[Direct](A54)(C54);
\Edge[Direct](C54)(B54);
\Edge[Direct](B54)(D54);
\Edge[Direct](D54)(A54);

\Vertex[x=4,y=-1.2, size=0.2]{A55};
\Vertex[x=4,y=-1.9, size=0.2]{B55};
\Vertex[x=3.4,y=-2.2, size=0.2]{C55};
\Vertex[x=4.6,y=-2.2, size=0.2]{D55};
\Text[x=4,y=-2.7]{$G_{29}$}
\Edge[Direct](C55)(A55);
\Edge[Direct](D55)(A55);
\Edge[Direct](A55)(B55);
\Edge[Direct](C55)(B55);
\Edge[Direct](D55)(B55);

\Vertex[x=6.5,y=-1.2, size=0.2]{A56};
\Vertex[x=6.5,y=-1.9, size=0.2]{B56};
\Vertex[x=5.9,y=-2.2, size=0.2]{C56};
\Vertex[x=7.1,y=-2.2, size=0.2]{D56};
\Text[x=6.5,y=-2.7]{$G_{30}$}
\Edge[Direct](C56)(A56);
\Edge[Direct](D56)(A56);
\Edge[Direct](A56)(B56);
\Edge[Direct](C56)(B56);
\Edge[Direct](B56)(D56);

\Vertex[x=-6,y=-3.9, size=0.2]{A61};
\Vertex[x=-6,y=-4.6, size=0.2]{B61};
\Vertex[x=-6.6,y=-4.9, size=0.2]{C61};
\Vertex[x=-5.4,y=-4.9, size=0.2]{D61};
\Text[x=-6,y=-5.4]{$G_{31}$}
\Edge[Direct](C61)(A61);
\Edge[Direct](A61)(D61);
\Edge[Direct](A61)(B61);
\Edge[Direct](C61)(B61);
\Edge[Direct](D61)(B61);

\Vertex[x=-3.5,y=-3.9, size=0.2]{A62};
\Vertex[x=-3.5,y=-4.6, size=0.2]{B62};
\Vertex[x=-4.1,y=-4.9, size=0.2]{C62};
\Vertex[x=-2.9,y=-4.9, size=0.2]{D62};
\Text[x=-3.5,y=-5.4]{$G_{32}$}
\Edge[Direct](A62)(C62);
\Edge[Direct](A62)(D62);
\Edge[Direct](B62)(A62);
\Edge[Direct](C62)(B62);
\Edge[Direct](D62)(B62);

\Vertex[x=-1,y=-3.9, size=0.2]{A63};
\Vertex[x=-1,y=-4.6, size=0.2]{B63};
\Vertex[x=-1.6,y=-4.9, size=0.2]{C63};
\Vertex[x=-0.4,y=-4.9, size=0.2]{D63};
\Text[x=-1,y=-5.4]{$G_{33}$}
\Edge[Direct](C63)(A63);
\Edge[Direct](A63)(D63);
\Edge[Direct](A63)(B63);
\Edge[Direct](C63)(B63);
\Edge[Direct](B63)(D63);

\Vertex[x=1.5,y=-3.9, size=0.2]{A64};
\Vertex[x=1.5,y=-4.6, size=0.2]{B64};
\Vertex[x=0.9,y=-4.9, size=0.2]{C64};
\Vertex[x=2.1,y=-4.9, size=0.2]{D64};
\Text[x=1.5,y=-5.4]{$G_{34}$}
\Edge[Direct](C64)(A64);
\Edge[Direct](A64)(D64);
\Edge[Direct](A64)(B64);
\Edge[Direct](B64)(C64);
\Edge[Direct](D64)(B64);

\Vertex[x=4,y=-3.9, size=0.2]{A65};
\Vertex[x=4,y=-4.6, size=0.2]{B65};
\Vertex[x=3.4,y=-4.9, size=0.2]{C65};
\Vertex[x=4.6,y=-4.9, size=0.2]{D65};
\Text[x=4,y=-5.4]{$G_{35}$}
\Edge[Direct](A65)(C65);
\Edge[Direct](A65)(D65);
\Edge[Direct](A65)(B65);
\Edge[Direct](C65)(B65);
\Edge[Direct](D65)(B65);

\Vertex[x=6.5,y=-3.9, size=0.2]{A66};
\Vertex[x=6.5,y=-4.6, size=0.2]{B66};
\Vertex[x=5.9,y=-4.9, size=0.2]{C66};
\Vertex[x=7.1,y=-4.9, size=0.2]{D66};
\Text[x=6.5,y=-5.4]{$G_{36}$}
\Edge[Direct](C66)(A66);
\Edge[Direct](A66)(D66);
\Edge[Direct](A66)(B66);
\Edge[Direct](B66)(C66);
\Edge[Direct](B66)(D66);

\Vertex[x=-6,y=-6.6, size=0.2]{A71};
\Vertex[x=-6,y=-7.3, size=0.2]{B71};
\Vertex[x=-6.6,y=-7.6, size=0.2]{C71};
\Vertex[x=-5.4,y=-7.6, size=0.2]{D71};
\Text[x=-6,y=-8.1]{$G_{37}$}
\Edge[Direct](A71)(C71);
\Edge[Direct](A71)(D71);
\Edge[Direct](A71)(B71);
\Edge[Direct](C71)(B71);
\Edge[Direct](B71)(D71);

\Vertex[x=-3.5,y=-6.6, size=0.2]{A72};
\Vertex[x=-3.5,y=-7.3, size=0.2]{B72};
\Vertex[x=-4.1,y=-7.6, size=0.2]{C72};
\Vertex[x=-2.9,y=-7.6, size=0.2]{D72};
\Text[x=-3.5,y=-8.1]{$G_{38}$}
\Edge[Direct](A72)(C72);
\Edge[Direct](A72)(D72);
\Edge[Direct](A72)(B72);
\Edge[Direct](B72)(C72);
\Edge[Direct](B72)(D72);

\Vertex[x=-1,y=-6.6, size=0.2]{A73};
\Vertex[x=-1,y=-7.3, size=0.2]{B73};
\Vertex[x=-1.6,y=-7.6, size=0.2]{C73};
\Vertex[x=-0.4,y=-7.6, size=0.2]{D73};
\Text[x=-1,y=-8.1]{$G_{39}$}
\Edge[Direct](C73)(A73);
\Edge[Direct](A73)(D73);
\Edge[Direct](C73)(D73);
\Edge[Direct](A73)(B73);
\Edge[Direct](C73)(B73);
\Edge[Direct](D73)(B73);

\Vertex[x=1.5,y=-6.6, size=0.2]{A74};
\Vertex[x=1.5,y=-7.3, size=0.2]{B74};
\Vertex[x=0.9,y=-7.6, size=0.2]{C74};
\Vertex[x=2.1,y=-7.6, size=0.2]{D74};
\Text[x=1.5,y=-8.1]{$G_{40}$}
\Edge[Direct](C74)(A74);
\Edge[Direct](A74)(D74);
\Edge[Direct](C74)(D74);
\Edge[Direct](B74)(A74);
\Edge[Direct](C74)(B74);
\Edge[Direct](D74)(B74);

\Vertex[x=4,y=-6.6,size=0.2]{A75};
\Vertex[x=4,y=-7.3, size=0.2]{B75};
\Vertex[x=3.4,y=-7.6, size=0.2]{C75};
\Vertex[x=4.6,y=-7.6, size=0.2]{D75};
\Text[x=4,y=-8.1]{$G_{41}$}
\Edge[Direct](C75)(A75);
\Edge[Direct](A75)(D75);
\Edge[Direct](D75)(C75);
\Edge[Direct](B75)(A75);
\Edge[Direct](C75)(B75);
\Edge[Direct](D75)(B75);

\Vertex[x=6.5,y=-6.6, size=0.2]{A76};
\Vertex[x=6.5,y=-7.3, size=0.2]{B76};
\Vertex[x=5.9,y=-7.6, size=0.2]{C76};
\Vertex[x=7.1,y=-7.6, size=0.2]{D76};
\Text[x=6.5,y=-8.1]{$G_{42}$}
\Edge[Direct](C76)(A76);
\Edge[Direct](A76)(D76);
\Edge[Direct](D76)(C76);
\Edge[Direct](A76)(B76);
\Edge[Direct](C76)(B76);
\Edge[Direct](D76)(B76);

\end{tikzpicture}
\caption{The $42$ isomorphism types of oriented graphs of order $4$.}
\label{fig:4graphs}
\end{figure}

As was elaborated in Subsection~\ref{subsec:SDP}, we use the set of types $\Sigma = \{\emptyset, \bar{E}, E\}$, where \emph{the empty type} $\emptyset$ has no vertices, \emph{the non-edge type} $\bar{E}$ has two vertices and no edges, and \emph{the edge type} $E$ has two vertices and the edge $(1,2)$ which is  directed from the vertex labelled $1$ to the vertex labelled $2$ (we will not use the type having two vertices and an edge in the opposite direction, as it will clearly provide no additional information).
Although the empty type is not really necessary (i.e., we can obtain the same results without it), we keep it, since it helps in illustrating some of our calculations.

As was further elaborated in Subsection~\ref{subsec:SDP}, the set of flags that we use is ${\mathcal F} = {\mathcal F}_{\emptyset} \cup {\mathcal F}_{\bar{E}} \cup {\mathcal F}_{E}$, where ${\mathcal F}_{\emptyset}$ is the set of all flags over $\emptyset$ with $2$ petals (see Figure~\ref{fig:flags_empty}), ${\mathcal F}_{\bar{E}}$ is the set of all flags over $\bar{E}$ with $1$ petal (see Figure~\ref{fig:flags_nonedge}), and ${\mathcal F}_{E}$ is the set of all flags over $E$ with $1$ petal (see Figure~\ref{fig:flags_edge}). Observe that $|{\mathcal F}_{\emptyset}|=2$ and $|{\mathcal F}_{\bar{E}}| = |{\mathcal F}_{E}|=9$. Hence, in total, $|{\mathcal F}| = 2 + 9 + 9 = 20$. 

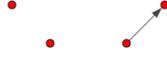
\begin{figure}
\centering
\begin{tikzpicture}

\Vertex[x=-2,y=-0.7, size=0.4]{A1};
\Vertex[x=-2,y=0.7, size=0.4]{B1};

\Vertex[x=2,y=-0.7, size=0.4]{A2};
\Vertex[x=2,y=0.7, size=0.4]{B2};
\Edge[Direct](B2)(A2);

\end{tikzpicture}
\caption{Flags over the empty type.}
\label{fig:flags_empty}
\end{figure}

\begin{figure}
\centering
\begin{tikzpicture}

\Vertex[x=-6.7,y=2.4, label=$1$, size=0.4]{A1};
\Vertex[x=-5.3,y=2.4, label=$2$, size=0.4]{B1};
\Vertex[x=-6,y=3.6,  size=0.4]{C1};

\Vertex[x=-3.7,y=2.4,label=$1$, size=0.4]{A2};
\Vertex[x=-2.3,y=2.4, label=$2$, size=0.4]{B2};
\Vertex[x=-3,y=3.6,  size=0.4]{C2};
\Edge[Direct](A2)(C2);

\Vertex[x=-0.7,y=2.4, label=$1$, size=0.4]{A3};
\Vertex[x=0.7,y=2.4, label=$2$, size=0.4]{B3};
\Vertex[x=0,y=3.6,  size=0.4]{C3};
\Edge[Direct](C3)(A3);

\Vertex[x=2.3,y=2.4, label=$1$, size=0.4]{A4};
\Vertex[x=3.7,y=2.4, label=$2$, size=0.4]{B4};
\Vertex[x=3,y=3.6,  size=0.4]{C4};
\Edge[Direct](B4)(C4);

\Vertex[x=5.3,y=2.4,label=$1$, size=0.4]{A5};
\Vertex[x=6.7,y=2.4, label=$2$, size=0.4]{B5};
\Vertex[x=6,y=3.6,  size=0.4]{C5};
\Edge[Direct](C5)(B5);

\Vertex[x=-5.2,y=0, label=$1$, size=0.4]{A6};
\Vertex[x=-3.8,y=0, label=$2$, size=0.4]{B6};
\Vertex[x=-4.5,y=1.2,  size=0.4]{C6};
\Edge[Direct](A6)(C6);
\Edge[Direct](B6)(C6);

\Vertex[x=-2.2,y=0, label=$1$, size=0.4]{A7};
\Vertex[x=-0.8,y=0, label=$2$, size=0.4]{B7};
\Vertex[x=-1.5,y=1.2,  size=0.4]{C7};
\Edge[Direct](A7)(C7);
\Edge[Direct](C7)(B7);

\Vertex[x=0.8,y=0,label=$1$, size=0.4]{A8};
\Vertex[x=2.2,y=0, label=$2$, size=0.4]{B8};
\Vertex[x=1.5,y=1.2,  size=0.4]{C8};
\Edge[Direct](B8)(C8);
\Edge[Direct](C8)(A8);

\Vertex[x=3.8,y=0, label=$1$, size=0.4]{A9};
\Vertex[x=5.2,y=0, label=$2$, size=0.4]{B9};
\Vertex[x=4.5,y=1.2,  size=0.4]{C9};
\Edge[Direct](C9)(A9);
\Edge[Direct](C9)(B9);

\end{tikzpicture}
\caption{Flags over the non-edge type.}
\label{fig:flags_nonedge}
\end{figure}
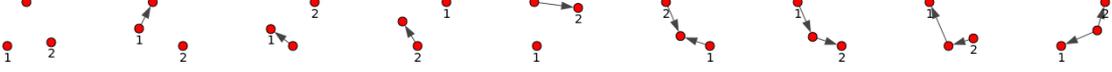

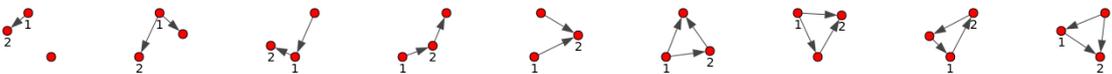
\begin{figure}
\centering
\begin{tikzpicture}

\Vertex[x=-6.7,y=2.4, label=$1$, size=0.4]{A1};
\Vertex[x=-5.3,y=2.4, label=$2$, size=0.4]{B1};
\Vertex[x=-6,y=3.6,  size=0.4]{C1};
\Edge[Direct](A1)(B1);

\Vertex[x=-3.7,y=2.4,label=$1$, size=0.4]{A2};
\Vertex[x=-2.3,y=2.4, label=$2$, size=0.4]{B2};
\Vertex[x=-3,y=3.6,  size=0.4]{C2};
\Edge[Direct](A2)(B2);
\Edge[Direct](A2)(C2);

\Vertex[x=-0.7,y=2.4, label=$1$, size=0.4]{A3};
\Vertex[x=0.7,y=2.4, label=$2$, size=0.4]{B3};
\Vertex[x=0,y=3.6,  size=0.4]{C3};
\Edge[Direct](A3)(B3);
\Edge[Direct](C3)(A3);

\Vertex[x=2.3,y=2.4, label=$1$, size=0.4]{A4};
\Vertex[x=3.7,y=2.4, label=$2$, size=0.4]{B4};
\Vertex[x=3,y=3.6,  size=0.4]{C4};
\Edge[Direct](A4)(B4);
\Edge[Direct](B4)(C4);

\Vertex[x=5.3,y=2.4,label=$1$, size=0.4]{A5};
\Vertex[x=6.7,y=2.4, label=$2$, size=0.4]{B5};
\Vertex[x=6,y=3.6,  size=0.4]{C5};
\Edge[Direct](A5)(B5);
\Edge[Direct](C5)(B5);

\Vertex[x=-5.2,y=0, label=$1$, size=0.4]{A6};
\Vertex[x=-3.8,y=0, label=$2$, size=0.4]{B6};
\Vertex[x=-4.5,y=1.2,  size=0.4]{C6};
\Edge[Direct](A6)(B6);
\Edge[Direct](A6)(C6);
\Edge[Direct](B6)(C6);

\Vertex[x=-2.2,y=0, label=$1$, size=0.4]{A7};
\Vertex[x=-0.8,y=0, label=$2$, size=0.4]{B7};
\Vertex[x=-1.5,y=1.2,  size=0.4]{C7};
\Edge[Direct](A7)(B7);
\Edge[Direct](A7)(C7);
\Edge[Direct](C7)(B7);

\Vertex[x=0.8,y=0,label=$1$, size=0.4]{A8};
\Vertex[x=2.2,y=0, label=$2$, size=0.4]{B8};
\Vertex[x=1.5,y=1.2,  size=0.4]{C8};
\Edge[Direct](A8)(B8);
\Edge[Direct](B8)(C8);
\Edge[Direct](C8)(A8);

\Vertex[x=3.8,y=0, label=$1$, size=0.4]{A9};
\Vertex[x=5.2,y=0, label=$2$, size=0.4]{B9};
\Vertex[x=4.5,y=1.2,  size=0.4]{C9};
\Edge[Direct](A9)(B9);
\Edge[Direct](C9)(A9);
\Edge[Direct](C9)(B9);

\end{tikzpicture}
\caption{Flags over the edge type.}
\label{fig:flags_edge}
\end{figure}

For every $1 \leq i \leq 42$, let $c_i = t(G_i) + i(G_i)$. As was explained in Subsection~\ref{subsec:SDP}, we seek the optimum of the following semidefinite program.

\begin{eqnarray}\label{sdp} 
\begin{aligned}
&\text{Variables: } p_1,\ldots,p_{42} \\
&\text{Goal: } minimize \sum_{i=1}^{42} p_i c_i \\
&\text{Constraints: } \\
& p_1,\ldots,p_{42} \geq 0 \\
& \sum_{i=1}^{42} p_i = 1 \\
& \sum_{i=1}^{42} p_i A_i \succeq 0 
\end{aligned} 
\end{eqnarray}

Setting $p_i = \lim_{n \to \infty} p(G_i, {\mathcal B}_n)$ for every $1 \leq i \leq 42$, that is, $p_1 = 1/27, \; p_7 = 4/27, \; p_{10} = 4/27, \; p_{27} = 6/27, \; p_{32} = 12/27$, and $p_i = 0$ for every $i \in [42] \setminus \{1, 7, 10, 27, 32\}$, shows that the optimum of SDP~\eqref{sdp} is at most $1/9$. 

By Proposition \ref{prop:certificate}, the following theorem implies Theorem~\ref{main}. 

\begin{theorem} \label{thm2}
There is a $\frac{1}{9}$-certificate $Q$ for the SDP~\eqref{sdp}.
\end{theorem}

Our goal is thus to prove Theorem~\ref{thm2} by finding such a matrix $Q$. To this end we ran the csdp solver~\cite{csdp} on SDP~\eqref{sdp}. Since this solver is inherently an approximation algorithm, it does not necessarily output the true optimum (inaccuracies may also be incurred due to the computational complexity of this task, the computer's limited numerical precision, and the fact that it operates with floating point). Given any arbitrarily small constant $\eta > 0$, we can only check whether the optimum is at least $\frac{1}{9} - \eta$. We chose $\eta = 10^{-8}$. 
Crucially, the solver's output includes a {\em rational} certificate showing that the optimum is indeed at least $\frac{1}{9} - \eta$. We have \emph{rounded} that certificate to rational numbers with $4$ decimal digits. We use this output as the starting point for the remainder of the proof. We aim to slightly perturb this certificate so as to make it a $\frac{1}{9}$-certificate. We start by finding certain constraints that any proper $\frac{1}{9}$-certificate must satisfy.

It clearly suffices to consider certificates $Q$ with the same block structure as the matrices $A_{i}$, i.e., block-diagonal matrices with blocks $Q_{\emptyset}, Q_{\bar{E}}, Q_{E}$ of sizes $2 \times 2$, $9 \times 9$, and $9 \times 9$, respectively. Since $Q$ is symmetric, we presently have only  $\binom{3}{2} + 2\binom{10}{2} = 93$ unknowns to discover. We now prove several auxiliary claims which will serve us in finding additional restrictions that $Q$ must obey.

\begin{claim}\label{claim:cMA}
If $G$ is an $n$-vertex oriented graph and $M$ is a real $20 \times 20$ matrix, then
$$ 
\sum_{i=1}^{42} p(G_i,G) \left(c_i - \langle M, A_i \rangle - \textstyle \frac{1}{9} \right) = t(G) + i(G) - \textstyle\frac{1}{9} - \langle M, A_G \rangle.
$$
\end{claim}

\begin{proof}
The claim readily follows since
$$
t(G) + i(G) - \textstyle \frac{1}{9} = \sum_{i=1}^{42} c_i  p(G_i, G) - \textstyle \frac{1}{9} \sum_{i=1}^{42} p(G_i, G) = \sum_{i=1}^{42} p(G_i, G) \left(c_i - \textstyle \frac{1}{9}\right),
$$
by Observation~\ref{obs:objective}, and 
$$
\langle M, A_G \rangle = \langle M, \sum_{i=1}^{42}  p(G_i, G) A_i \rangle = \sum_{i=1}^{42} p(G_i, G) \langle M, A_i \rangle,
$$
by Observation~\ref{obs:matrix}.
\end{proof}

\begin{claim}\label{claim:MA}
For every $20 \times 20$ matrix $M$, there is a positive constant $C_M$ such that  
\begin{equation*}
\lvert\langle M, A_G\rangle-\langle M, \tilde{A}_G \rangle\rvert \leq \textstyle\frac{C_M}{n}
\end{equation*}
holds for every $n$-vertex oriented graph $G$.
\end{claim}

\begin{proof}
By the second part of Lemma~\ref{lem::almostPSD}, there is a positive constant $C$ such that, for every $n$-vertex oriented graph $G$, it holds that
$$
\lVert A_G - \tilde{A}_G \rVert_{\infty} \leq \textstyle\frac{C}{n},
$$
and thus
\begin{equation*}
\lvert \langle M, A_G \rangle - \langle M, \tilde{A}_G \rangle \rvert = \lvert \langle M, A_G - \tilde{A}_G \rangle \rvert \leq \lVert M \rVert_1 \lVert A_G - \tilde{A}_G \rVert_{\infty} \leq \textstyle\frac{\lVert M \rVert_1 C}{n}.
\qedhere\end{equation*}
\end{proof}

\begin{claim}\label{claim:cQA}
Let $Q$ be a $\frac{1}{9}$-certificate for the SDP~\eqref{sdp}. 
Then there exists a positive constant $C_Q$ such that  
\begin{equation*}
p(G_i,G) \left(c_i - \langle Q, A_i \rangle - \textstyle\frac{1}{9}\right) \leq i(G) + t(G) - \textstyle\frac{1}{9} + \textstyle\frac{C_Q}{n}
\end{equation*}
holds for every $n$-vertex oriented graph $G$ and every $1 \leq i \leq 42$.
\end{claim}

\begin{proof}
Let $G$ be an arbitrary $n$-vertex oriented graph. By Claim~\ref{claim:MA} there exists a positive constant $C_Q$, which depends only on the matrix $Q$, such that 
\begin{equation*}
\lvert \langle Q, A_G \rangle - \langle Q, \tilde{A}_G \rangle \rvert \leq \textstyle\frac{C_Q}{n}.
\end{equation*}
It thus follows by Lemma~\ref{lem:scalar_product} that
\begin{equation*}
\langle Q, A_G \rangle \geq \langle Q, \tilde{A}_G \rangle - \textstyle\frac{C_Q}{n} \geq - \textstyle\frac{C_Q}{n}.
\end{equation*}
Then, for every $1 \leq i \leq 42$, 
\begin{align*}
p(G_i,G) \left(c_i - \langle Q, A_i \rangle - \textstyle\frac{1}{9}\right) &\leq \sum_{j=1}^{42} p(G_j,G) \left(c_j - \langle Q, A_j \rangle - \textstyle\frac{1}{9}\right)\\
& = i(G) + t(G) - \textstyle\frac{1}{9} - \langle Q, A_G \rangle\leq i(G) + t(G) - \textstyle\frac{1}{9} + \textstyle\frac{C_Q}{n},
\end{align*}
where the first inequality holds since $Q$ is a $\frac{1}{9}$-certificate for the SDP~\eqref{sdp} and thus $c_j - \langle Q, A_j \rangle \geq \frac{1}{9}$ for every $1 \leq j \leq 42$, and the equality holds by Claim~\ref{claim:cMA}.
\end{proof}

\section{The Kernel of \texorpdfstring{$Q$}{Q}} \label{sec::kernel}

In this section we investigate the kernel of $Q_\sigma$ for $\sigma \in \Sigma$, where $\Sigma = \{\emptyset, \bar{E}, E\}$ and $Q$ is a potential $\frac{1}{9}$-certificate for the SDP~\eqref{sdp} which is a block-diagonal matrix with blocks $Q_{\emptyset}, Q_{\bar{E}}, Q_{E}$ of sizes $2 \times 2$, $9 \times 9$, and $9 \times 9$, respectively. 
It will be crucial to find all the necessary kernel vectors (i.e., those which are in the kernel of every $\frac{1}{9}$-certificate $Q$). As will be shown below, our extremal oriented graph ${\cal B}_n$ yields one kernel vector for each $Q_\sigma$.
The oriented graphs ${\cal B}_n^{\varepsilon}$ mentioned above yield two more kernel vectors for $Q_{\bar{E}}$. 

For every type $\sigma$ in $\Sigma$, let $F_{\sigma,1},\ldots,F_{\sigma,m_\sigma}$ be the flags in ${\mathcal F}_{\sigma}$ (in our case $m_{\emptyset}=2$ and $m_E=m_{\bar E}=9$), and for every $\sigma$-rooting $r$, let
\begin{equation}
    \label{eq:v_r}
v_r:=(p(F_{\sigma,1}, r),  \ldots , p(F_{\sigma,m_{\sigma}}, r))^{\rm T}. 
\end{equation}

\begin{claim}\label{claim:ker}
For every $\frac{1}{9}$-certificate $Q$ for the SDP~\eqref{sdp} and for every type $\sigma$ in $\Sigma$, there are positive constants $C_1$ and $C_2$ such that the following is true.
Let $G$ be an $n$-vertex oriented graph, let ${\mathcal R}_{\sigma}$ be the set of all rootings of $G$ over $\sigma$. Then, for every non-empty ${\mathcal R}\subseteq {\mathcal R}_{\sigma}$, it holds that
\begin{equation}\label{eq:ker_R}
\left\lVert Q_{\sigma}\left(\textstyle\frac{1}{|{\mathcal R}|}\sum_{r\in{\mathcal R}}v_r\right)\right\rVert_2 \leq C_1 \sqrt{\textstyle\frac{|{\mathcal R}_{\sigma}|}{|{\mathcal R}|}} \sqrt{t(G)+i(G)-\textstyle\frac{1}{9}+\textstyle\frac{C_2}{n}},
\end{equation}
where $C_1$ may depend only on $Q$ and $\sigma$, and $C_2$ may depend only on $Q$. In particular,
\begin{equation}\label{eq:ker_R_sigma}
\left\lVert Q_{\sigma}\left(\textstyle\frac{1}{|{\mathcal R}_{\sigma}|}\sum_{r\in{\mathcal R}_{\sigma}}v_r\right)\right\rVert_2 \leq C_1 \sqrt{t(G)+i(G)-\textstyle\frac{1}{9}+\textstyle\frac{C_2}{n}}.
\end{equation}
\end{claim}

\begin{remark} \label{rem::sanityCheck}
There is an implicit assumption in~\eqref{eq:ker_R} that $t(G) + i(G) \geq 1/9 - o(1)$, which seems odd as this is what we are striving to prove. It is thus a good time to emphasize that in this section as well as the next, we are simply proving that if a $\frac{1}{9}$-certificate exists, then it must satisfy certain properties. 
\end{remark}

\begin{proof} [Proof of Claim~\ref{claim:ker}]
Let $G$ be an $n$-vertex oriented graph. Since $Q$ is a $\frac{1}{9}$-certificate for the SDP \eqref{sdp}, and thus
$c_i - \langle Q, A_i \rangle \geq \frac{1}{9}$ for every $1 \leq i \leq 42$, it follows by Claim~\ref{claim:cMA} that 
$$
i(G) + t(G) - \textstyle\frac{1}{9} - \langle Q, A_G \rangle = \sum_{i=1}^{42} p(G_i,G) \left(c_i - \langle Q, A_i \rangle - \textstyle\frac{1}{9}\right) \geq 0.
$$
It follows by Claim~\ref{claim:MA} that there exists a positive constant $C_Q$ such that
$$
\lvert \langle Q, A_G \rangle - \langle Q, \tilde{A}_G \rangle \rvert \leq \textstyle \frac{C_Q}{n}
$$
holds for every $n$-vertex oriented graph $G$. Therefore
\begin{equation}\label{eq:tilde}
\langle Q, \tilde{A}_G \rangle \leq \langle Q, A_G \rangle + \textstyle\frac{C_Q}{n} \leq t(G) + i(G) - \textstyle\frac{1}{9} + \textstyle\frac{C_Q}{n}.
\end{equation}
Since $Q_{\sigma}$ is PSD, it can be expressed as $Q_{\sigma} = S_{\sigma}^{\rm T} S_{\sigma}$ for some matrix $S_{\sigma}$. 
Let $B_{\sigma}$ be the $|{\mathcal F}_{\sigma}| \times|{\mathcal R}_{\sigma}|$ matrix, whose $r$th column is $v_r$.
By the first part of Lemma~\ref{lem::almostPSD}, $\tilde{A}_G$ is the block matrix
$$
\tilde{A}_G=\begin{pmatrix} \frac{1}{|{\mathcal R}_{\emptyset}|}B_{\emptyset}  B_{\emptyset}^{\rm T} & 0 & 0\\
0 &  \frac{1}{|{\mathcal R}_{\bar E}|}B_{\bar E}  B_{\bar E}^{\rm T}& 0\\
0 & 0 & \frac{1}{|{\mathcal R}_E|}B_E  B_E^{\rm T}
\end{pmatrix}.
$$
Hence 
\begin{align} \label{eq:ker}
\textstyle\frac{1}{|{\mathcal R}_{\sigma}|}\sum_{r\in{\mathcal R}_{\sigma}}\lVert S_{\sigma} v_r\rVert_2^2&=\textstyle\frac{1}{|{\mathcal R}_{\sigma}|}{\rm Tr}\left((S_{\sigma}B_{\sigma})^{\rm T}S_{\sigma}B_{\sigma}\right) =\frac{1}{|{\mathcal R}_{\sigma}|}{\rm Tr}\left(B_{\sigma}^{\rm T}S_{\sigma}^{\rm T}S_{\sigma}B_{\sigma}\right) \nonumber \\
&={\rm Tr}\left(B_{\sigma}^{\rm T}Q_{\sigma}\textstyle\frac{1}{|{\mathcal R}_{\sigma}|}B_{\sigma}\right)={\rm Tr}\left(Q_{\sigma}\textstyle\frac{1}{|{\mathcal R}_{\sigma}|}B_{\sigma} B_{\sigma}^{\rm T}\right) \nonumber \\
&\leq{\rm Tr}\left(Q_{\emptyset}\textstyle\frac{1}{|{\mathcal R}_{\emptyset}|}B_{\emptyset} B_{\emptyset}^{\rm T}\right)+{\rm Tr}\left(Q_{\bar E}\textstyle\frac{1}{|{\mathcal R}_{\bar E}|}B_{\bar E} B_{\bar E}^{\rm T}\right)+{\rm Tr}\left(Q_E\textstyle\frac{1}{|{\mathcal R}_E|}B_E B_E^{\rm T}\right)\nonumber\\
&={\rm Tr}\left(Q \tilde{A}_G\right)=\langle Q,\tilde{A}_G\rangle\leq t(G)+i(G)-\textstyle\frac{1}{9}+\textstyle\frac{C_Q}{n},
\end{align}
where the fourth equality holds by the cyclic property of the trace operator, the first inequality holds by Lemma~\ref{lem:scalar_product} since both $Q_{\sigma}$ and $\textstyle\frac{1}{|{\mathcal R}_{\sigma}|} B_{\sigma} B_{\sigma}^{\rm T}$ are PSD matrices, and the last inequality holds by~\eqref{eq:tilde}. 

Therefore, for every ${\mathcal R}\subseteq {\mathcal R}_{\sigma}$, we have
\begin{align} \label{eq::Snorm}
\left\lVert S_{\sigma}\left(\textstyle\frac{1}{|{\mathcal R}|}\sum_{r\in{\mathcal R}}v_r\right)\right\rVert_2^2&=\left\lVert \textstyle\frac{1}{|{\mathcal R}|}\sum_{r\in{\mathcal R}} S_{\sigma} v_r\right\rVert_2^2\leq \left(\textstyle\frac{1}{|{\mathcal R}|}\sum_{r\in{\mathcal R}}\lVert S_{\sigma} v_r\rVert_2\right)^2\leq \textstyle\frac{1}{|{\mathcal R}|}\sum_{r\in{\mathcal R}}\lVert S_{\sigma} v_r\rVert_2^2 \nonumber \\
&\leq\textstyle\frac{1}{|{\mathcal R}|}\sum_{r\in{\mathcal R}_{\sigma}}\lVert S_{\sigma} v_r\rVert_2^2
\leq \textstyle\frac{|{\mathcal R}_{\sigma}|}{|{\mathcal R}|}\left(t(G)+i(G)-\textstyle\frac{1}{9}+\textstyle\frac{C_Q}{n}\right)
\end{align}
where the first inequality is the triangle inequality, the second inequality holds by the convexity of the function $x \mapsto x^2$, and the last inequality holds by~\eqref{eq:ker}. Hence
\begin{align*}
\left\lVert Q_{\sigma}\left(\textstyle\frac{1}{|{\mathcal R}|}\sum_{r\in{\mathcal R}}v_r\right)\right\rVert_2&=\left\lVert S_{\sigma}^{\rm T}S_{\sigma}\left(\textstyle\frac{1}{|{\mathcal R}|}\sum_{r\in{\mathcal R}}v_r\right)\right\rVert_2\leq \lVert S_{\sigma}^{\rm T}\rVert_2\,\left\lVert S_{\sigma}\left(\textstyle\frac{1}{|{\mathcal R}|}\sum_{r\in{\mathcal R}}v_r\right)\right\rVert_2\\
&\leq \lVert S_{\sigma}^{\rm T}\rVert_2\sqrt{\textstyle\frac{|{\mathcal R}_{\sigma}|}{|{\mathcal R}|}}\sqrt{t(G)+i(G)-\textstyle\frac{1}{9}+\textstyle\frac{C_Q}{n}},
\end{align*}
where the first inequality is a simple corollary of the Cauchy-Schwarts inequality and the last inequality holds by~\eqref{eq::Snorm}.
\end{proof}

\begin{lemma}\label{lem:ker} 
Let $Q$ be a $\frac{1}{9}$-certificate for the SDP~\eqref{sdp}. Then, with coordinates ordered as in Figures~\ref{fig:flags_empty}, \ref{fig:flags_edge} and~\ref{fig:flags_nonedge}, respectively, it holds that
\begin{align*}
(1,2)^{\rm T} &\in \emph{Ker}(Q_{\emptyset}),\\
(0,1,0,0,1,0,0,1,0)^{\rm T} &\in \emph{Ker}(Q_E),\\
(1,0,0,0,0,1,0,0,1)^{\rm T} &\in \emph{Ker}(Q_{\bar E}),\\
(0,1,0,0,1,0,0,1,0)^{\rm T} &\in \emph{Ker}(Q_{\bar E}),\\
(0,0,1,1,0,0,1,0,0)^{\rm T} &\in \emph{Ker}(Q_{\bar E}).
\end{align*}
\end{lemma}

\begin{proof}
Applying Claim~\ref{claim:ker} to the oriented graph ${\cal B}_{3n}$, it follows by~\eqref{eq:ker_R_sigma} and Observation~\ref{obs:blowup} that
\begin{align*}
\left\lVert Q_{\emptyset}\left(\textstyle\frac{1}{|{\mathcal R}_{\emptyset}|}\sum_{r\in{\mathcal R}_{\emptyset}}v_r\right)\right\rVert_2&\leq \textstyle\frac{C_{Q,\emptyset}}{\sqrt{n}},\\
\left\lVert Q_{E}\left(\textstyle\frac{1}{|{\mathcal R}_{E}|}\sum_{r\in{\mathcal R}_{E}}v_r\right)\right\rVert_2&\leq \textstyle\frac{C_{Q,E}}{\sqrt{n}},\\
\left\lVert Q_{\bar E}\left(\textstyle\frac{1}{|{\mathcal R}_{\bar E}|}\sum_{r\in{\mathcal R}_{\bar E}}v_r\right)\right\rVert_2&\leq \textstyle\frac{C_{Q,\bar E}}{\sqrt{n}}.
\end{align*}
Therefore, since 
\begin{subequations} \label{3eq:lim_v_r}
\begin{align} \label{3eq:lim_v_ra}
\textstyle\frac{1}{|{\mathcal R}_{\emptyset}|}\sum_{r\in{\mathcal R}_{\emptyset}}v_r=\left(\textstyle\frac{n-1}{3n-1},\textstyle\frac{2n}{3n-1}\right)^{\rm T}&\xrightarrow[n\to\infty]{}\textstyle\frac{1}{3}\left(1,2\right)^{\rm T},\\
\textstyle\frac{1}{|{\mathcal R}_{E}|}\sum_{r\in{\mathcal R}_{E}}v_r=\left(0,\textstyle\frac{n-1}{3n-2},0,0,\textstyle\frac{n-1}{3n-2},0,0,\textstyle\frac{n}{3n-2},0\right)^{\rm T}&\xrightarrow[n\to\infty]{}\textstyle\frac{1}{3}\left(0,1,0,0,1,0,0,1,0\right)^{\rm T},\\
\textstyle\frac{1}{|{\mathcal R}_{\bar E}|}\sum_{r\in{\mathcal R}_{\bar E}}v_r=\left(\textstyle\frac{n-2}{3n-2},0,0,0,0,\textstyle\frac{n}{3n-2},0,0,\textstyle\frac{n}{3n-2}\right)^{\rm T}&\xrightarrow[n\to\infty]{}\textstyle\frac{1}{3}\left(1,0,0,0,0,1,0,0,1\right)^{\rm T},
\end{align}
\end{subequations}
it follows that
\begin{equation*}
\lVert Q_{\emptyset}(1,2)^{\rm T}\rVert_2=
\lVert Q_{E}(0,1,0,0,1,0,0,1,0)^{\rm T}\rVert_2=
\lVert Q_{\bar E}(1,0,0,0,0,1,0,0,1)^{\rm T}\rVert_2=0,
\end{equation*}
and thus
\begin{align*}
(1,2)^{\rm T} &\in \text{Ker}(Q_{\emptyset}),\\
(0,1,0,0,1,0,0,1,0)^{\rm T} &\in \text{Ker}(Q_E),\\
(1,0,0,0,0,1,0,0,1)^{\rm T} &\in \text{Ker}(Q_{\bar E}).
\end{align*}

Next, apply Claim~\ref{claim:ker} to the oriented graph $G = {\mathcal B}_{3n}^{\varepsilon}$ for some arbitrary positive integer $n$ and $0 < \varepsilon < 1$.
Let ${\mathcal R} \subseteq {\mathcal R}_{\bar E}$ be the set of rootings over deleted edges which agree with their direction, that is, edges $\vec{xy} \in E({\mathcal B}_{3n}) \setminus E({\mathcal B}_{3n}^{\varepsilon})$, where $x$ is labelled 1 and $y$ is labelled 2. 
For every three distinct vertices $x,y,z\in V$ and every $F\in{\mathcal F}_{\bar E}$, let ${\mathcal A}^F_{x,y,z}$ be the event that the subgraph of $E({\mathcal B}_{3n}^{\varepsilon})$ induced on the vertices $x,y,z$, where $x$ is labelled 1 and $y$ is labelled 2, is isomorphic to $F$ (in particular, the edge $\vec{xy}$ was deleted).  
Note that for every $F\in{\mathcal F}_{\bar E}$,
\begin{align*}
\mathbb{E} &\left(\sum_{r\in{\mathcal R}}p(F,r)\right)\\
&=\sum_{\substack{0\leq i\leq 2,\\x\in V_i,\, y\in V_{i+1}}}\textstyle\frac{1}{3n-2}\left(\sum_{z\in V_{i-1}}\Pr({\mathcal A}^F_{x,y,z})+\sum_{x\neq z\in V_i}\Pr({\mathcal A}^F_{x,y,z})+\sum_{y\neq z\in V_{i+1}}\Pr({\mathcal A}^F_{x,y,z})\right),
\end{align*}
where the expectation and probabilities are taken with respect to the random deletion of edges which results in ${\mathcal B}_{3n}^{\varepsilon}$.
It follows that
\begin{align*}
\textstyle\frac{1}{3n^2} \mathbb{E} \left(\sum_{r\in{\mathcal R}} v_r \right) =& \textstyle\frac{n}{3n-2} \left(\varepsilon^3, 0, \varepsilon^2 (1-\varepsilon), \varepsilon^2 (1-\varepsilon), 0, 0, 0, \varepsilon(1-\varepsilon)^2, 0\right)^{\rm T}\\
&+ \textstyle\frac{n-1}{3n-2} \left(2\varepsilon^2, \varepsilon (1-\varepsilon), 0, 0, \varepsilon(1-\varepsilon), 0, 0, 0, 0\right)^{\rm T}.
\end{align*}
Therefore
\begin{equation}\label{eq:lim_epsilon}
\textstyle\frac{1}{n^2\varepsilon} \mathbb{E} \left(\sum_{r\in{\mathcal R}} v_r \right) \xrightarrow[n\to\infty]{}\left(\varepsilon^2+2\varepsilon,1-\varepsilon,\varepsilon(1-\varepsilon),\varepsilon(1-\varepsilon),1-\varepsilon,0,0,(1-\varepsilon)^2,0\right)^{\rm T}.
\end{equation}
Now, by~\eqref{eq:ker_R} we have
\begin{align*} 
\left\lVert Q_{\bar E}\left(\sum_{r \in{\mathcal R}} v_r\right)\right\rVert_2 &\leq C_1 \sqrt{|{\mathcal R}_{\bar E}|} \sqrt{|{\mathcal R}|} \sqrt{t({\mathcal B}_{3n}^\varepsilon) + i({\mathcal B}_{3n}^\varepsilon) - \textstyle\frac{1}{9} + \textstyle\frac{C_2}{n}} \\
&= C_1 \sqrt{|{\mathcal R}_{\bar E}|} \sqrt{|{\mathcal R}|} \sqrt{i({\mathcal B}_{3n}^\varepsilon) - \textstyle\frac{1}{9} + \textstyle\frac{C_2}{n}} \leq 3 C_1 n \sqrt{|{\mathcal R}|} \sqrt{i({\mathcal B}_{3n}^\varepsilon) - \textstyle\frac{1}{9} + \textstyle\frac{C_2}{n}}, 
\end{align*}
where $C_1$ may depend only on $Q$ and ${\bar E}$, and $C_2$ may depend only on $Q$. Hence 
\begin{align*}
\left\lVert Q_{\bar E} \left(\textstyle\frac{1}{n^2\varepsilon} \mathbb{E} \left(\sum_{r \in {\mathcal R}} v_r \right)\right)\right\rVert_2 &= \textstyle\frac{1}{n^2\varepsilon} \left\lVert\mathbb{E} \left(Q_{\bar E} \left(\sum_{r \in {\mathcal R}} v_r\right)\right)\right\rVert_2 \leq \textstyle\frac{1}{n^2\varepsilon} \mathbb{E}\left\lVert Q_{\bar E} \left(\sum_{r \in {\mathcal R}} v_r\right)\right\rVert_2\\
&\leq \textstyle\frac{3 C_1}{n\varepsilon} \mathbb{E}\left(\sqrt{|{\mathcal R}|} \sqrt{i({\mathcal B}_{3n}^\varepsilon) - \textstyle\frac{1}{9} + \textstyle\frac{C_2}{n}}\right) \leq \textstyle\frac{3 C_1}{n \varepsilon} \sqrt{\mathbb{E}|{\mathcal R}|} \sqrt{\mathbb{E}\left(i({\mathcal B}_{3n}^\varepsilon) - \textstyle\frac{1}{9} + \textstyle\frac{C_2}{n}\right)}\\
 &= \textstyle\frac{3 C_1}{n \varepsilon} \sqrt{3n^2 \varepsilon} \sqrt{\textstyle\frac{1}{\binom{3n}{3}} \left(3\binom{n}{3} + 3 \binom{n}{2} 2n \varepsilon^2 + \varepsilon^3 n^3\right) - \textstyle\frac{1}{9} + \textstyle\frac{C_2}{n}} \\ &\xrightarrow[n\to \infty]{} 3 \sqrt{2} C_1 \sqrt{\varepsilon} \sqrt{1 + \varepsilon/3},
\end{align*}
where the first inequality holds by Jensen's inequality and the convexity of the Euclidean norm and the third inequality holds by the Cauchy-Schwartz inequality. 
Therefore, for every $0 < \varepsilon < 1$, it follows by~\eqref{eq:lim_epsilon} that
\begin{multline*}
    \lVert Q_{\bar E}\left(\varepsilon^2+2\varepsilon,1-\varepsilon,\varepsilon(1-\varepsilon),\varepsilon(1-\varepsilon),1-\varepsilon,0,0,(1-\varepsilon)^2,0\right)^{\rm T}\rVert_2 \\
    =\lim_{n\to\infty}\lVert Q_{\bar E}\left(\textstyle\frac{1}{n^2\varepsilon} \mathbb{E} \left(\sum_{r\in{\mathcal R}} v_r \right)\right)\rVert_2\leq 3 \sqrt{2} C_1 \sqrt{\varepsilon} \sqrt{1 + \varepsilon/3}.
\end{multline*}
We conclude that
$$
\lVert Q_{\bar E}(0,1,0,0,1,0,0,1,0)^{\rm T}\rVert_2 = 0
$$
and thus
$$
(0,1,0,0,1,0,0,1,0)^{\rm T} \in \text{Ker}(Q_{\bar E}).
$$
An analogous argument, this time considering all ${\bar E}$-rootings over deleted edges in the opposite direction shows that 
\begin{equation*}
(0,0,1,1,0,0,1,0,0)^{\rm T} \in \text{Ker}(Q_{\bar E}).
\qedhere\end{equation*}
\end{proof}

\begin{remark}\label{rem:big:picture}
The, widely used, method by which we found the first three kernel vectors, and to some extent also the other two, is general to any flag algebra application. The practical flag algebra guideline is to check all the near zero eigenvalues of an approximate computer generated certificate. Before trying to round it, one verifies that all these eigenvalues match the expected eigenvalues from known extremal constructions. Any unexplained near zero eigenvalue may hint at the existence of other extremal constructions -- either a completely different graph, or a variation on an existing one, as is the case here. Once we have all the needed extremal constructions, we can accomodate for all the sharp graph equations (see the following section). Thus, it is not enough to simply force eigenvectors corresponding to near-zero eigenvalues to be in the kernel; one must find the constructions that explain them.
\end{remark}

\section{Sharp graphs} \label{sec::sharp}
Let
\begin{align*}
{\mathcal W}_{\,\emptyset} = \{M_{\,\emptyset} \in M_{2\times 2}({\mathbb R}) :&  M_{\emptyset}^{\rm T} = M_{\emptyset},\,(1,2)^{\rm T} \in \text{Ker}(M_{\emptyset})\},\\
{\mathcal W}_E = \{M_E \in M_{9 \times 9}({\mathbb R}) :& M_E^{\rm T} = M_E,\,(0,1,0,0,1,0,0,1,0)^{\rm T} \in \text{Ker}(M_E)\},\\
{\mathcal W}_{\bar{E}} = \{M_{\bar{E}} \in M_{9 \times 9}({\mathbb R}) :&  M_{\bar{E}}^{\rm T} = M_{\bar{E}},\,(1,0,0,0,0,1,0,0,1)^{\rm T} \in \text{Ker}(M_{\bar{E}}),\\
& (0,1,0,0,1,0,0,1,0)^{\rm T} \in \text{Ker}(M_{\bar{E}}),\,(0,0,1,1,0,0,1,0,0)^{\rm T} \in \text{Ker}(M_{\bar{E}})\}.
\end{align*}
Lemma~\ref{lem:ker} may be rephrased in the following way. Every $\frac{1}{9}$-certificate of the SDP~\eqref{sdp} which is a block-diagonal matrix with blocks of sizes $2 \times 2$, $9 \times 9$, and $9 \times 9$ is a member of the linear subspace
$$
{\mathcal W} := \left\{
\begin{pmatrix}
M_{\emptyset} & 0 & 0\\
0 & M_{\bar{E}} & 0\\
0 & 0 & M_E
\end{pmatrix}
: M_{\emptyset}\in {\mathcal W}_{\emptyset}, \, M_{\bar{E}}\in {\mathcal W}_{\bar{E}}, \, M_{E}\in {\mathcal W}_{E}\right\}
$$
of the space of $20 \times 20$ symmetric matrices. In this section we will find an affine subspace of ${\mathcal W}$, of smaller dimension, which still  contains all $\frac{1}{9}$-certificates of the SDP~\eqref{sdp} with the aforementioned block structure.

For every $1 \leq i \leq 42$, we say that the $4$-vertex oriented graph $G_i$ is {\em sharp} if $\mathbb{E}(p(G_i, {\mathcal B}_n^{\varepsilon}) ) = \Omega(\varepsilon)$ as $\varepsilon \to 0^+$, where the expectation is taken with respect to the random deletion of edges which results in ${\mathcal B}_n^{\varepsilon}$. 
It is not hard to check that there are eleven sharp graphs. Five of which, namely, $G_{1}$, $G_{7}$, $G_{10}$, $G_{27}$, and $G_{32}$ are induced subgraphs of ${\mathcal B}_n$, and six additional graphs, namely, $G_{3}$, $G_{5}$, $G_{15}$, $G_{19}$, $G_{23}$, and $G_{25}$ appear abundantly as induced subgraphs of ${\mathcal B}_n^{\varepsilon}$. 

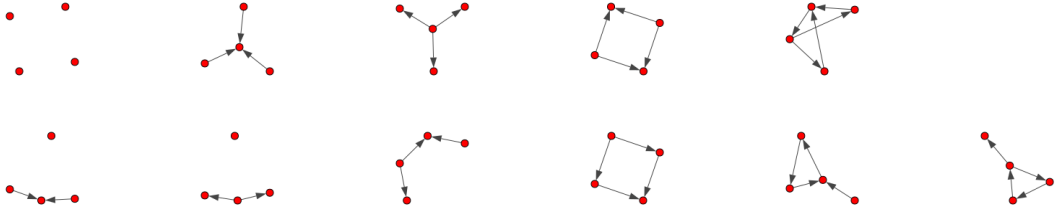
\begin{figure}
\centering
\begin{tikzpicture}

\Vertex[x=-4.75, y=9.6, size=0.2]{A11};
\Vertex[x=-4.75, y=8.9, size=0.2]{B11};
\Vertex[x=-5.35, y=8.6, size=0.2]{C11};
\Vertex[x=-4.15,y=8.6, size=0.2]{D11};
\Text[x=-4.75, y=8.1]{$G_{1}$}

\Vertex[x=-2.25, y=9.6, size=0.2]{A21};
\Vertex[x=-2.25, y=8.9, size=0.2]{B21};
\Vertex[x=-2.85, y=8.6, size=0.2]{C21};
\Vertex[x=-1.65, y=8.6, size=0.2]{D21};
\Text[x=-2.25, y=8.1]{$G_{7}$}
\Edge[Direct](A21)(B21);
\Edge[Direct](C21)(B21);
\Edge[Direct](D21)(B21);

\Vertex[x=0.25, y=9.6, size=0.2]{A24};
\Vertex[x=0.25, y=8.9, size=0.2]{B24};
\Vertex[x=-0.35, y=8.6, size=0.2]{C24};
\Vertex[x=0.85, y=8.6, size=0.2]{D24};
\Text[x=0.25,  y=8.1]{$G_{10}$}
\Edge[Direct](B24)(A24);
\Edge[Direct](B24)(C24);
\Edge[Direct](B24)(D24);

\Vertex[x=2.75, y=9.6, size=0.2]{A53};
\Vertex[x=2.75, y=8.9, size=0.2]{B53};
\Vertex[x=2.15, y=8.6, size=0.2]{C53};
\Vertex[x=3.35, y=8.6, size=0.2]{D53};
\Text[x=2.75, y=8.1]{$G_{27}$}
\Edge[Direct](C53)(A53);
\Edge[Direct](C53)(B53);
\Edge[Direct](D53)(B53);
\Edge[Direct](D53)(A53);

\Vertex[x=5.25, y=9.6, size=0.2]{A62};
\Vertex[x=5.25, y=8.9, size=0.2]{B62};
\Vertex[x=4.65, y=8.6, size=0.2]{C62};
\Vertex[x=5.85, y=8.6, size=0.2]{D62};
\Text[x=5.25, y=8.1]{$G_{32}$}
\Edge[Direct](A62)(C62);
\Edge[Direct](A62)(D62);
\Edge[Direct](B62)(A62);
\Edge[Direct](C62)(B62);
\Edge[Direct](D62)(B62);

\Vertex[x=-6,y=6.9, size=0.2]{A13};
\Vertex[x=-6,y=6.2, size=0.2]{B13};
\Vertex[x=-6.6,y=5.9, size=0.2]{C13};
\Vertex[x=-5.4,y=5.9, size=0.2]{D13};
\Text[x=-6,y=5.4]{$G_{3}$}
\Edge[Direct](A13)(D13);
\Edge[Direct](C13)(D13);

\Vertex[x=-3.5, y=6.9, size=0.2]{A15};
\Vertex[x=-3.5, y=6.2, size=0.2]{B15};
\Vertex[x=-4.1, y=5.9, size=0.2]{C15};
\Vertex[x=-2.9, y=5.9, size=0.2]{D15};
\Text[x=-3.5, y=5.4]{$G_{5}$}
\Edge[Direct](C15)(A15);
\Edge[Direct](C15)(D15);

\Vertex[x=-1,y=6.9, size=0.2]{A33};
\Vertex[x=-1,y=6.2, size=0.2]{B33};
\Vertex[x=-1.6,y=5.9, size=0.2]{C33};
\Vertex[x=-0.4,y=5.9, size=0.2]{D33};
\Text[x=-1,y=5.4]{$G_{15}$}
\Edge[Direct](A33)(D33);
\Edge[Direct](C33)(D33);
\Edge[Direct](A33)(B33);

\Vertex[x=1.5, y=6.9, size=0.2]{A41};
\Vertex[x=1.5, y=6.2, size=0.2]{B41};
\Vertex[x=0.9, y=5.9, size=0.2]{C41};
\Vertex[x=2.1, y=5.9, size=0.2]{D41};
\Text[x=1.5, y=5.4]{$G_{19}$}
\Edge[Direct](A41)(B41);
\Edge[Direct](C41)(B41);
\Edge[Direct](B41)(D41);
\Edge[Direct](D41)(A41);

\Vertex[x=4,y=6.9, size=0.2]{A45};
\Vertex[x=4,y=6.2, size=0.2]{B45};
\Vertex[x=3.4,y=5.9, size=0.2]{C45};
\Vertex[x=4.6,y=5.9, size=0.2]{D45};
\Text[x=4,y=5.4]{$G_{23}$}
\Edge[Direct](A45)(B45);
\Edge[Direct](B45)(C45);
\Edge[Direct](B45)(D45);
\Edge[Direct](C45)(A45);

\Vertex[x=6.5, y=6.9, size=0.2]{A51};
\Vertex[x=6.5, y=6.2, size=0.2]{B51};
\Vertex[x=5.9, y=5.9, size=0.2]{C51};
\Vertex[x=7.1, y=5.9, size=0.2]{D51};
\Text[x=6.5, y=5.4]{$G_{25}$}
\Edge[Direct](C51)(A51);
\Edge[Direct](B51)(C51);
\Edge[Direct](B51)(D51);
\Edge[Direct](D51)(A51);

\end{tikzpicture}
\caption{The sharp graphs. The graphs in the first row are induced subgraphs of $\mathcal{B}_n$. The density in $\mathcal{B}_n^{\varepsilon}$ of the graphs in the second row is linear in $\varepsilon$.}
\end{figure}

Our interest in sharp graphs is due to the following lemma which asserts that every sharp graph imposes a linear equation which the entries of any $\frac{1}{9}$-certificate must satisfy.
\begin{lemma}
Let $Q$ be a $\frac{1}{9}$-certificate for the SDP~\eqref{sdp}. If $G_i$ is sharp, then 
$$
c_i - \langle Q, A_i \rangle =\textstyle\frac{1}{9}.
$$
\end{lemma}

\begin{proof}
For every $1 \leq i \leq 42$, every positive integer $n$, and every $\varepsilon > 0$, it follows by Claim~\ref{claim:cQA} that
$$
\left(c_i - \langle Q, A_{G_i} \rangle-\textstyle\frac{1}{9}\right) p(G_i,{\mathcal B}_n^{\varepsilon})\leq t({\mathcal B}_n^{\varepsilon})+i({\mathcal B}_n^{\varepsilon})-\textstyle\frac{1}{9}+O\left(\frac{1}{n}\right)=i({\mathcal B}_n^{\varepsilon})-\textstyle\frac{1}{9}+O\left(\textstyle\frac{1}{n}\right).
$$
Hence
\begin{align*}
    \left(c_i - \langle Q, A_{G_i} \rangle - \textstyle\frac{1}{9}\right) \mathbb{E}(p(G_i,{\mathcal B}_n^{\varepsilon})) &\leq \mathbb{E}(i({\mathcal B}_n^{\varepsilon})) - \textstyle\frac{1}{9} + O\left(\textstyle\frac{1}{n}\right)\\
    &= i({\mathcal B}_n) + O(\varepsilon^2) - \textstyle\frac{1}{9} + O\left(\textstyle\frac{1}{n}\right) \leq
O(\varepsilon^2) + O\left(\textstyle\frac{1}{n}\right),
\end{align*}
where the last inequality holds by Observation~\ref{obs:blowup}. Therefore, if $G_i$ is sharp, then for every positive integer $n$ and every $\varepsilon > 0$, it holds that
$$
0 \leq c_i - \langle Q, A_{G_i} \rangle - \textstyle\frac{1}{9} \leq O(\varepsilon) + O\left(\textstyle\frac{1}{n\varepsilon}\right)
$$
and thus
\begin{equation*}
c_i - \langle Q, A_i \rangle - \textstyle\frac{1}{9} = 0.
\qedhere\end{equation*}
\end{proof}

Therefore, in addition to the linear constraints that were already found in the previous section, we have found $11$ linear constraints that every $\frac{1}{9}$-certificate of the SDP~\eqref{sdp} must satisfy. Namely, let $I_{\text{sharp}} = \{1,3,5,7,10,15,19,23,25,27,32\}$. Then every $\frac{1}{9}$-certificate of the SDP~\eqref{sdp} which is a block-diagonal matrix with blocks of sizes $2 \times 2$, $9 \times 9$, and $9 \times 9$ is a member of the affine subspace
$$
\tilde{\mathcal W} := \left\{M \in {\mathcal W} : c_i - \langle A_i, M \rangle = \textstyle\frac{1}{9} \textrm{ for every } i \in I_{\text{sharp}}\right\}
$$
of the linear space ${\mathcal W}$. We note that $\dim{\mathcal W} - \dim\tilde{\mathcal W}$ is not $11$, as one might hope, but rather smaller, as is stated in the following lemma.
\begin{lemma}
Let $I_{\text{induced}} = \{1,7,10,27,32\}$. Then, for every $M \in {\mathcal W}$, it holds that
$$
\sum_{i \in I_{\text{induced}}} \lambda_i \left(c_i - \langle A_i, M \rangle - \textstyle\frac{1}{9}\right) = 0,
$$
where $\lambda_i = \lim_{n \to \infty} p(G_i, {\mathcal B}_{3n})$ for every $i \in I_{\text{induced}}$.
\end{lemma}

\begin{proof}
Let
$$
M=\begin{pmatrix}
M_{\emptyset} & 0 & 0\\
0 & M_{\bar{E}} & 0\\
0 & 0 & M_E
\end{pmatrix}
$$
be a matrix in ${\mathcal W}$ and let $n$ be a positive integer. Then
\begin{align} \label{eq:tildeAM}
\sum_{i \in I_{\text{induced}}} p(G_i,{\mathcal B}_{3n})\left(c_i-\langle A_i,M\rangle-\textstyle\frac{1}{9}\right)&=\sum_{i=1}^{42}p(G_i,{\mathcal B}_{3n}) \left(c_i-\langle A_i,M\rangle-\textstyle\frac{1}{9}\right)\nonumber\\
&=i({\mathcal B}_{3n})+t({\mathcal B}_{3n})-\langle A_{{\mathcal B}_{3n}}, M\rangle-\textstyle\frac{1}{9}\nonumber\\
&=\left(\textstyle\frac{1}{9} + O\left(\textstyle\frac{1}{n}\right)\right) - \left(\langle \tilde{A}_{{\mathcal B}_{3n}}, M \rangle + O\left(\textstyle\frac{1}{n}\right)\right) - \textstyle\frac{1}{9} \nonumber\\
&= -\langle \tilde{A}_{{\mathcal B}_{3n}}, M\rangle + O\left(\textstyle\frac{1}{n}\right),
\end{align}
where the second equality holds by Claim~\ref{claim:cMA} and the third equality holds by Claim~\ref{claim:MA}.

Fix an arbitrary type $\sigma \in \Sigma$. Let 
$$
v_{\sigma,3n} = \frac{1}{|{\mathcal R}_{\sigma}|} \sum_{r \in {\mathcal R}_{\sigma}} v_r
$$
where ${\mathcal R}_{\sigma}$ is the set of all rootings of ${\mathcal B}_{3n}$ over $\sigma$ and $v_r$ is as in \eqref{eq:v_r}. 
Note that $p(F,r_1) = p(F,r_2)$ holds for every flag $F$ in ${\mathcal F}_{\sigma}$ and every two rootings $r_1$ and $r_2$ in ${\mathcal R}_{\sigma}$. Therefore, it follows by the first part of Lemma~\ref{lem::almostPSD} that
$$
\tilde{A}_{{\mathcal B}_{3n}} = 
\begin{pmatrix} 
v_{\emptyset,3n} {v_{\emptyset,3n}}^{\rm T} & 0 & 0\\
0 & v_{\bar{E},3n} {v_{\bar{E},3n}}^{\rm T} & 0\\
0 & 0 & v_{E,3n} {v_{E,3n}}^{\rm T}
\end{pmatrix},
$$
Hence
\begin{align} \label{eq::dotProductZero}
\langle \tilde{A}_{{\mathcal B}_{3n}}, M \rangle &= {\rm Tr}\left( \tilde{A}_{{\mathcal B}_{3n}} M\right) = \sum_{\sigma \in \Sigma} {\rm Tr} \left(v_{\sigma,3n} {v_{\sigma,3n}}^{\rm T} M_{\sigma}\right) \nonumber \\
&=\sum_{\sigma \in \Sigma}{\rm Tr} \left({v_{\sigma,3n}}^{\rm T} M_{\sigma} v_{\sigma,3n}\right) = \sum_{\sigma \in \Sigma} {v_{\sigma,3n}}^{\rm T} M_{\sigma} v_{\sigma,3n} \xrightarrow[n\to \infty]{} 0,
\end{align}
where the third equality holds by the cyclic property of the trace operator, and the last sum converges to zero as, by \eqref{3eq:lim_v_r}, the vectors $v_{\sigma,3n}$ approach kernel vectors of $M_{\sigma}$ as $n$ tends to infinity. 
Hence
\begin{align*}
\sum_{i \in I_{\text{induced}}} \lambda_i \left(c_i - \langle A_i, M \rangle - \textstyle\frac{1}{9}\right) &= \lim_{n \to \infty} \sum_{i \in I_{\text{induced}}} p(G_i,{\mathcal B}_{3n}) \left(c_i - \langle A_i, M \rangle - \textstyle\frac{1}{9}\right)\\
&= \lim_{n \to \infty} \left(- \langle \tilde{A}_{{\mathcal B}_{3n}}, M \rangle + O\left(\textstyle\frac{1}{n}\right)\right) = 0,
\end{align*}
where the second equality holds by~\eqref{eq:tildeAM} and the third equality holds by~\eqref{eq::dotProductZero}.
\end{proof}

A similar, but somewhat more involved, argument shows that for every $M \in {\mathcal W}$, it also holds that
$$
\sum_{i \in I_{\text{sharp}} \setminus I_{\text{induced}}} \lambda_i \left(c_i - \langle A_i, M \rangle - \textstyle\frac{1}{9}\right) = 0,
$$
where $\lambda_i = \lim_{\varepsilon \to 0^+} \frac{1}{\varepsilon} \lim_{n \to \infty} \mathbb{E}\left(p(G_i,{\mathcal B}_{3n}^{\varepsilon})\right)$ for every $i \in I_{\text{sharp}} \setminus I_{\text{induced}}$.

The above calculations suggest that perhaps $\dim\tilde{\mathcal W} = \dim{\mathcal W} - 9$. Straightforward computer-aided calculations reveal that this is indeed the case.

\section{Projection} \label{sec::projection}

Recall that our general plan is to use computer software to find a $\left(\frac{1}{9}-\delta\right)$-certificate for some small $\delta > 0$ and then round it to a $\frac{1}{9}$-certificate. By now, we are aware of two conditions that any $\frac{1}{9}$-certificate must satisfy, namely, its kernel must include the five vectors listed in Lemma~\ref{lem:ker}, and it must satisfy the eleven sharp graphs equations. In this section, we use the first of these two conditions to reduce the order of the certificate matrix we seek. This is done via a projection to the orthogonal complement of the linear space spanned by the five kernel vectors from Lemma~\ref{lem:ker}. 

The projection will reduce the order of the matrices $A_i$ from $20$ to $15$. In fact, the main benefit of this projection is that it will allow us to find a {\em strictly} positive definite certificate for the projected problem; such a matrix may be slightly perturbed without the risk of generating negative eigenvalues. 

For every $\sigma \in \Sigma$, let $R_\sigma$ be a matrix whose columns form an orthonormal basis of the space perpendicular to the kernel vectors of $Q_\sigma$ that we found in Section~\ref{sec::kernel}. In particular, $R_\sigma^{\rm T} R_\sigma$ is an identity matrix. Observe that $R_{\emptyset}$ is a $2 \times 1$ matrix, $R_{\bar E}$ is a $9 \times 6$ matrix and $R_E$ is a $9 \times 8$ matrix. Let
\begin{equation*}
R = \begin{pmatrix}R_{\emptyset}& 0_{2\times 6}&0_{2\times 8}\\0_{9\times 1}&R_{\bar E}&0_{9\times 8}\\0_{9\times 1}&0_{9\times 6}&R_E
\end{pmatrix}
\end{equation*}
be a $20 \times 15$ block matrix, where $0_{k \times \ell}$ denotes the $k \times \ell$ all-zeros matrix. For every $1\leq i\leq 42$, let $\bar{A}_i = R^{\rm T} A_i R$.

\begin{lemma} \label{lem::projectedSDP}
Suppose $\bar{Q}$ is an $\alpha$-certificate for the projected SDP:
\begin{eqnarray}\label{proj_sdp} 
\begin{aligned}
&\text{Variables: } p_1,\ldots,p_{42} \\
&\text{Goal: } minimize \sum_{i=1}^{42} p_i c_i \\
&\text{Constraints: } \\
& p_1,\ldots,p_{42} \geq 0 \\
& \sum_{i=1}^{42} p_i = 1 \\
& \sum_{i=1}^{42} p_i \bar{A}_i \succeq 0 
\end{aligned} 
\end{eqnarray}
Then $Q := R \bar{Q} R^T$ is an $\alpha$-certificate for the SDP~\eqref{sdp}. Moreover, if $\alpha = \frac{1}{9}$, then $c_i - \langle \bar{Q}, \bar{A_i} \rangle = \frac{1}{9}$ whenever $G_i$ is a sharp graph.
\end{lemma}

\begin{proof} 
Both claims follow by observing that, for every $1 \leq i \leq 42$, it holds that
\begin{align*}
\langle \bar{Q}, \bar{A}_i\rangle &= {\rm Tr}(\bar{Q}\bar{A_i}) = {\rm Tr}(\bar{Q} (R^T A_i R)) = {\rm Tr}((\bar{Q} R^T A_i) R)\\
&= {\rm Tr}(R(\bar{Q} R^T A_i)) = {\rm Tr}((R\bar{Q} R^T) A_i) = {\rm Tr}(Q A_i) = \langle Q, A_i \rangle.
\qedhere
\end{align*}
\end{proof}

With Lemma~\ref{lem::projectedSDP} in mind, we now turn to seek a $\frac {1}{9}$-certificate $\bar{Q}$ for the projected SDP~\eqref{proj_sdp}. Note that $\bar{Q}$ will be a symmetric block-diagonal matrix with blocks of sizes $1 \times 1$, $6 \times 6$, and $8 \times 8$.

\section{Finding \texorpdfstring{$\bar{Q}$}{\bar{Q}} by rounding an approximate solution} \label{sec:rounding}
Using flagmatic to compute $A_i$ and then the SDP solver~\cite{csdp}, we have found an approximate solution to the SDP problem~\eqref{proj_sdp}. 
The solver yields a $\left(\frac{1}{9} - 10^{-8}\right)$-certificate $\tilde{Q}$ for the SDP~\eqref{proj_sdp} (in particular, verifying that the optimal solution is indeed very close to $1/9$). 

Since $\bar{Q}$ is symmetric and has the block structure described above, there are $\binom{2}{2} + \binom{7}{2} + \binom{9}{2} = 58$ degrees of freedom for the entries of $\bar{Q}$. As noted in Section~\ref{sec::sharp}, the $11$ sharp graphs equations impose $9$ additional independent restrictions on the entries of $\bar{Q}$. This leaves us with $49$ degrees of freedom left. 

We chose $49$ coordinates to equal the corresponding coordinates of $\tilde{Q}$. Then we calculated the remaining coordinates which are uniquely determined by the sharp graph equations. We ordered the coordinates lexicographically and chose the values one by one from the computer generated certificate, as long as the sharp graph equations were not violated. Otherwise, we chose the only value that would allow for sharp graph equations to be satisfied. 

A word about computational precision is in order at this point. 
The entries that we set to equal the ones in $\tilde {Q}$ are taken with $4$ digits after the floating point -- this is the part where things are not precise, and we were lucky enough that the method worked and a reasonable number of digits sufficed; an important boost to this luck comes from the projection we performed in Section~\ref{sec::projection} by which, assuming we found all the kernel vectors in Section~\ref{sec::kernel}, it is plausible that $\tilde {Q}$ will have no near-zero eigenvalues. They are then presented as rationals whose denominator divides $10^4$. The remaining $9$ entries are then uniquely determined by the sharp graph equations, but they are not necessarily rational. Yet, they reside in a finite extension of the rationals (namely $\mathbb{Q}[\sqrt{2},\sqrt{3}]$), which allows the computations at that point to be infinitely precise. 
We ran this calculation in the Mathematica software which resulted in the matrix $\bar{Q}$ whose blocks appear below. 
\begin{align*}
& \bar{Q}_{\emptyset}=\frac1{10000}
\begin{pmatrix}
337
\end{pmatrix}, \\
& \bar{Q}_{\bar{E}}=\frac{1}{150000}
\begin{pmatrix}
193934 & 705 & 705 & 1230 & 1230 & 0 \\ 
705 & 257730 & -34095 & -45285 & -75735 & 80205 \\
705 & -34095 & 257730 & -75735 & -45285 & 80205 \\
1230 & -45285 & -75735 & 170280 & -86385 & -46305 \\ 
1230 & -75735 & -45285 & -86385 & 170280 & -46305 \\ 
0 & 80205 & 80205 & -46305 & -46305 & 153796 + 6480 \sqrt{3}
\end{pmatrix}, \\
& \bar{Q}_E=\frac{1}{450000}\left(M_1+P^{\rm T}\left(\sqrt{2}M_2+\sqrt{3}M_3+\sqrt{6}M_6\right)P\right),
\end{align*}
where
$$
M_1=
\begin{pmatrix}
527985 & 0 & -315450 & -315450 & 0 & -430920 & -375705 & -430920 \\
0 & 993198 & -268740 & 150840 & -29160 & 67680 & -27090 & -186480 \\
-315450 & -268740 & 536490 & -42030 & 0 & 233550 & 168435 & 220815 \\
-315450 & 150840 & -42030 & 536490 & 0 & 220815 & 168435 & 233550 \\
0 & -29160 & 0 & 0 & 663612 & -176265 & -46935 & -29475 \\ 
-430920 & 67680 & 233550 & 220815 & -176265 & 638010 & 313920 & 281700 \\
-375705 & -27090 & 168435 & 168435 & -46935 & 313920 & 542430 & 313920 \\
-430920 & -186480 & 220815 & 233550 & -29475 & 281700 & 313920 & 638010
\end{pmatrix},
$$
\begin{align*}
P=&\begin{pmatrix}
1 & 0 & 0 & 0 & 0 & 0 & 0 & 0 \\
0 & 1 & 0 & 0 & 0 & 0 & 0 & 0 \\
0 & 0 & 1 & 0 & 0 & 0 & 0 & 0 \\
0 & 0 & 0 & 1 & 0 & 0 & 0 & 0 \\
0 & 0 & 0 & 0 & 1 & 0 & 0 & 0 
\end{pmatrix},
\\
M_2=&\begin{pmatrix}
0 & -3690 & 0 & 0 & 209271 \\
-3690 & 0 & 0 & 0 & 0  \\
0 & 0 & 0 & 0 & -93902  \\
0 & 0 & 0 & 0 & -586954  \\
209271 & 0 & -93902 & -586954 & 0  
\end{pmatrix},
\\
M_3=&\begin{pmatrix}
0 & 0 & 0 & 0 & 0  \\
0 & 0 & 0 & 0 & -164793  \\ 
0 & 0 & 0 & 0 & 190140  \\
0 & 0 & 0 & 0 & 229440  \\
0 & -164793 & 190140 & 229440 & -19440  
\end{pmatrix},
\\
M_6=&\begin{pmatrix}
0 & 27442 & 0 & 0 & -76965 \\
27442 & 0 & 0 & 0 & 0  \\
0 & 0 & 0 & 0 & -72495  \\
0  & 0 & 0 & 0 & 85455  \\
-76965 & 0 & -72495 & 85455 & 0 
\end{pmatrix}.
\end{align*}

We have verified by computer software that the matrix $\bar{Q}$ is PD (positive definite) and that it satisfies $c_i - \langle \bar{Q}, \bar{A}_i\rangle \geq \frac{1}{9}$ for every $1 \leq i \leq 42$. Verifying that a matrix is PD can be done by calculating its leading principal minors.

\begin{remark}
Since $\bar{Q}$ is PD, it follows that the kernel of $Q$ is spanned by the five vectors that were listed in Lemma~\ref{lem:ker}. This demonstrates that we have indeed found all the necessary kernel vectors. 
\end{remark}

\begin{remark}
 If one only wishes to have a formal proof of Theorem~\ref{main}, one could just present $\bar{Q}$ (or the pulled back $Q$) and show that it is indeed a $\frac{1}{9}$-certificate for the corresponding SDP. This is common practice in many flag algebra applications, where the certificate $Q$ is presented without bothering to explain all the details of how it was found. 
\end{remark}

\section{Stability}\label{sec:stability}
In this section we prove Theorem~\ref{thm:stabil}. It will be obtained by combining several results. Our proof is quite long, partly because we wish to obtain very specific error terms which will serve us when proving Theorem~\ref{thm:exact} in the next section.  
\begin{lemma}\label{lem:first_step}
There is a positive constant $C_1$ such that for every $\delta > 0$ there is a positive integer $n_1(\delta)$ for which the following statement is true. If $G$ is an $n$-vertex oriented graph such that $n \geq n_1(\delta)$ and $t(G)+i(G) \leq \frac{1}{9} + \delta$, and $G^{(0)}$ is the underlying undirected graph of $G$, then $p(K_4, G^{(0)}) \leq C_1 \delta$.
\end{lemma}

\begin{proof}
Let $Q$ be the $\frac{1}{9}$-certificate that we found for the SDP~\eqref{sdp}. A straightforward albeit tedious calculation (which can be performed by computer software) shows that  
\begin{equation*}
\eta_i := c_i - \langle Q, A_{G_i} \rangle - \textstyle\frac{1}{9} > 0
\end{equation*}
holds for every $39 \leq i \leq 42$. By Claim~\ref{claim:cQA}, there is a positive integer $n_1(\delta)$ such that for every $n$-vertex oriented graph $G$ with $n \geq n_1(\delta)$, and every $1 \leq i \leq 42$, we have
$$
p(G_i,G) \left(c_i - \langle Q, A_{G_i} \rangle - \textstyle\frac{1}{9}\right) \leq t(G)+i(G) - \textstyle\frac{1}{9} + \delta.
$$
Therefore, if $G$ is an $n$-vertex oriented graph such that $n \geq n_1(\delta)$ and $t(G) + i(G) \leq \textstyle\frac{1}{9} + \delta$, then for
every $39 \leq i \leq 42$, it holds that
$$
\eta_i p(G_i,G) = p(G_i,G) \left(c_i - \langle Q, A_{G_i} \rangle - \textstyle\frac{1}{9}\right) \leq t(G) + i(G) - \textstyle\frac{1}{9} + \delta \leq 2 \delta.
$$
Hence, if $G^{(0)}$ is the underlying undirected graph of $G$, then 
\begin{align*}
p(K_4,G^{(0)}) &= p(G_{39},G) + p(G_{40},G) + p(G_{41},G) + p(G_{42},G)\\
&\leq \left(\frac{1}{\eta_{39}} + \frac{1}{\eta_{40}} + \frac{1}{\eta_{41}} + \frac{1}{\eta_{42}}\right) 2 \delta.
\qedhere
\end{align*}
\end{proof}

The following is a reformulation of Theorem 5.1 from~\cite{DHMNS} (proved, incidentally, by using flag algebras), for the complement graph. 

\begin{theorem}[Theorem 5.1 in~\cite{DHMNS}] \label{thm:DHMNS}
Any $n$-vertex $K_4$-free (undirected) graph $G$ satisfies 
$$
i(G) - \frac{47}{4036n} \sum_{v \in V}\left(\frac{d_G(v)}{n − 1} − \frac{2}{3}\right)^2 \geq \frac{1}{9} - o_n(1).
$$
\end{theorem}

As in~\cite{DHMNS}, we will also make use of the following result.
\begin{theorem}[\cite{AES}]\label{thm:AES}
Any $n$-vertex $K_r$-free (undirected) graph, whose minimum degree is larger than $\frac{3r−7}{3r−4}n$, is $(r−1)$-partite. 
\end{theorem}

Combining Theorem~\ref{thm:DHMNS} and Theorem~\ref{thm:AES} yields the following corollary.
\begin{corollary} \label{cor:DHMNS+AES}
For every $0 < \varepsilon < \frac{1}{13}$ there is a positive integer $n_2(\varepsilon)$ such that the following statement is true. Let $G$ be an $n$-vertex $K_4$-free (undirected) graph such that $n \geq n_2(\varepsilon)$ and $i(G) \leq \frac{1}{9} + \varepsilon^6$, and let $B$ be the set of 
vertices of $G$ whose degree is less than $\left(\frac{2}{3} - 5\varepsilon^2 \right) n$.
Then 
\begin{enumerate}
\item $|B| \leq 5\varepsilon^2 n$; 
\item The graph obtained from $G$ by deleting the vertices of $B$ (and the edges incident to those vertices) is $3$-partite.
\end{enumerate}
\end{corollary}

\begin{proof} 
Starting with 1, it follows from Theorem~\ref{thm:DHMNS} that there is a positive integer $n_0(\varepsilon)$ such that for every $n \geq n_0(\varepsilon)$, if $G$ is an $n$-vertex $K_4$-free (undirected) graph, then 
\begin{equation} \label{eq::DegreesTwoThirds}
i(G) - \frac{47}{4036n} \sum_{v \in V}\left(\frac{d_G(v)}{n − 1} − \frac{2}{3}\right)^2 \geq \frac{1}{9} -\frac{1}{6} \varepsilon^6.
\end{equation}
Let 
$$
n_2(\varepsilon) = \max \left\{n_0(\varepsilon), \frac{4 - 27\varepsilon^2}{3 \varepsilon^2}\right\}.
$$ 
Let $G$ be an $n$-vertex $K_4$-free undirected graph with $n \geq n_2(\varepsilon)$ and $i(G) \leq \frac{1}{9} + \varepsilon^6$. Let $B$ be the set of vertices of $G$ whose degree is smaller than $\left(\frac{2}{3} - 5\varepsilon^2 \right) n$. 
First, note that for every $v \in B$, it holds that
\begin{equation} \label{eq::degreesInB}
d_G(v) < \left(\frac{2}{3} - 5\varepsilon^2 \right) n = \left(\frac{2}{3} - \frac{9}{2} \varepsilon^2 \right) (n-1) - \frac{1}{2} \varepsilon^2 \left(n - \frac{4 - 27\varepsilon^2}{3 \varepsilon^2}\right) \leq \left(\frac{2}{3} - \frac{9}{2} \varepsilon^2\right)(n-1).
\end{equation}
Therefore
$$
\frac{47}{4036n}|B|\left(\frac{9}{2}\varepsilon^2\right)^2\leq\frac{47}{4036n}\sum_{v\in V}\left(\frac{d_G(v)}{n − 1}−\frac{2}{3}\right)^2\leq i(G)-\frac{1}{9}+\frac{1}{6}\varepsilon^6\leq \frac{7}{6}\varepsilon^6,
$$
where the first inequality holds by~\eqref{eq::degreesInB} and the second inequality holds by~\eqref{eq::DegreesTwoThirds}. Hence
$$
|B| \leq \frac{4036}{47} \left(\frac{2}{9}\right)^2 \frac{7}{6} \varepsilon^2 n < 5\varepsilon^2 n.
$$

Next, we prove 2. Let $H$ be the graph obtained from $G$ by deleting all the vertices of $B$. Since $G$ is $K_4$-free, then clearly so is $H$. Moreover, for every vertex $v$ of $H$, it holds that
\begin{align*}
d_H(v) &\geq d_G(v) - |B| = d_G(v) - \textstyle\frac{3}{8}|B| - \textstyle\frac{5}{8}|B| \geq \left(\frac{2}{3} - 5\varepsilon^2 \right) n - \frac{3}{8} 5\varepsilon^2 n -\frac{5}{8}|B|\\
&= \textstyle\frac{55}{8} \left(\textstyle\frac{1}{165} - \varepsilon^2\right)n + \frac{5}{8}(n - |B|) > \textstyle\frac{5}{8}(n-|B|).
\end{align*}
Therefore, $H$ is $3$-partite by Theorem~\ref{thm:AES}.
\end{proof}

\begin{lemma}\label{lem:last_step}
Let $\varepsilon > 0$, let $G=(V,E)$ be an $n$-vertex oriented graph, let $B$ be the set of vertices of $G$ whose degree is less than $\left(\frac{2}{3} - 5\varepsilon^2\right)n$, and suppose that $|B| \leq 5\varepsilon^2 n$ and that $V \setminus B$ is the disjoint union of three independent sets $V_0, V_1, V_2$. Then
\begin{enumerate}
\item For every $0 \leq i \leq 2$, it holds that
$$
 \left(\textstyle\frac{1}{3} - 15 \varepsilon^2\right) n \leq |V_i| \leq \left(\textstyle\frac{1}{3} + 5 \varepsilon^2\right) n;
$$
\item Assume that, additionally, $\varepsilon < \frac{1}{360}$, $n \geq \frac{2}{\varepsilon^2}$ and $t(G) + i(G) \leq \frac{1}{9} + \varepsilon^6$. For all integers $0 \leq i \neq j \leq 2$, let $E(V_i, V_j)$ denote the set of edges of $G$ which are directed from a vertex of $V_i$ to a vertex of $V_j$. Then, either
$$
|E(V_0,V_1)| + |E(V_1,V_2)| + |E(V_2,V_0)| \leq 12 \varepsilon n^2
$$
or
$$
|E(V_1,V_0)| + |E(V_2,V_1)| + |E(V_0,V_2)| \leq 12 \varepsilon n^2.
$$
Moreover, in the former case, for every $0 \leq i \leq 2$, it holds that
$$
|\{u \in V_i : d^+_G(u,V_{i-1})< \left(\textstyle\frac{1}{3} -4\varepsilon - 15 \varepsilon^2\right) n\}|\leq \textstyle\frac{15}{2}\varepsilon n
$$
and
$$
|\{u \in V_i : d^-_G(u,V_{i+1}) < \left(\textstyle\frac{1}{3} - 4 \varepsilon - 15 \varepsilon^2\right) n\}| \leq \textstyle\frac{15}{2} \varepsilon n,
$$
(where the indices are reduced modulo 3) and in the latter case, for every $0 \leq i \leq 2$, it holds that
$$
|\{u \in V_i : d^+_G(u,V_{i+1}) < \left(\textstyle\frac{1}{3} - 4 \varepsilon - 15 \varepsilon^2\right) n\}| \leq \textstyle\frac{15}{2} \varepsilon n
$$
and
$$
|\{u \in V_i : d^-_G(u,V_{i-1}) < \left(\textstyle\frac{1}{3} - 4 \varepsilon - 15 \varepsilon^2\right) n\}| \leq \textstyle\frac{15}{2} \varepsilon n.
$$
\end{enumerate}
\end{lemma}

Note that the part of the statement of Lemma~\ref{lem:last_step} referring to the number of vertices whose degrees are atypically small (a similar statement will be made in Proposition~\ref{prop:stabil} as well), is not needed for the proof of Theorem~\ref{thm:stabil}; it will be used in the next section when we will prove Theorem~\ref{thm:exact}.

\begin{proof} [Proof of Lemma~\ref{lem:last_step}]
Starting with 1, fix some integer $0 \leq i \leq 2$ and some vertex $v_i \in V_i$. Since $V_i$ is an independent set, it follows that
\begin{equation}\label{eq:last_step_1}
|V_i| \leq n - d_G(v_i) \leq n - \left(\textstyle\frac{2}{3} - 5\varepsilon^2\right) n = \left(\textstyle\frac{1}{3} + 5\varepsilon^2\right) n.
\end{equation}
Therefore, for all integers $0 \leq i \neq j \leq 2$ and for every $v_j \in V_j$, it holds that
\begin{equation} \label{eq:last_step_2}
d_G(v_j, V_i) \geq d_G(v_j) - |V_{3-i-j}| - |B| \geq \left(\textstyle\frac{2}{3} - 5 \varepsilon^2\right) n - \left(\textstyle\frac{1}{3} + 5\varepsilon^2\right) n - 5\varepsilon^2 n = \left(\textstyle\frac{1}{3} - 15 \varepsilon^2\right) n. 
\end{equation}
In particular, for every $0 \leq i \leq 2$, it holds that
\begin{equation}\label{eq:last_step_3}
|V_i| \geq \left(\textstyle\frac{1}{3} - 15 \varepsilon^2\right)n.
\end{equation}

Next, we prove 2. It follows by~\eqref{eq:last_step_3} that 
\begin{align*}
i(G) &\geq \frac{1}{\binom{n}{3}} \left[\binom{|V_0|}{3} + \binom{|V_1|}{3} + \binom{|V_2|}{3}\right] \geq \left(\frac{|V_0| - 2}{n}\right)^3 + \left(\frac{|V_1| - 2}{n}\right)^3 + \left(\frac{|V_2| - 2}{n}\right)^3 \\
&\geq 3 \left(\textstyle\frac{1}{3} - 15 \varepsilon^2 - \textstyle\frac{2}{n}\right)^3 \geq 3 \left(\textstyle\frac{1}{3} - 16 \varepsilon^2\right)^3 = \textstyle\frac{1}{9} - 16 \varepsilon^2 +3 \cdot 16^2 \left(\textstyle\frac{1}{\varepsilon^2} - 16\right) \varepsilon^6 > \textstyle\frac{1}{9} - 16 \varepsilon^2 + \varepsilon^6.
\end{align*}
Hence 
\begin{equation}\label{eq:last_step_4}
t(G) = t(G) + i(G) - i(G) < \left(\textstyle\frac{1}{9} + \varepsilon^6\right) - \left(\textstyle\frac{1}{9} - 16\varepsilon^2 + \varepsilon^6\right) = 16 \varepsilon^2.
\end{equation}
Since $\frac{1}{3} - 15 \varepsilon^2 > 2 \cdot 4 \varepsilon$, it follows by~\eqref{eq:last_step_2} that, for all integers $0 \leq i \neq j \leq 2$, the set $V_i$ is the disjoint union of the sets  
\begin{align*}
V^+_{i,j} &:= \{v \in V_i : d_G^-(v, V_j) \leq 4 \varepsilon n\},\\
V^-_{i,j} &:= \{v \in V_i : d_G^+(v, V_j) \leq 4 \varepsilon n\},\\
\tilde{V}_{i,j} &:= \{v \in V_i : d_G^+(v, V_j) > 4 \varepsilon n,\, d_G^-(v, V_j) > 4 \varepsilon n\}.
\end{align*}
We would now like to show that $\tilde{V}_{i,j}$ is fairly small. Let $v_i$ be some vertex of $\tilde{V}_{i,j}$ (if $\tilde{V}_{i,j} = \emptyset$, then there is nothing to prove) and let $k = 3-i-j$. By~\eqref{eq:last_step_2}, either $d_G^+(v_i,V_k) \geq \left(\frac{1}{6} - \frac{15}{2} \varepsilon^2\right) n$ or $d_G^-(v_i,V_k) \geq \left(\frac{1}{6} - \frac{15}{2} \varepsilon^2\right) n$. Without loss of generality, assume that $d_G^+(v_i,V_k) \geq \left(\frac{1}{6} - \frac{15}{2}\varepsilon^2\right) n$. For every $v_k \in N_G^+(v_i, V_k)$ we have 
\begin{align} \label{eq::NotATransitiveTriangle}
|V_j \setminus N_G(v_k, V_j)| &\leq |(V \setminus V_k) \setminus N_G(v_k)| = n - |V_k| - d_G(v_k) \nonumber \\
&\leq n - \left(\textstyle\frac{1}{3} - 15 \varepsilon^2\right) n - \left(\textstyle\frac{2}{3} - 5 \varepsilon^2\right) n = 20 \varepsilon^2 n,
\end{align}
where the second inequality holds by~\eqref{eq:last_step_3}. Hence 
$$
|N_G^+(v_i, V_j) \cap N_G(v_k, V_j)| \geq d_G^+(v_i, V_j) - |V_j \setminus N_G(v_k, V_j)| > 4 \varepsilon n - 20 \varepsilon^2 n > \textstyle\frac{11}{3} \varepsilon n,
$$
where the penultimate inequality holds by the definition of $\tilde{V}_{i,j}$ and by~\eqref{eq::NotATransitiveTriangle}. It follows that $v_i$ participates in at least $\left(\frac{1}{6} - \frac{15}{2} \varepsilon^2\right) n \cdot \frac{11}{3} \varepsilon n$ transitive triangles, implying that
$$
t(G) \geq \frac{1}{\binom{n}{3}} |\tilde{V}_{i,j}| \left(\textstyle\frac{1}{6} - \textstyle\frac{15}{2} \varepsilon^2\right) n \cdot \textstyle\frac{11}{3} \varepsilon n > 6 \left(\textstyle\frac{1}{6} - \textstyle\frac{15}{2} \varepsilon^2\right) \textstyle\frac{11}{3} \cdot \textstyle\frac{\varepsilon}{n} |\tilde{V}_{i,j}| > \textstyle\frac{32\varepsilon}{9n} |\tilde{V}_{i,j}|.
$$
It thus follows by~\eqref{eq:last_step_4} that
\begin{equation} \label{eq:last_step_5}
|\tilde{V}_{i,j}| \leq \textstyle\frac{9n}{32\varepsilon} t(G) < \textstyle\frac{9}{2} \varepsilon n.
\end{equation}

Let
$$
{\mathcal E}^+ = \{(i,j) \in \{0,1,2\}^2 : i \neq j, \, |V^+_{i,j}| \leq 3 \varepsilon n\}
$$
and let 
$$
{\mathcal E}^- = \{(i,j) \in \{0,1,2\}^2 : i \neq j, \, |V^-_{i,j}| \leq 3 \varepsilon n\}.
$$
For every $(i,j) \in {\mathcal E}^+$, it holds that
\begin{align}
|E(V_i,V_j)| &= \sum_{v \in V_i} d^+(v, V_j) = \sum_{v \in V^-_{i,j}} d^+(v, V_j) + \sum_{v \in V^+_{i,j} \cup \tilde{V}_{i,j}} d^+(v, V_j)\nonumber\\
&\leq |V_i| \cdot 4 \varepsilon n + (|V^+_{i,j}| + |\tilde{V}_{i,j}|) \cdot |V_j|\nonumber\\
&\leq \left(\textstyle\frac{1}{3} + 5\varepsilon^2\right) n \cdot 4 \varepsilon n + \left(3\varepsilon n + \textstyle\frac{9}{2} \varepsilon n \right) \left(\textstyle\frac{1}{3} + 5 \varepsilon^2\right) n = \textstyle\frac{23}{2} \left(\textstyle\frac{1}{3} + 5 \varepsilon^2\right) \varepsilon n^2 < 4 \varepsilon n^2,\label{eq:E}
\end{align}
where the second inequality holds by~\eqref{eq:last_step_1} and~\eqref{eq:last_step_5}.
Furthermore, for every $(i,j) \in {\mathcal E}^+$, it holds that
\begin{equation}\label{eq:degrees_plus}
|\{u \in V_i : d^-_G(u, V_j) < \left(\textstyle\frac{1}{3} - 4 \varepsilon - 15 \varepsilon^2\right) n\}| \leq |V_i \setminus V^-_{i,j}| = |V^+_{i,j}| + |\tilde{V}_{i,j}| \leq \textstyle\frac{15}{2} \varepsilon n,
\end{equation}
where the first inequality holds by~\eqref{eq:last_step_2} and the last inequality holds by~\eqref{eq:last_step_5}.
Similarly, for every $(i,j) \in {\mathcal E}^-$, it holds that 
\begin{equation}\label{eq:degrees_minus}
|\{u \in V_i : d^+_G(u, V_j) < \left(\textstyle\frac{1}{3} - 4\varepsilon - 15 \varepsilon^2\right) n\}| \leq |V_i \setminus V^+_{i,j}| = |V^-_{i,j}| + |\tilde{V}_{i,j}|\leq \textstyle\frac{15}{2}\varepsilon n.
\end{equation}  
Fix some $i \neq j$ such that $(i,j) \in \{0,1,2\}^2 \setminus {\mathcal E}^+$, and let $k = 3-i-j$.
For every vertex $v_k \in V^+_{k,j}$, we have
\begin{align*}
|N_G(v_k, V_i) \cap V^+_{i,j}| &\geq d_G(v_k, V_i) + |V^+_{i,j}| - |V_i|
> \left(\textstyle\frac{1}{3} - 15 \varepsilon^2\right) n + 3 \varepsilon n - \left(\textstyle\frac{1}{3} + 5 \varepsilon^2\right) n \\ 
&= (3 - 20 \varepsilon) \varepsilon n > \textstyle\frac{53}{18} \varepsilon n,
\end{align*}
where the second inequality holds by~\eqref{eq:last_step_1} and~\eqref{eq:last_step_2}. Moreover, for every $v_i \in N_G(v_k, V_i) \cap V^+_{i,j}$, we have 
\begin{align*}
|N_G^+(v_k, V_j) \cap N_G^+(v_i,V_j)| &\geq d^+_G(v_k, V_j) + d^+_G(v_i, V_j) - |V_j|\\
&= d_G(v_k, V_j) - d^-_G(v_k, V_j) + d_G(v_i, V_j) - d^-_G(v_i, V_j) - |V_j|\\
&> 2\left(\left(\textstyle\frac{1}{3} - 15\varepsilon^2\right) n - 4 \varepsilon n\right) - \left(\textstyle\frac{1}{3} + 5\varepsilon^2\right) n = \left(\textstyle\frac{1}{3} - 8 \varepsilon - 35 \varepsilon^2\right) n > \textstyle\frac{16}{53} n,
\end{align*}
where the second inequality holds by~\eqref{eq:last_step_1} and~\eqref{eq:last_step_2}, and since $v_k \in V^+_{k,j}$ and $v_i \in V^+_{i,j}$. 
It thus follows by~\eqref{eq:last_step_4} that
\begin{align*}
\frac{16}{6} \varepsilon^2 n^3 &> \binom{n}{3} t(G) \geq \sum_{v_k \in V^+_{k,j}} \sum_{v_i \in N_G(v_k, V_i) \cap V^+_{i,j}} |N_G^+(v_k, V_j) \cap N_G^+(v_i,V_j)|\\
&\geq \sum_{v_k \in V^+_{k,j}} |N_G(v_k, V_i) \cap V^+_{i,j}| \cdot \frac{16}{53} n \geq  |V^+_{k,j}| \cdot \frac{53}{18} \varepsilon n \cdot \frac{16}{53} n = \frac{16}{6} \varepsilon^2 n^3 \frac{|V^+_{k,j}|}{3 \varepsilon n},
\end{align*}
implying that $(k,j) \in {\mathcal E}^+$. Moreover, for all integers $0 \leq i \neq j \leq 2$, it follows by~\eqref{eq:last_step_2} and~\eqref{eq:last_step_3} that
$$
|E(V_i,V_j)| + |E(V_j,V_i)| = \sum_{v \in V_i} d_G(v,V_j) \geq  \left(\left(\textstyle\frac{1}{3} - 15 \varepsilon^2\right) n \right)^2 > 2 \cdot 4 \varepsilon n^2.
$$
Therefore, by~\eqref{eq:E}, we cannot have both $(i,j) \in{\mathcal E}^+$ and $(j,i) \in {\mathcal E}^+$.

It follows that either ${\mathcal E}^+ = \{(0,1),(1,2),(2,0)\}$ or ${\mathcal E}^+ = \{(1,0),(2,1),(0,2)\}$. Similarly, either ${\mathcal E}^- = \{(0,1),(1,2),(2,0)\}$ or ${\mathcal E}^- = \{(1,0),(2,1),(0,2)\}$. Moreover, for all integers $0 \leq i \neq j \leq 2$, we cannot have both $(i,j) \in {\mathcal E}^+$ and $(i,j) \in {\mathcal E}^-$, as $|V^+_{i,j}| + |V^-_{i,j}| = |V_i| - |\tilde{V}_{i,j}| \geq (\frac{1}{3} - 15 \varepsilon^2) n - \frac{9}{2} \varepsilon n > 2 \cdot 3 \varepsilon n$. We conclude that either ${\mathcal E}^+ = \{(0,1),(1,2),(2,0)\}, \, {\mathcal E}^- = \{(1,0),(2,1),(0,2)\}$ or ${\mathcal E}^+ = \{(1,0),(2,1),(0,2)\}, \, {\mathcal E}^- = \{(0,1),(1,2),(2,0)\}$. Part 2 of the lemma now readily follows by~\eqref{eq:E}, \eqref{eq:degrees_plus} and~\eqref{eq:degrees_minus}.
\end{proof}

\begin{proposition} \label{prop:stabil}
For every $0<\varepsilon<\textstyle\frac{1}{360}$, there exist a positive integer $n_0(\varepsilon)$ and $\delta(\varepsilon) > 0$ such that the following holds for every $n \geq n_0(\varepsilon)$. If $G$ is an $n$-vertex oriented graph satisfying 
$$
t(G) + i(G) \leq \textstyle\frac{1}{9} + \delta(\varepsilon),
$$
then the set of vertices of $G$ is the disjoint union of four sets $B,V_0,V_1,V_2$, and there is an oriented graph $\tilde{G}$ obtained from $G$ by deleting at most $\frac{1}{12}\varepsilon^6 n^2$ edges such that the following hold.
\begin{enumerate}
\item The oriented graph $\tilde{G}$ has at least $\frac{1}{3}n^2-\frac{1}{15}\varepsilon n^2$ edges and 
$$|E_{\tilde{G}}(V_1,V_0)| + |E_{\tilde{G}}(V_2,V_1)| + |E_{\tilde{G}}(V_0,V_2)| \leq 12 \varepsilon n^2.$$
\item For every $0 \leq i \leq 2$, it holds that $V_i$ is independent in $\tilde{G}$, and
$$
 \left(\textstyle\frac{1}{3} - 15 \varepsilon^2\right) n \leq |V_i| \leq \left(\textstyle\frac{1}{3} + 5 \varepsilon^2\right) n;
$$
\item $|B| \leq 5\varepsilon^2 n$ and for every $0 \leq i \leq 2$, it holds that
$$
|\{u \in V_i : d^+_{\tilde{G}}(u,V_{i+1}) < \left(\textstyle\frac{1}{3} - 4 \varepsilon - 15 \varepsilon^2\right) n\}| \leq \textstyle\frac{15}{2} \varepsilon n
$$
and
$$
|\{u \in V_i : d^-_{\tilde{G}}(u,V_{i-1}) < \left(\textstyle\frac{1}{3} - 4 \varepsilon - 15 \varepsilon^2\right) n\}| \leq \textstyle\frac{15}{2} \varepsilon n,
$$
where the indices are reduced modulo 3.
\end{enumerate}
\end{proposition}

\begin{proof}
By the (undirected) graph removal lemma~\cite{ADLRY} (see also~\cite{CF} and the many references therein) there is a $\delta_0 > 0$ and a positive integer $n_4$ such that for every (undirected) graph $G$ on $n \geq n_4$ vertices for which $p(K_4, G) \leq \delta_0$, we can delete at most $\frac{1}{12} \varepsilon^6 n (n-1)$ edges of $G$ to obtain an undirected $K_4$-free graph. 
Let $C_1$ be as in Lemma~\ref{lem:first_step} and let
$$
\delta = \min\left\{\frac{1}{C_1}\delta_0, \frac{1}{2}\varepsilon^6\right\}.
$$
Let $n_1(\delta)$ be as in Lemma \ref{lem:first_step}, let $n_2(\varepsilon)$ be as in Corollary \ref{cor:DHMNS+AES}, and let
$$
n_0 = \max\left\{n_1(\delta), n_2(\varepsilon), \frac{2}{\varepsilon^2}, n_4\right\}.
$$
Let $G$ be an oriented graph on $n \geq n_0$ vertices such that $t(G) + i(G) \leq \frac{1}{9} + \delta$. Let $G^{(0)}$ be the underlying undirected graph of $G$. It follows by Lemma~\ref{lem:first_step} that $p(K_4, G^{(0)}) \leq C_1 \delta \leq \delta_0$, and therefore, we can delete at most $\frac{1}{12}\varepsilon^6 n^2$ edges of $G^{(0)}$ to obtain an undirected $K_4$-free graph $G^{(1)}$. 
Note that 
$$
i(G^{(1)}) \leq i(G^{(0)}) + 6 \frac{1}{12} \varepsilon^6 = i(G) + \frac{1}{2}
\varepsilon^6 \leq \frac{1}{9} + \delta + \frac{1}{2} \varepsilon^6 \leq \frac{1}{9} + \varepsilon^6.
$$
By Corollary~\ref{cor:DHMNS+AES}, the set $B$ of vertices of $G^{(1)}$ whose degree is less than $\left(\frac{2}{3} - 5\varepsilon^2\right) n$ is of size at most $5\varepsilon^2 n$ and $V \setminus B$ is the disjoint union of three independent sets $V_0,V_1,V_2$.
Let $\tilde{G}$ be the oriented graph obtained from $G^{(1)}$ by orienting each of its edges as it was oriented in $G$. Clearly, $\tilde{G}$ is obtained from $G$ by deleting at most $\frac{1}{12}\varepsilon^6 n^2$ edges, and the number of edges in $\tilde{G}$ is at least
$$\textstyle\frac{1}{2}|V\setminus B|\left(\textstyle\frac{2}{3}-5\varepsilon^2\right)n\geq \textstyle\frac{1}{2}\left(n-5\varepsilon^2 n\right)\left(\frac{2}{3}-5\varepsilon^2\right)n\geq \textstyle\frac{1}{3}n^2-\textstyle\frac{25}{6}\varepsilon^2 n^2>\textstyle\frac{1}{3}n^2-\textstyle\frac{1}{15}\varepsilon n^2.$$
Observe also that $t(\tilde{G}) \leq t(G)$ and that $i(\tilde{G}) = i(G^{(1)}) \leq i(G) + \frac{1}{2} \varepsilon^6$, implying that 
$$
t(\tilde{G}) + i(\tilde{G}) \leq t(G) + i(G) + \textstyle\frac{1}{2}
\varepsilon^6 \leq \textstyle\frac{1}{9} + \delta + \textstyle\frac{1}{2} \varepsilon^6 \leq \textstyle\frac{1}{9} + \varepsilon^6.
$$
Without loss of generality, it then follows by Lemma~\ref{lem:last_step} that 
$$
|E_{\tilde{G}}(V_1,V_0)| + |E_{\tilde{G}}(V_2,V_1)| + |E_{\tilde{G}}(V_0,V_2)| \leq 12 \varepsilon n^2,
$$ 
and that, for every $0 \leq i \leq 2$, we have
$$
 \left(\textstyle\frac{1}{3}-15\varepsilon^2\right)n\leq |V_i|\leq \left(\textstyle\frac{1}{3}+5\varepsilon^2\right)n,$$ 
$$
|\{u \in V_i : d^+_{\tilde{G}}(u,V_{i+1}) < \left(\textstyle\frac{1}{3} - 4 \varepsilon - 15 \varepsilon^2\right) n\}| \leq \textstyle\frac{15}{2} \varepsilon n
$$
and
\begin{equation*}
|\{u \in V_i : d^-_{\tilde{G}}(u,V_{i-1}) < \left(\textstyle\frac{1}{3} - 4 \varepsilon - 15 \varepsilon^2\right) n\}| \leq \textstyle\frac{15}{2} \varepsilon n.
\qedhere\end{equation*}
\end{proof}

\begin{proof}[Proof of Theorem \ref{thm:stabil}]
Let $\varepsilon_0$ be a real number satisfying 
$$
0 < \varepsilon_0 < \min\left\{\frac{\varepsilon}{25}, \frac{1}{360}\right\},
$$
and let $n_0 = n_0(\varepsilon_0)$ and $\delta = \delta(\varepsilon_0)$ be as in Proposition~\ref{prop:stabil}.
Let $G$ be an oriented graph on $n \geq n_0$ vertices such that $t(G) + i(G) \leq \frac{1}{9} + \delta$.
By Proposition~\ref{prop:stabil}, the set of vertices of $G$ is the disjoint union of four sets $B,V_0,V_1,V_2$ and there exists an oriented graph $G'$ which is obtained from $G$ by deleting at most $\frac{1}{12} \varepsilon_0^6 n^2$ edges, and it satisfies the following properties:
\begin{enumerate}
\item $|E(G')| \geq \frac{1}{3} n^2 - \frac{1}{15} \varepsilon_0 n^2$;
\item $|E_{G'}(V_1,V_0)| + |E_{G'}(V_2,V_1)| + |E_{G'}(V_0,V_2)| \leq 12 \varepsilon_0 n^2$;
\item for every $0 \leq i \leq 2$, it holds that $V_i$ is independent and
$$
 \left(\textstyle\frac{1}{3} - 15 \varepsilon_0^2 \right) n \leq |V_i| \leq \left(\textstyle\frac{1}{3} + 5 \varepsilon_0^2 \right) n.
$$
\end{enumerate}
For every $0 \leq i \leq 2$, let $\tilde{V}_i \subseteq V_i$ be an arbitrary set of size $\left\lceil\left(\frac{1}{3} - 15 \varepsilon_0^2 \right) n \right\rceil$. Let $G''$ be the oriented graph obtained from $G'$ by deleting all edges in $E_{G'}(V_1,V_0) \cup E_{G'}(V_2,V_1) \cup E_{G'}(V_0,V_2)$ and all edges with an endpoint in $V \setminus \left(\tilde{V}_0 \cup \tilde{V}_1 \cup \tilde{V}_2\right)$. Altogether, at most
$$
12\varepsilon_0 n^2+\left(n-3\left(\textstyle\frac{1}{3}-15\varepsilon_0^2\right)n\right)(n-1)<(12\varepsilon_0+45\varepsilon_0^2)n^2\leq\left(12+\textstyle\frac{1}{8}\right)\varepsilon_0 n^2
$$
edges were deleted. Hence the oriented graph $G''$ has at least $ \frac{1}{3} n^2 - \frac{1}{15} \varepsilon_0 n^2 - \left(12 + \frac{1}{8}\right) \varepsilon_0 n^2$ edges, all of which are direced from $\tilde{V}_0$ to $\tilde{V}_1$, from $\tilde{V}_1$ to $\tilde{V}_2$, or from $\tilde{V}_2$ to $\tilde{V}_0$. Finally, we turn $G''$ into ${\cal B}_n$ by distributing the vertices of $V \setminus \left(\tilde{V}_0 \cup \tilde{V}_1 \cup \tilde{V}_2\right)$ among the sets $\tilde{V}_0, \tilde{V}_1, \tilde{V}_2$ in a way which forms a balanced partition, and then adding all absent edges. Note that we need to add at most 
$$
\textstyle\frac{1}{3}n^2-\left(\textstyle\frac{1}{3}n^2-\textstyle\frac{1}{15}\varepsilon_0 n^2-\left(12+\textstyle\frac{1}{8}\right)\varepsilon_0 n^2\right)=\left(12+\textstyle\frac{1}{8}+\textstyle\frac{1}{15}\right)\varepsilon_0 n^2
$$
edges.
To summarize, we have turned $G$ into ${\cal B}_n$ by deleting or adding at most
$$
\textstyle\frac{1}{12} \varepsilon_0^6 n^2 + \left(12+\textstyle\frac{1}{8}\right)\varepsilon_0 n^2+\left(12+\textstyle\frac{1}{8}+\textstyle\frac{1}{15}\right)\varepsilon_0 n^2 < 25\varepsilon_0 n^2\leq\varepsilon n^2
$$
edges.
\end{proof}

\section{An exact result}\label{sec:exact}
In this section we use the stability result we proved in the previous section, to prove Theorem~\ref{thm:exact}.
Our argument builds  on the proof of Theorem 5.4 in~\cite{DHMNS}, but also requires several new ideas. 

First, let us introduce some additional notation. For an oriented graph $G = (V,E)$ and a set $S \subseteq V$, let $T_3(S, G)$ denote the number of transitive triangles in $G$ that contain all the vertices of $S$ and let $I_3(S, G)$ denote the number of independent triples in $G$ that contain all the vertices of $S$. We abbreviate $T_3(\emptyset, G)$ to $T_3(G)$ and $I_3(\emptyset, G)$ to $I_3(G)$. Moreover, for every $u \in V$ we abbreviate $T_3(\{u\}, G)$ to $T_3(u, G)$ and $I_3(\{u\}, G)$ to $I_3(u, G)$. 

\begin{proof}[Proof of Theorem \ref{thm:exact}]
Fix $\varepsilon > 0$ to be sufficiently small so as to handle all the calculations that are spread throughout the proof and let $n_0 = n_0(\varepsilon)$ be as in Proposition~\ref{prop:stabil}. Let $n$ and $G = (V,E)$ be as in the statement of the theorem. In order to prove Theorem~\ref{thm:exact}, we will prove that $G$ satisfies the following five properties:
\begin{description}
\item [(i)] $V_0 \cup V_1 \cup V_2$ is an equipartition of $V$, i.e., $\lfloor n/3 \rfloor \leq |V_0|, |V_1|, |V_2| \leq \lceil n/3 \rceil$;

\item [(ii)] $V_i$ is independent for every $0 \leq i \leq 2$;

\item [(iii)] There are no directed edges from $V_i$ to $V_{i-1}$ for any $0 \leq i \leq 2$ (where the indices are reduced modulo $3$).

\item [(iv)] For every $0 \leq i\neq j \leq 2$, every vertex in $V_i$ has at most one non-neighbour in $V_j$.  

\item [(v)] $E \cap \{\vec{xy}, \vec{yz}, \vec{zx}\} \neq \emptyset$ for every $x \in V_0$, $y \in V_1$ and $z \in V_2$.  
\end{description}

It follows from Observation~\ref{obs:blowup} that $t(G) + i(G) < 1/9$. Therefore, by Proposition \ref{prop:stabil}, the set of vertices of $G$ is the disjoint union of four sets $\tilde{B},\tilde{V}_0,\tilde{V}_1,\tilde{V}_2$ and there is an oriented graph $\tilde{G}$ obtained from $G$ by deleting some edges such that $|\tilde{B}|\leq 5\varepsilon^2 n$ and for every $0 \leq i \leq 2$, it holds that
$$
 \left(\textstyle\frac{1}{3} - 15 \varepsilon^2\right) n \leq |\tilde{V}_i| \leq \left(\textstyle\frac{1}{3} + 5 \varepsilon^2\right) n,
$$
$$
|\{u \in \tilde{V}_i : d^+_{\tilde{G}}(u, \tilde{V}_{i+1}) < \left(\textstyle\frac{1}{3} - 4\varepsilon - 15 \varepsilon^2\right) n\}| \leq \textstyle\frac{15}{2} \varepsilon n
$$
and
$$
|\{u \in \tilde{V}_i : d^-_{\tilde{G}}(u, \tilde{V}_{i-1}) < \left(\textstyle\frac{1}{3} - 4\varepsilon - 15 \varepsilon^2\right) n\}| \leq \textstyle\frac{15}{2} \varepsilon n.
$$
For every $0 \leq i \leq 2$, let
$$
A_i = \left\{u \in \tilde{V}_i : \min\{d^-_G(u, \tilde{V}_{i-1}), d^+_G(u, \tilde{V}_{i+1})\} \geq \left(\textstyle\frac{1}{3} - 4\varepsilon - 15 \varepsilon^2\right) n \right\}
$$
and let 
$$
B_0 = \tilde{B} \cup (\tilde{V}_0 \setminus A_0) \cup (\tilde{V}_1 \setminus A_1) \cup (\tilde{V}_2 \setminus A_2).
$$
Observe that $|\tilde{V}_i \setminus A_i| \leq 15 \varepsilon n$ holds for every $0 \leq i \leq 2$, and thus, for every $0 \leq i \leq 2$ and every $u \in A_i$, it holds that
\begin{align*}
\min\{d^-_G(u,A_{i-1}),d^+_G(u,A_{i+1})\}&\geq \min\{d^-_{\tilde{G}}(u,A_{i-1}),d^+_{\tilde{G}}(u,A_{i+1})\}\\
&\geq \min\{d^-_{\tilde{G}}(u,\tilde{V}_{i-1})-|\tilde{V}_{i-1}\setminus A_{i-1}|,d^+_{\tilde{G}}(u,\tilde{V}_{i+1})-|\tilde{V}_{i+1}\setminus A_{i+1}|\}\\
&\geq \left(\textstyle\frac{1}{3} -4\varepsilon - 15 \varepsilon^2\right) n-15\varepsilon n=\left(\textstyle\frac{1}{3} -19\varepsilon - 15 \varepsilon^2\right) n.
\end{align*}
Therefore, $A_0 \cup A_1 \cup A_2 \cup B_0$ is a partition of $V$ for which the following conditions hold. 
\begin{description}
\item [(1')] $(1/3 - 15 \varepsilon - 15 \varepsilon^2) n \leq |A_0|, |A_1|, |A_2| \leq (1/3 + 5 \varepsilon^2) n$;

\item [(2')] $\min \{d_G^+(u, A_{i+1}), d_G^-(u, A_{i-1})\} \geq (1/3 - 19 \varepsilon - 15 \varepsilon^2) n$ for every $0 \leq i \leq 2$ and every $u \in A_i$;

\item [(3')] $|B_0| \leq 45 \varepsilon n + 5 \varepsilon^2 n$;   

\end{description}

For as long as there exists a vertex $u \in B_0$ and an index $0 \leq i \leq 2$ such that 
$$
\min \{d_G^+(u, A_{i+1}), d_G^-(u, A_{i-1})\} \geq (1/3 - \sqrt[3]{\varepsilon}) n,
$$
remove $u$ from $B_0$ and add it to $A_i$ (observe that if such an $i$ exists, then it is unique, since for every $0 \leq j \leq 2$ it holds that $\min \{d_G^+(u, A_j), d_G^-(u, A_j)\} \leq |A_j|/2 < (1/3 - \sqrt[3]{\varepsilon}) n$, for a sufficiently small $\varepsilon > 0$).
Note that the $A_i$'s are updated in every step of this process and $\min \{d_G^+(u, A_{i+1}), d_G^-(u, A_{i-1})\}$ is considered with respect to those updated sets. Once this process is over, denote the resulting partition of $V$ by $V_0 \cup V_1 \cup V_2 \cup B$, where $B \subseteq B_0$ and $V_i \supseteq A_i$ for every $i \in \{0,1,2\}$. Observe that, for sufficiently small $\varepsilon$, this new partition satisfies the following properties: 
\begin{description}
\item [(1)] $(1/3 - 20 \varepsilon) n \leq |V_0|, |V_1|, |V_2| \leq (1/3 + 48 \varepsilon) n$;

\item [(2)] $\min \{d_G^+(u, V_{i+1}), d_G^-(u, V_{i-1})\} \geq (1/3 - \sqrt[3]{\varepsilon}) n$ for every $0 \leq i \leq 2$ and every $u \in V_i$;

\item [(3)] $|B| \leq 48 \varepsilon n$;   

\item [(4)] For every $u \in B$ and every $i \in \{0,1,2\}$ it holds that $d_G^+(u, V_{i+1}) < (1/3 - \sqrt[3]{\varepsilon}) n$ or $d_G^-(u, V_{i-1}) < (1/3 - \sqrt[3]{\varepsilon}) n$.
\end{description}

Using the minimality of $G$, we will prove that in fact this partition satisfies stronger conditions.

\begin{lemma} \label{lem::badEdges}
Let $G, V_0, V_1, V_2$, and $B$ be as above. Then
\begin{description}
\item [(a)] $V_i$ is independent in $G$ for every $0 \leq i \leq 2$; 

\item [(b)] $\vec{xy} \notin E$ for every $0 \leq i \leq 2$, $x \in V_i$, and $y \in V_{i-1}$;

\item [(c)] $B = \emptyset$.
\end{description}
\end{lemma}

\begin{proof}
Starting with (a), suppose for a contradiction that $V_i$ is not independent for some $0 \leq i \leq 2$. Fix an arbitrary directed edge $\vec{xy} \in E(G[V_i])$. Let $Z = \{z \in V : z \notin N_G(x) \cup N_G(y)\}$ denote the set of common non-neighbours of $x$ and $y$. Observe that
\begin{align} \label{eq::I3increase}
I_3(G \setminus \vec{xy}) &= I_3(G) + |Z| \leq I_3(G) + |V_i| + |B| + 2 (48 \varepsilon + \sqrt[3]{\varepsilon}) n \nonumber \\
&\leq I_3(G) + (1/3 + 48\varepsilon) n + 48\varepsilon n+2 (48 \varepsilon + \sqrt[3]{\varepsilon}) n < I_3(G) + (1/3 + 3 \sqrt[3]{\varepsilon}) n,
\end{align}
where the first inequality holds by properties (1) and (2), the second inequality holds by properties (1) and (3), and the last inequality holds for a sufficiently small $\varepsilon>0$.

On the other hand, let $W_1 = N_G^+(x, V_{i+1}) \cap N_G^+(y, V_{i+1})$ and let $W_2 = N_G^-(x, V_{i-1}) \cap N_G^-(y, V_{i-1})$. Then
\begin{align} \label{eq::T3decrease}
T_3(G \setminus \vec{xy}) &\leq T_3(G) - |W_1| - |W_2| \leq T_3(G) - 2[(1/3 - \sqrt[3]{\varepsilon}) n - (48 \varepsilon + \sqrt[3]{\varepsilon}) n] \nonumber \\
&<T_3(G) - (1/3 + 3 \sqrt[3]{\varepsilon})n,
\end{align}
where the second inequality holds by properties (1) and (2) and the last inequality holds for a sufficiently small $\varepsilon>0$. 

Combining~\eqref{eq::I3increase} and~\eqref{eq::T3decrease} we conclude that
$$
T_3(G \setminus \vec{xy}) + I_3(G \setminus \vec{xy}) < T_3(G) + I_3(G)
$$
contrary to the assumed minimality of $G$.

Next, we prove (b). Let $E' = \{\vec{xy} \in E : x \in V_0, y \in V_2\} \cup \{\vec{xy} \in E : x \in V_1, y \in V_0\} \cup \{\vec{xy} \in E : x \in V_2, y \in V_1\}$. Suppose for a contradiction that $E' \neq \emptyset$. Let $\vec{xy} \in E'$ be arbitrary and let $G'$ be the oriented graph obtained from $G$ by reversing the direction of $\vec{xy}$, that is, $G' = (G \setminus \vec{xy}) \cup \vec{yx}$. Clearly 
\begin{equation} \label{eq::I3nochange}
I_3(G') = I_3(G).
\end{equation} 
Assume without loss of generality that $x \in V_1$ and $y \in V_0$. It follows by properties (1) and (2) that, for a sufficiently small $\varepsilon > 0$,
\begin{equation} \label{eq::T3G}
T_3(\{x, y\}, G) \geq |N_G^+(x, V_2) \cap N_G(y, V_2)| \geq (1/3 - \sqrt[3]{\varepsilon}) n - (48 \varepsilon + \sqrt[3]{\varepsilon}) n > (1/3 - 3 \sqrt[3]{\varepsilon}) n.
\end{equation}
On the other hand, it follows by (a) and by properties (1), (2) and (3) that, for a sufficiently small $\varepsilon > 0$, 
\begin{equation} \label{eq::T3G'}
T_3(\{x, y\}, G') \leq |B| + d_G^-(x, V_2) + d_G^+(y, V_2) \leq 48 \varepsilon n + 2 (48 \varepsilon + \sqrt[3]{\varepsilon}) n < (1/3 - 3 \sqrt[3]{\varepsilon}) n.
\end{equation}
Combining~\eqref{eq::I3nochange}, \eqref{eq::T3G} and~\eqref{eq::T3G'} we conclude that
$$
T_3(G') + I_3(G') < T_3(G) + I_3(G)
$$
contrary to the assumed minimality of $G$.

Finally, we prove (c). We will first prove the following simple claim.
\begin{claim} \label{cl::smallT3plusI3}
$T_3(u, G) + I_3(u, G) \leq \binom{|V_i| + |B|}{2}$ for every vertex $u \in V$ and every $i \in \{0,1,2\}$.
\end{claim}

\begin{proof}
Suppose for a contradiction that there exists some vertex $u \in V$ and an $i \in \{0,1,2\}$ such that $T_3(u, G) + I_3(u, G) > \binom{|V_i| + |B|}{2}$. Let $G'$ be the oriented graph which is obtained from $G \setminus \{u\}$ by adding a new vertex $u'$ such that $N_{G'}^+(u') = V_{i+1}$ and $N_{G'}^-(u') = V_{i-1}$. Note that
$$
T_3(u', G') + I_3(u', G') \leq |\{\vec{xy} \in E : x \in V_{i-1}, y \in V_{i+1}\}| + \binom{|V_i| + |B|}{2}  = \binom{|V_i| + |B|}{2},
$$
where the inequality holds by (a) and the equality holds by (b). Hence
\begin{align*}
T_3(G') + I_3(G') &= T_3(G) + I_3(G) - (T_3(u, G) + I_3(u, G)) + (T_3(u', G') + I_3(u', G')) \\ 
&< T_3(G) + I_3(G)
\end{align*}
contrary to the assumed minimality of $G$.
\end{proof}

Now, suppose for a contradiction that $B \neq \emptyset$. In the remainder of the proof, we will use the notation $\binom{x}{2}$ for any real $x$, not necessarily a non-negative integer, in the sense of $x(x-1)/2$.
Note that for every $0 \leq i \leq 2$ and any real number $\alpha \geq 0$, it follows by Property (1) that
\begin{align*}
\binom{|V_i| + 48 \varepsilon n}{2} - \binom{|V_i| - \alpha n}{2} &= (\alpha+48\varepsilon)n\left(2|V_i|+(48\varepsilon-\alpha)n-1\right)/2\\
&< (\alpha+48\varepsilon)n\left(|V_i|+24\varepsilon n\right)\leq(\alpha+48\varepsilon)\left(1/3+72\varepsilon\right)n^2
\end{align*}
and hence, by Property (3),
\begin{equation}\label{eq:48}
\binom{|V_i| - \alpha n}{2} + (\alpha+48\varepsilon)\left(1/3+72\varepsilon\right)n^2 > \binom{|V_i| + 48 \varepsilon n}{2} \geq \binom{|V_i| + |B|}{2}.
\end{equation}
Let $u \in B$ be an arbitrary vertex. 
Let $0 \leq i \leq 2$ be such that
$$
d_G(u,V_i) = \min\{d_G(u,V_0), d_G(u,V_1), d_G(u,V_2)\}.
$$
We distinguish between the following three cases.

\begin{description}

\item [Case 1:] $d_G(u,V_i) < 100 \varepsilon n$.

It follows by Property (4) that
$\min \{d_G^+(u, V_{i+1}), d_G^-(u, V_{i-1})\} < (1/3 - \sqrt[3]{\varepsilon}) n$. Assume that $d_G^+(u, V_{i+1}) < (1/3 - \sqrt[3]{\varepsilon}) n$ (the complementary case $d_G^-(u, V_{i-1}) < (1/3 - \sqrt[3]{\varepsilon}) n$ can be handled  similarly).
We further divide this case into the following three sub-cases.
\begin{description}
\item [Case a:] $d_G(u, V_{i-1}) \leq n/6$.
Then, using (a) and properties (1) and (3), we obtain
\begin{align*}
T_3(u, G) + I_3(u,G) &\geq I_3(u,G) \geq \binom{|V_i| - d_G(u,V_{i})}{2} + \binom{|V_{i-1}| - d_G(u, V_{i-1})}{2} \\
&\geq \binom{|V_i| - 100 \varepsilon n}{2} + \binom{(1/6 - 20 \varepsilon) n}{2}. 
\end{align*}
Hence, for a sufficiently small $\varepsilon > 0$, 
$$
T_3(u, G) + I_3(u,G)>\binom{|V_i| - 100\varepsilon n}{2} + 148\varepsilon\left(\frac{1}{3}+72\varepsilon\right) n^2>\binom{|V_i| + |B|}{2},
$$
where the second inequality holds by~\eqref{eq:48} for $\alpha = 100 \varepsilon$;
this contradicts the assertion of Claim~\ref{cl::smallT3plusI3}.

\item [Case b:] $d_G(u, V_{i-1}) > n/6$ and $d_G^-(u, V_{i+1}) > 1000 \varepsilon n$.

Let $A_{i-1} = N_G(u, V_{i-1})$. It follows by properties (1) and (2) that, for a sufficiently small $\varepsilon>0$,
\begin{align*}
d_G^+(w, A_{i-1}) &\geq |A_{i-1}| - |V_{i-1}| + d_G^+(w, V_{i-1}) \\ 
&\geq n/6 - \left(1/3 + 48 \varepsilon \right) n + \left(1/3 - \sqrt[3]{\varepsilon}\right) n > \left(1/6 - 2 \sqrt[3]{\varepsilon}\right) n 
\end{align*}
holds for every $w \in N_G^-(u, V_{i+1})$. Hence
\begin{align*}
T_3(u, G) + I_3(u,G) &\geq \sum_{w \in N_G^-(u, V_{i+1})} d_G^+(w, A_{i-1}) + \binom{|V_i| - d_G(u, V_{i})}{2} \\
&\geq 1000 \varepsilon (1/6 - 2 \sqrt[3]{\varepsilon}) n^2 + \binom{|V_i| - 100 \varepsilon n}{2}. 
\end{align*}
Therefore, for a sufficiently small $\varepsilon>0$, 
$$T_3(u, G) + I_3(u,G)>\binom{|V_i| - 100\varepsilon n}{2} + 148\varepsilon\left(\frac{1}{3}+72\varepsilon\right) n^2>\binom{|V_i| + |B|}{2},
$$
where the second inequality holds by~\eqref{eq:48} for $\alpha = 100 \varepsilon$;
this contradicts the assertion of Claim~\ref{cl::smallT3plusI3}.

\item [Case c:] $d_G^-(u, V_{i+1}) \leq 1000 \varepsilon n$.
Then, using Property (1) and our assumption that $d_G^+(u, V_{i+1}) < (1/3 - \sqrt[3]{\varepsilon}) n$, we obtain, for a sufficiently small $\varepsilon>0$,
\begin{align*}
d_G(u, V_{i+1}) &= d_G^-(u, V_{i+1}) + d_G^+(u, V_{i+1}) < \left(1/3 - \sqrt[3]{\varepsilon} + 1000 \varepsilon \right) n \\
&\leq |V_{i+1}| - (\sqrt[3]{\varepsilon} - 20 \varepsilon - 1000 \varepsilon) n < |V_{i+1}| - \sqrt[3]{\varepsilon} n/2.
\end{align*}
Hence
\begin{align*}
T_3(u, G) + I_3(u,G) &\geq I_3(u,G) \geq \binom{|V_i| - d_G(u, V_{i})}{2} +  \binom{|V_{i+1}| - d_G(u, V_{i+1})}{2} \\
&\geq \binom{|V_i| - 100 \varepsilon n}{2} + \binom{\sqrt[3]{\varepsilon} n/2}{2}. 
\end{align*}
Therefore, for a sufficiently small $\varepsilon>0$, 
$$T_3(u, G) + I_3(u,G)>\binom{|V_i| - 100\varepsilon n}{2} + 148\varepsilon\left(\frac{1}{3}+72\varepsilon\right) n^2>\binom{|V_i| + |B|}{2},
$$
where the second inequality holds by~\eqref{eq:48} for $\alpha = 100 \varepsilon$;
this contradicts the assertion of Claim~\ref{cl::smallT3plusI3}.
\end{description}

\item [Case 2:] $100 \varepsilon n \leq d_G(u,V_i) < 10^{-4} n$.

We further divide this case into the following three sub-cases.
\begin{description}
\item [Case a:] $d_G(u, V_{i-1}) \leq |V_{i-1}| - n/100$ or $d_G(u, V_{i+1}) \leq |V_{i+1}| - n/100$.
Assume without loss of generality that $d_G(u, V_{i-1}) \leq |V_{i-1}| - n/100$ (the complementary case $d_G(u, V_{i+1}) \leq |V_{i+1}| - n/100$ is analogous). Then, using Property (1), we obtain
\begin{align*}
T_3(u, G) + I_3(u,G) &\geq I_3(u,G) \geq \binom{|V_i| - d_G(u,V_{i})}{2} + \binom{|V_{i-1}| - d_G(u, V_{i-1})}{2} \\
&\geq \binom{|V_i| - 10^{-4} n}{2} + \binom{n/100}{2}. 
\end{align*}
Hence, for a sufficiently small $\varepsilon>0$, 
$$T_3(u, G) + I_3(u,G)>\binom{|V_i| - 10^{-4} n}{2} + (10^{-4}+48\varepsilon)\left(\frac{1}{3}+72\varepsilon\right)n^2>\binom{|V_i| + |B|}{2},
$$
where the second inequality holds by~\eqref{eq:48} for $\alpha = 10^{-4}$;
this contradicts the assertion of Claim~\ref{cl::smallT3plusI3}.

\item [Case b:] $d_G^-(u, V_{i+1}) \geq n/100$ or $d_G^+(u, V_{i-1}) \geq n/100$.
Assume without loss of generality that $d_G^-(u, V_{i+1}) \geq n/100$ (the complementary case $d_G^+(u, V_{i-1}) \geq n/100$ is analogous). Let $A_{i-1} = N_G(u, V_{i-1})$. By Case (a) we may assume that $|A_{i-1}| > |V_{i-1}| - n/100$. It then follows by Property (2) that, for a sufficiently small $\varepsilon>0$,
\begin{align*}
d_G^+(w, A_{i-1}) &\geq |A_{i-1}| - |V_{i-1}| + d_G^+(w, V_{i-1}) \\ 
&> |V_{i-1}| - n/100 - |V_{i-1}| + (1/3 - \sqrt[3]{\varepsilon}) n 
> n/4
\end{align*}
holds for every $w \in N_G^-(u, V_{i+1})$. Hence
\begin{align*}
T_3(u, G) + I_3(u,G) &\geq \sum_{w \in N_G^-(u, V_{i+1})} d_G^+(w, A_{i-1}) + \binom{|V_i| - d_G(u, V_{i})}{2} \\
&\geq \frac{n}{100} \cdot \frac{n}{4} + \binom{|V_i| - 10^{-4} n}{2}. 
\end{align*}
Therefore, for a sufficiently small $\varepsilon>0$, 
$$T_3(u, G) + I_3(u,G)>\binom{|V_i| - 10^{-4} n}{2} + (10^{-4}+48\varepsilon)\left(\frac{1}{3}+72\varepsilon\right)n^2>\binom{|V_i| + |B|}{2},
$$
where the second inequality holds by~\eqref{eq:48} for $\alpha = 10^{-4}$;
this contradicts the assertion of Claim~\ref{cl::smallT3plusI3}.

\item [Case c:] $d_G^+(u, V_{i+1}) \geq |V_{i+1}| - n/50$ and $d_G^-(u, V_{i-1}) \geq |V_{i-1}| - n/50$.
It then follows by Property (2) that, for a sufficiently small $\varepsilon>0$,
\begin{align*}
d_G^+(w, N_G^+(u, V_{i+1})) &\geq |N_G^+(u, V_{i+1})| - |V_{i+1}| + d_G^+(w, V_{i+1}) \\ 
&\geq |V_{i+1}| - n/50 - |V_{i+1}| + (1/3 - \sqrt[3]{\varepsilon}) n 
> 2n/7
\end{align*}
and, similarly,
$$
d_G^-(w, N_G^-(u, V_{i-1})) > 2n/7
$$
hold for every $w \in N_G(u, V_i)$. It then follows by Property (1) that
\begin{align*}
T_3(u, G) + I_3(u,G) &\geq d_G(u,V_i) \cdot 2n/7 +d_G(u,V_i) \cdot 2n/7 + \binom{|V_i| - d_G(u,V_i)}{2} \\
&=\binom{|V_i|}{2}+\left(\frac{4n}{7}-|V_i|+\frac{d_G(u,V_i)+1}{2}\right)d_G(u,V_i)\\
&>\binom{|V_i|}{2}+\left(\frac{4n}{7}-\left(\frac{1}{3} + 48 \varepsilon\right)n \right)100 \varepsilon n.
\end{align*}
Hence, for a sufficiently small $\varepsilon>0$, 
$$T_3(u, G) + I_3(u,G)>\binom{|V_i|}{2} + 48\varepsilon\left(\frac{1}{3}+72\varepsilon\right)n^2>\binom{|V_i| + |B|}{2},
$$
where the second inequality holds by~\eqref{eq:48} for $\alpha = 0$;
this contradicts the assertion of Claim~\ref{cl::smallT3plusI3}.
\end{description}

\item [Case 3:] $d_G(u,V_i) \geq 10^{-4} n$. 

Denote
$$
d_* = \frac{1}{2} \left(d_G(u,V_{i-1}) + d_G(u,V_{i+1})\right).
$$
It follows by the minimality of $d_G(u,V_i)$ that
\begin{align*} 
d_G&(u,V_i) \left(d^+_G(u,V_0) d^-_G(u,V_1) + d^+_G(u,V_1) d^-_G(u,V_2) + d^+_G(u,V_2) d^-_G(u,V_0)\right) \\
&\leq d_G(u,V_2) d^+_G(u,V_0) d^-_G(u,V_1) + d_G(u,V_0) d^+_G(u,V_1) d^-_G(u,V_2) + d_G(u,V_1) d^+_G(u,V_2) d^-_G(u,V_0)\\
& = d_G(u,V_0) d_G(u,V_1) d_G(u,V_2) - d^+_G(u,V_0) d^+_G(u,V_1) d^+_G(u,V_2) - d^-_G(u,V_0) d^-_G(u,V_1) d^-_G(u,V_2)\\
& \leq  d_G(u,V_0) d_G(u,V_1) d_G(u,V_2) 
\end{align*}
and thus
\begin{align}
& d^+_G(u,V_0) d^-_G(u,V_1) + d^+_G(u,V_1) d^-_G(u,V_2) + d^+_G(u,V_2) d^-_G(u,V_0)\nonumber\\
&\leq d_G(u,V_{i-1}) d_G(u,V_{i+1}).\label{eq:dd}
\end{align}
For every $0\leq i\leq 2$, let
$$
{\mathcal D}_{i,i+1} := \{(x,y) \in (N_G(u,V_i)\times N_G(u,V_{i+1})) \setminus (N^+_G(u,V_i)\times N^-_G(u,V_{i+1})) : \vec{xy}\in E\}.
$$
For every $0 \leq i \leq 2$, it follows by properties (1) and (2) that
$$
|V_i|\,|V_{i+1}| - |E_G(V_i,V_{i+1})|=\sum_{v\in V_i}\left(|V_{i+1}|-d^+_G(v,V_{i+1})\right)\leq |V_i|(48\varepsilon+\sqrt[3]{\varepsilon})n
$$
and thus
$$
|{\mathcal D}_{i,i+1}|\geq d_G(u,V_i)d_G(u,V_{i+1})-d^+_G(u,V_i)d^-_G(u,V_{i+1})-|V_i|(48\varepsilon+\sqrt[3]{\varepsilon})n.
$$ 
Therefore, using \eqref{eq:dd}, we obtain
\begin{align} \label{eq::T3}
T_3(u, G) \geq &|{\mathcal D}_{0,1}|+|{\mathcal D}_{1,2}|+|{\mathcal D}_{2,0}|\nonumber\\
\geq &\left(d_G(u,V_0) d_G(u,V_1) + d_G(u,V_1) d_G(u,V_2) + d_G(u,V_2) d_G(u,V_0)\right) \nonumber\\
&- \left(d^+_G(u,V_0) d^-_G(u,V_1) + d^+_G(u,V_1) d^-_G(u,V_2) + d^+_G(u,V_2) d^-_G(u,V_0)\right) \nonumber\\
&-(|V_0|+|V_1|+|V_2|)(48 \varepsilon +\sqrt[3]{\varepsilon})  n \nonumber\\
\geq & 2 d_G(u,V_i) d_* - (48\varepsilon +\sqrt[3]{\varepsilon}) n^2
\end{align}
It follows by the convexity of the function $x \mapsto \binom{x}{2}$ that
\begin{align} \label{eq::I3}
I_3(u, G) &\geq \binom{|V_i| - d_G(u,V_i)}{2} + \binom{|V_{i-1}| - d_G(u,V_{i-1})}{2} + \binom{|V_{i+1}| - d_G(u,V_{i+1})}{2} \nonumber\\
&\geq \binom{|V_i| - d_G(u,V_i)}{2} + 2\binom{\frac{1}{2}(|V_{i-1}|+|V_{i+1}|) - d_*}{2}.
\end{align}
Combining~\eqref{eq::T3} and~\eqref{eq::I3}, and using Property (1), we conclude that
\begin{align*}
T_3&(u, G) + I_3(u,G)\\
\geq & \binom{|V_i| - d_G(u,V_i)}{2} + 2\binom{\frac{1}{2}(|V_{i-1}|+|V_{i+1}|) - d_*}{2} + 2 d_G(u,V_i) d_* - (48\varepsilon + \sqrt[3]{\varepsilon}) n^2\\
= &\binom{|V_i|}{2} + 2\binom{\frac{1}{2}(|V_{i-1}|+|V_{i+1}|) - d_*-d_G(u,V_i) }{2} \\
&+ d_G(u,V_i) \left(|V_{i-1}|+|V_{i+1}|-|V_i|-\textstyle\frac{1}{2}d_G(u,V_i) - \textstyle\frac{1}{2} \right)- (48\varepsilon + \sqrt[3]{\varepsilon})  n^2\\
\geq &\binom{|V_i|}{2} - \frac{1}{4} + 10^{-4} n \left(\frac{1}{6} n - 112 \varepsilon n - \frac{1}{2} \right)- (48\varepsilon + \sqrt[3]{\varepsilon})  n^2.
\end{align*}
Hence, for a sufficiently small $\varepsilon>0$, 
$$
T_3(u, G) + I_3(u,G)>\binom{|V_i|}{2} + 48\varepsilon\left(\frac{1}{3}+72\varepsilon\right)n^2>\binom{|V_i| + |B|}{2},
$$
where the second inequality holds by~\eqref{eq:48} for $\alpha = 0$;
this contradicts the assertion of Claim~\ref{cl::smallT3plusI3}.
\qedhere\end{description}     
\end{proof}

\begin{lemma} \label{lem::properties145}
$G$ satisfies properties (i), (iv) and (v).
\end{lemma}

\begin{proof}
Recall that $V_0 \cup V_1 \cup V_2 = V$ holds by Lemma~\ref{lem::badEdges}(c). Let 
$$
m_1 = \binom{|V_0|}{3} + \binom{|V_1|}{3} + \binom{|V_2|}{3} 
$$
and let 
$$
m_2 = \binom{\lfloor n/3 \rfloor}{3} + \binom{\lfloor (n + 1)/3 \rfloor}{3} + \binom{\lfloor (n + 2)/3 \rfloor}{3}.
$$
Observe that $m_1 \geq m_2$ and that $m_1 = m_2$ if and only if $G$ satisfies Property (i). Since $G$ satisfies Property (ii) by Lemma~\ref{lem::badEdges}(a), it follows that
$$
T_3(G) + I_3(G) \geq I_3(G) \geq m_1.
$$
Moreover
$$
T_3({\mathcal B}_n) + I_3({\mathcal B}_n) = I_3({\mathcal B}_n) = m_2.
$$
It thus follows by the assumed minimality of $G$ that it satisfies Property (i). 

Now suppose for a contradiction that $G$ does not satisfy Property (iv). That is, there exist two distinct indices $i, j \in \{0,1,2\}$ and three vertices $u \in V_i$ and $v, w \in V_j$ such that both $v$ and $w$ are non-neighbours of $u$. Since $G$ satisfies Property (ii), $u, v, w$ form an independent triple in $G$. Hence 
$$
T_3(G) + I_3(G) \geq I_3(G) \geq 1 + m_1 > m_2 = T_3({\mathcal B}_n) + I_3({\mathcal B}_n) 
$$
contrary to the assumed minimality of $G$.

Finally, suppose for a contradiction that $G$ does not satisfy Property (v). 
By Lemma \ref{lem::badEdges}(b), this implies that there exist vertices $x \in V_0$, $y \in V_1$ and $z \in V_2$ such that $x, y, z$ form an independent triple in $G$. Hence 
$$
T_3(G) + I_3(G) \geq I_3(G) \geq 1 + m_1 > m_2 = T_3({\mathcal B}_n) + I_3({\mathcal B}_n)
$$ 
contrary to the assumed minimality of $G$.  
\end{proof}

Since $B = \emptyset$ by Lemma~\ref{lem::badEdges}(c), $G$ satisfies properties (i), (iv) and (v) by Lemma~\ref{lem::properties145}, $G$ satisfies Property (ii) by Lemma~\ref{lem::badEdges}(a), and $G$ satisfies Property (iii) by Lemma~\ref{lem::badEdges}(b), the assertion of Theorem~\ref{thm:exact} follows.
\end{proof}

\section{Concluding remarks}\label{sec:concluding}

The problems that were considered in this paper can be extended in various directions. 
For example, it would be interesting to determine all possible pairs $(t(G), i(G))$. 
More formally, let $\bar{\mathcal{S}}$ be the set of all ordered pairs $(t,i) \in [0,1]^2$ for which there exists a sequence of oriented graphs $(G_n)_{n=1}^{\infty}$ such that $\lim_{n \to \infty} |V(G_n)| = \infty$, $\lim_{n \to \infty} t(G_n) = t$ and $\lim_{n \to \infty} i(G_n) = i$. We would like to determine the set $\bar{\mathcal{S}}$. 
Note first that the set $\mathcal{S}$ which corresponds to the undirected case (i.e., it is defined the same as $\bar{\mathcal{S}}$ except that $G_n$ is undirected for every $n$ and $t(G_n)$ stands for the number of triangles in $G_n$) was completely determined in~\cite{HLNPS}. Determining $\bar{\mathcal{S}}$ seems to be more challenging, but we are able to prove some partial results. 
First, since every undirected graph has an acyclic orientation, it immediately follows that $\bar{\mathcal{S}} \supseteq \mathcal{S}$. 
Trying to determine the lower envelope of $\bar{\mathcal{S}}$, for every $n$-vertex oriented graph $G$, it follows by Theorem~\ref{main} that
$$
t(G)+i(G) \geq \textstyle\frac{1}{9} - o_n(1),
$$
and by Proposition~\ref{prop:refined} that
$$
\textstyle\frac{2}{3}t(G)+i(G) \geq \textstyle\frac{1}{10} - o_n(1).
$$ 
Moreover, Observation~\ref{obs:blowup} and Theorem~\ref{main} imply that
$$
\min \{i(G) : G \textrm{ is an oriented graph on } n \textrm{ vertices for which } t(G) = 0\} = \textstyle\frac{1}{9} - o_n(1).
$$
Note that, using the removal lemma, one can also deduce the latter result from results in~\cite{DHMNS} and~\cite{pikhurko-vaughan}.

Finally, using a similar argument to the one used in the proof of Theorem~\ref{main} (but with oriented graphs on $5$ vertices instead of $4$) we believe that it is possible to show that
$$
\min \{t(G) : G \textrm{ is an oriented graph on } n \textrm{ vertices for which } i(G)=0\} \geq \textstyle\frac{3}{16} - o_n(1).
$$
Note that this bound is tight asymptotically as is demonstrated by the disjoint union of $K_{\lfloor n/2 \rfloor}$ and $K_{\lceil n/2 \rceil}$, where each edge is oriented independently at random with probability $1/2$ for each direction. A quick check with flagmatic (with oriented graphs on $5$ vertices) yields an approximate bound which is very close to $3/16$. Rounding it to a precise bound (and possibly also proving stability and uniqueness) is left for future work.

\bigskip

Similarly to the case of undirected graphs (as in, e.g., \cite{DHMNS} and~\cite{pikhurko-vaughan}), all the problems that were considered in this paper can be extended to larger independent sets and transitive tournaments. In particular, consider the following problems. Let $G = (V,E)$ be an oriented graph on $n$ vertices and let $k \geq 2$ be an integer. Let $T_k(G)$ denote the family of all $k$-sets $X \in \binom{V}{k}$ for which $G[X]$ is a transitive tournament and let $t_k(G) = |T_k(G)|/\binom{n}{k}$. Similarly, let $I_k(G)$ denote the family of all independent $k$-sets $X \in \binom{V}{k}$ and let $i_k(G) = |I_k(G)|/\binom{n}{k}$. Let $f(k, \ell, n) = \min \{t_k(G) : G \textrm{ is an oriented graph on } n \textrm{ vertices with } i_{\ell}(G) = 0\}$ and let $g(k, \ell, n) = \min \{i_k(G) : G \textrm{ is an oriented graph on } n \textrm{ vertices with } t_{\ell}(G) = 0\}$. It is not hard to see that the limits $f(k, \ell) := \lim_{n \to \infty} f(k, \ell, n)$ and $g(k, \ell) := \lim_{n \to \infty} g(k, \ell, n)$ exist for all $k$ and $\ell$. The last two results listed in the previous paragraph can then be restated as $g(3, 3) = 1/9$ and $f(3, 3) \approx 3/16$. Moreover, it is evident that $g(k, 2) = 1$ for every $k$ and, using Tur\'an's Theorem and the removal lemma, it is not hard to prove that $g(2, \ell) = 1/d$, where $d = d(\ell)$ is the so-called Ramsey number of the transitive tournament on $\ell$ vertices, i.e., it is the largest integer for which there exists an orientation $D$ of $K_d$ such that $|T_{\ell}(D)| = 0$. Note that the bounds $2^{\ell/2} \leq d(\ell) \leq 2^{\ell}$ are known, but determining $d(\ell)$ is, in general, an open problem (see, e.g., \cite{EM, Stearns}). Similarly, it is an easy consequence of Tur\'an's Theorem and the removal lemma that $f(2, \ell) = 1/(\ell - 1)$ for every $\ell \geq 2$. Moreover, it is not hard to prove by induction on $k$ that $f(k, 2) = k! \cdot 2^{- \binom{k}{2}}$ for every $k \geq 2$. It would be interesting to study $f(k, \ell)$ and $g(k, \ell)$ for additional  values of $k$ and $\ell$. It would also be interesting to study 
$$
\lim_{n \to \infty} \min \{t_k(G) + i_k(G) : G \textrm{ is an oriented graph on } n \textrm{ vertices}\}
$$ 
for every $k \geq 4$.

\subsection*{Acknowledgements}
We are grateful to the anonymous referee for many helpful comments.

\end{document}